\documentclass[reqno]{amsart}

\usepackage[utf8]{inputenc}
\usepackage[T1]{fontenc}
\usepackage[british]{babel,isodate}
\cleanlookdateon
\usepackage{lmodern}
\usepackage{enumerate}
\usepackage{tikz}
\usepackage{tikz-cd}
\usetikzlibrary{arrows}
\usepackage{cancel}
\usetikzlibrary{babel}
\usepackage{lscape}
\usepackage{url}

\usepackage{hyperref}
\hypersetup{linktocpage,colorlinks=true,citecolor=purple,linkcolor=blue,urlcolor=blue,filecolor=black}

\usepackage[margin=1in]{geometry}
\usepackage{float}
\usepackage{graphicx, import}
\usepackage{mathtools}
\usepackage{dsfont}
\usepackage{pinlabel}
\usepackage{yhmath}
\usepackage{enumitem}

\usepackage{comment}

\usepackage{xargs}                      
\usepackage{xcolor} 
\definecolor{OliveGreen}{rgb}{0,0.6,0}
 
\usepackage[colorinlistoftodos, prependcaption,textsize=tiny,textwidth=20mm]{todonotes}
\newcommandx{\unsure}[2][1=]{\todo[linecolor=red,backgroundcolor=red!25,bordercolor=red,#1]{#2}}
\newcommandx{\change}[2][1=]{\todo[linecolor=blue,backgroundcolor=blue!25,bordercolor=blue,#1]{#2}}
\newcommandx{\info}[2][1=]{\todo[linecolor=OliveGreen,backgroundcolor=OliveGreen!25,bordercolor=OliveGreen,#1]{#2}}

\newcommand{\centre}[1]{\begin{array}{c} #1 \end{array}}

\usepackage{stackengine}

\usepackage{amsmath,amsthm,amssymb, amsfonts, amscd}
\usepackage{rotating,subcaption}
\usepackage{url}
\usepackage{graphicx,bm}
\usepackage{mathtools}
\usepackage[mathscr]{euscript}
\allowdisplaybreaks
\usepackage{lipsum}
\usepackage{bm}
\usepackage[capitalize]{cleveref} 

\usepackage{xpatch}
\makeatletter
\AtBeginDocument{\xpatchcmd{\@thm}{\thm@headpunct{.}}{\thm@headpunct{}}{}{}}
\makeatother

\newtheorem{theorem}{Theorem}[section]
\newtheorem{proposition}[theorem]{Proposition}
\newtheorem{lemma}[theorem]{Lemma}
\newtheorem{corollary}[theorem]{Corollary}

\theoremstyle{definition}

\newtheorem{example}[theorem]{Example}

\theoremstyle{remark}

\newtheorem{remark}[theorem]{Remark}

\numberwithin{equation}{section}

\newcommand{\KPP}[1]{%
  \begin{tikzpicture}[scale=0.8, baseline=-\dimexpr\fontdimen22\textfont2\relax]
  #1
  \end{tikzpicture}%
}

\newcommand{\PC}{%
  \KPP{
    \draw[color=black] (0.2,0.2) -- (-0.2,-0.2);
    \draw[color=black] (0.2,-0.2) -- (0.05,-0.05);
    \draw[color=black] (-0.05,0.05) -- (-0.2, 0.2);
     \draw[color=black] (-0.05,0.05) -- (-0.2, 0.2);
      \draw[color=black] (0.1,0.2) -- (0.21, 0.2);
      \draw[color=black] (0.2,0.1) -- (0.2, 0.21);
   \draw[color=black] (-0.21,0.2) -- (-0.1, 0.2);
    \draw[color=black] (-0.2,0.21) -- (-0.20, 0.1);
  }%
}

\newcommand{\NC}{%
  \KPP{
    \draw[color=black] (-0.2,0.2) -- (0.2,-0.2);
    \draw[color=black] (-0.2,-0.2) -- (-0.05,-0.05);
    \draw[color=black] (0.05,0.05) -- (0.2, 0.2);
          \draw[color=black] (0.1,0.2) -- (0.21, 0.2);
      \draw[color=black] (0.2,0.1) -- (0.2, 0.21);
   \draw[color=black] (-0.21,0.2) -- (-0.1, 0.2);
    \draw[color=black] (-0.2,0.21) -- (-0.20, 0.1);
  }%
}

\newcommand{\R}{\mathbb{R}}

\newcommand{\Tup}{\mathcal{T}^{\mathrm{up}}}
\newcommand{\id}{\mathrm{Id}}

\renewcommand{\to}{\longrightarrow}

\usepackage{accents}

\DeclareFontFamily{U}{mathx}{}
\DeclareFontShape{U}{mathx}{m}{n}{ <-> mathx10 }{}
\DeclareSymbolFont{mathx}{U}{mathx}{m}{n}
\DeclareFontSubstitution{U}{mathx}{m}{n}
\DeclareMathAccent{\widecheck}{0}{mathx}{"71}

\usetikzlibrary{decorations.markings}
\tikzset{double line with arrow/.style args={#1,#2}{decorate,decoration={markings,%
mark=at position 0 with {\coordinate (ta-base-1) at (0,1pt);
\coordinate (ta-base-2) at (0,-1pt);},
mark=at position 1 with {\draw[#1] (ta-base-1) -- (0,1pt);
\draw[#2] (ta-base-2) -- (0,-1pt);
}}}}

\begin{document}


\title{Minimal generating sets of rotational Reidemeister moves}

\date{\today}

\author{Jorge Becerra}
\address{Université Bourgogne Europe, CNRS, IMB UMR 5584, 21000 Dijon, France}
\email{\href{mailto:Jorge.Becerra-Garrido@ube.fr}{Jorge.Becerra-Garrido@ube.fr}}
\urladdr{ \href{https://sites.google.com/view/becerra/}{https://sites.google.com/view/becerra/}} 

\author{Kevin van Helden}
\address{University of Groningen, Nijenborgh 9, 9747 AG, Groningen, The Netherlands}
\email{\href{mailto:k.s.van.helden@rug.nl}{k.s.van.helden@rug.nl}}
\urladdr{ \href{https://sites.google.com/view/ksvanhelden}{https://sites.google.com/view/ksvanhelden}} 





\begin{abstract}
Rotational tangle diagrams have been proven to be extremely important in the study of quantum invariants, as they provide a natural passage between topology and quantum algebra. In this paper, we give a detailed description of several generating sets of rotational Reidemeister moves for rotational  diagrams of both unframed and framed tangles. In particular, we prove that the minimal number of moves needed to generate all oriented unframed (resp. framed) rotational Reidemeister moves is 8 (resp. 9). The latter implies that a minimal generating set of Reidemeister moves for oriented, framed links contains 5 moves.
\end{abstract}



\maketitle

\setcounter{tocdepth}{1}
\tableofcontents


\section{Introduction}

The Reidemeister theorem is one of the most foundational and classical results in knot theory \cite{reidemeister}. It states that two diagrams $D,D' \subset \R^2$ of unoriented, unframed links in $\R^3$   are isotopic if and only if they are related by a sequence of planar isotopies and the so-called \textit{Reidemeister moves} $\Omega 1$, $\Omega 2$ and $\Omega 3$ depicted below:
\begin{equation}\label{eq:unfr_unor_Rmoves}
\centre{
\labellist \small \hair 2pt
\pinlabel{$\leftrightsquigarrow$}  at 439 190
\pinlabel{{\scriptsize $\Omega 1$}}  at 445 260
\pinlabel{$\leftrightsquigarrow$}  at 1620 190
\pinlabel{{\scriptsize $\Omega 2$}}  at 1625 260
\pinlabel{$\leftrightsquigarrow$}  at 2865 190
\pinlabel{{\scriptsize $\Omega 3$}}  at 2869 260
\endlabellist
\centering
\includegraphics[width=0.9\textwidth]{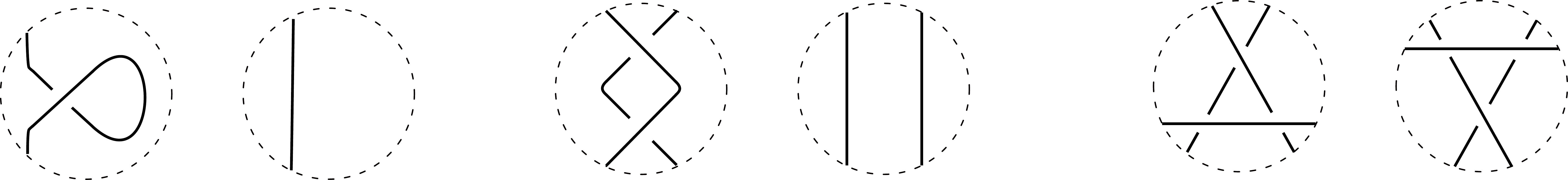}}.
\end{equation}
Each of the depicted equivalences should be understood as identifying two link diagrams that are identical except in some open neighbourhoods of the diagrams, represented by the dashed circles, where they look as shown. It is immediate to see that the moves $\Omega 1'$ and $\Omega 3'$ below are consequences of the other three:
\begin{equation}\label{eq:unfr_unor_Rmoves_2}
\centre{
\labellist \small \hair 2pt
\pinlabel{$\leftrightsquigarrow$}  at 439 190
\pinlabel{{\scriptsize $\Omega 1'$}}  at 445 260
\pinlabel{$\leftrightsquigarrow$}  at 1860 190
\pinlabel{{\scriptsize $\Omega 3'$}}  at 1865 260
\endlabellist
\centering
\includegraphics[width=0.6\textwidth]{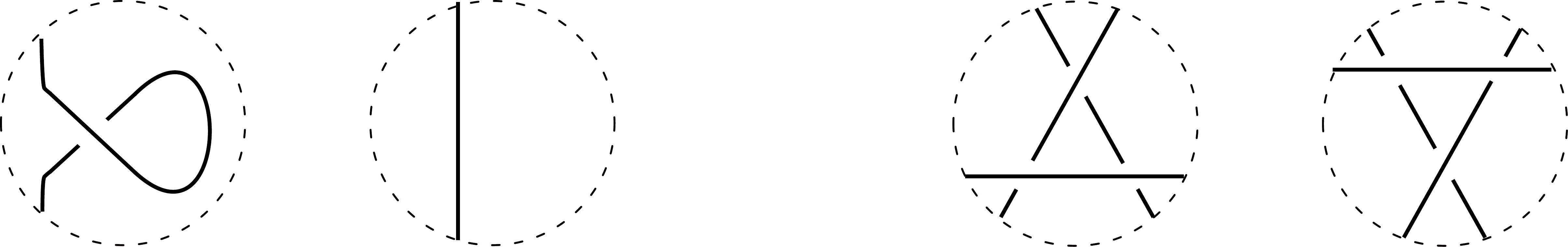}}.
\end{equation}

In the case of unframed, oriented links, one would have to consider --a priori--  all possible diagrams resulting from assigning arbitrary orientations to the strands of the diagrams depicted in   \eqref{eq:unfr_unor_Rmoves} and \eqref{eq:unfr_unor_Rmoves_2}. This yields four different oriented versions of the $\Omega 1$ move, four of the $\Omega 2$ moves and eight of the $\Omega 3$ move. Of course,  it is expected that some of these moves are consequence of a smaller subset of  moves. For instance, the oriented versions of  $\Omega 1'$ and $\Omega 3'$ are consequence of the oriented counterparts of the other three moves.

In a beautiful paper \cite{polyak10}, Polyak studied minimal generating sets of Reidemeister moves for unframed, oriented links. He found that one needs at least four moves:  two versions of the $\Omega 1$, and one version of $\Omega 2$ and $\Omega 3$ each,
\begin{equation}\label{eq:polyak_gen_1}
\centre{
\labellist \small \hair 2pt
\pinlabel{$\leftrightsquigarrow$}  at 439 190
\pinlabel{{\scriptsize $\Omega 1a$}}  at 445 260
\pinlabel{$\leftrightsquigarrow$}  at 1860 190
\pinlabel{{\scriptsize $\Omega 1b$}}  at 1865 260
\endlabellist
\centering
\includegraphics[width=0.6\textwidth]{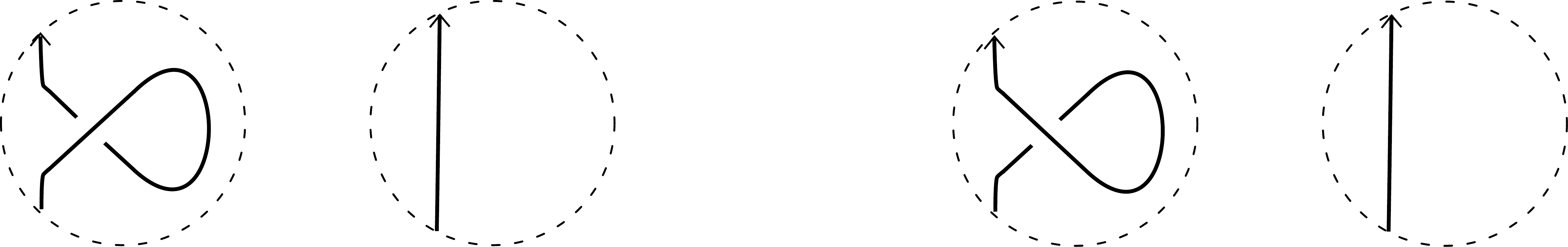}}
\end{equation}
\begin{equation}\label{eq:polyak_gen_2}
\centre{
\labellist \small \hair 2pt
\pinlabel{$\leftrightsquigarrow$}  at 439 190
\pinlabel{{\scriptsize $\Omega 2a$}}  at 445 260
\pinlabel{$\leftrightsquigarrow$}  at 1860 190
\pinlabel{{\scriptsize $\Omega 3a$}}  at 1865 260
\endlabellist
\centering
\includegraphics[width=0.6\textwidth]{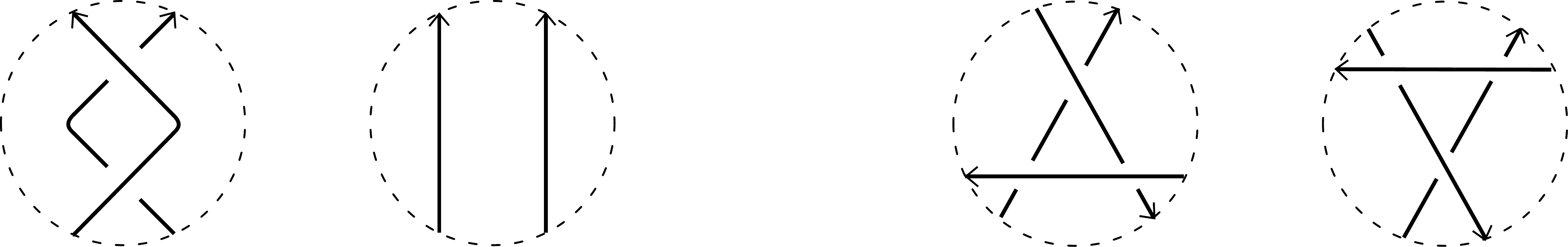}}.
\end{equation}
The $\Omega 3$ move in this generating set is the so-called \textit{cyclic} $\Omega 3$ move. Recently, Caprau and Scott explicitly described all minimal (i.e. 4-element) sets of generating oriented Reidemeister moves \cite{CS}. There are 12 such sets, all of them with two $\Omega 1$ moves, one $\Omega 2$ move and  with either the same type 3 move $\Omega 3a$  as in Polyak's set or the version $\Omega 3h$ resulting from mirror imaging $\Omega 3a$.   It is surprising that this move has different properties from the more common type 3 move $\Omega 3b$ with three positive crossings  coming from braid theory. If one wants to include the latter in a set of generators, it turns out that one needs at least five different moves to obtain a generating set, see \cite[Theorem 1.2]{polyak10}.

It is puzzling that, despite the Reidemeister moves lying at the foundations of knot theory, providing a (minimal) generating set for their oriented version is non-trivial and has led a number of authors to use sets that are, in fact, not generating, see \cite{polyak10} for an account. 

In this paper, we study the analogous problem of finding a minimal generating set of Reidemeister moves for rotational diagrams of knots, links and tangles. These diagrams have proven to be extremely convenient in the study of quantum invariants \cite{barnatanveenpolytime,barnatanveengaussians, barnatanveenAPAI,becerra_thesis,becerra_refined}. An oriented link or tangle diagram is said to be rotational if all endpoints and crossings point up and and maxima and minima appear in ``full rotation'' pairs, as depicted  in \eqref{eq:full_rotations}.

We regard these diagrams up to Morse planar isotopy, that is, isotopy that preserves maxima and minima. A large class of objects admits rotational diagrams: knots, links, upwards tangles and more generally  tangles in the bigon $D^2 - \{ (0, \pm 1)  \}$.

Focusing on these diagrams, it is therefore important to seek a rotational analogue of the  Reidemeister theorem, that is, a theorem that contains a list of moves in rotational form such that two rotational diagrams represent the same object if and only if they are related by a sequence of Morse planar isotopies and those moves.

Contrary to what one could think, adapting the arguments of \cite{polyak10} to obtain the analogue theorem is not straightforward. Firstly, restricting planar isotopy to Morse isotopy implies that we will surely need additional rotational Reidemeister moves that encode planar isotopies that create or destroy maxima and minima. Secondly,  the unoriented Reidemeister moves from \eqref{eq:unfr_unor_Rmoves} and \eqref{eq:unfr_unor_Rmoves_2} split, in the oriented, rotational case, into even more cases than in the oriented case. We obtain four different versions of the $\Omega 1$ move,
\begin{equation*}
\centre{
\labellist \small \hair 2pt
\pinlabel{$\leftrightsquigarrow$}  at 439 190
\pinlabel{{\scriptsize $\Omega 1a$}}  at 445 260
\pinlabel{$\leftrightsquigarrow$}  at 1860 190
\pinlabel{{\scriptsize $\Omega 1b$}}  at 1865 260
\endlabellist
\centering
\includegraphics[width=0.6\textwidth]{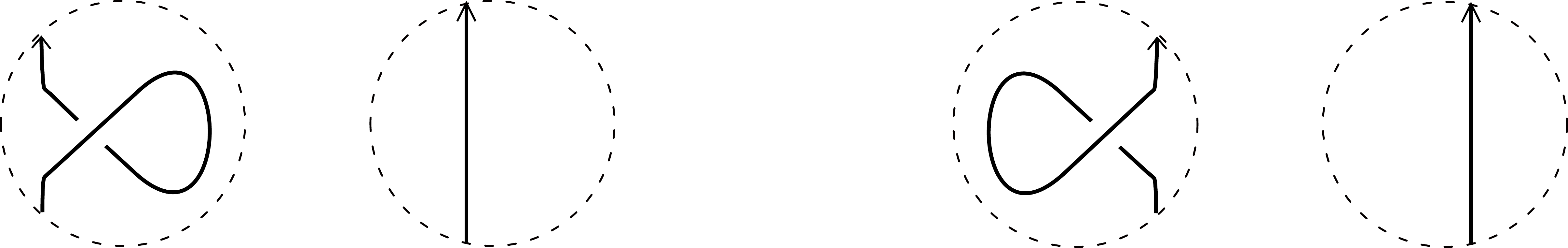}}
\end{equation*}
\begin{equation*}
\centre{
\labellist \small \hair 2pt
\pinlabel{$\leftrightsquigarrow$}  at 439 190
\pinlabel{{\scriptsize $\Omega 1c$}}  at 445 260
\pinlabel{$\leftrightsquigarrow$}  at 1860 190
\pinlabel{{\scriptsize $\Omega 1d$}}  at 1865 260
\endlabellist
\centering
\includegraphics[width=0.6\textwidth]{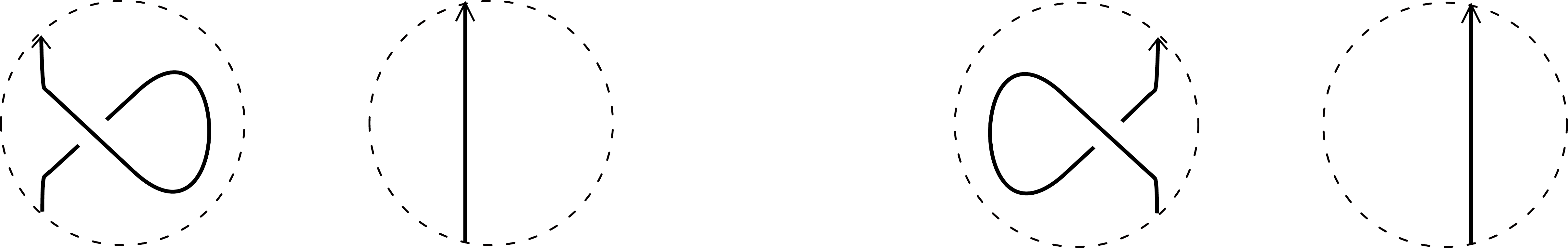}},
\end{equation*}
six different versions of the $\Omega 2$ move,
\begin{equation*}
\centre{
\labellist \small \hair 2pt
\pinlabel{$\leftrightsquigarrow$}  at 439 190
\pinlabel{{\scriptsize $\Omega 2a$}}  at 445 260
\pinlabel{$\leftrightsquigarrow$}  at 1625 190
\pinlabel{{\scriptsize $\Omega 2b$}}  at 1630 260
\pinlabel{$\leftrightsquigarrow$}  at 2865 190
\pinlabel{{\scriptsize $\Omega 2c1$}}  at 2869 260
\endlabellist
\centering
\includegraphics[width=0.9\textwidth]{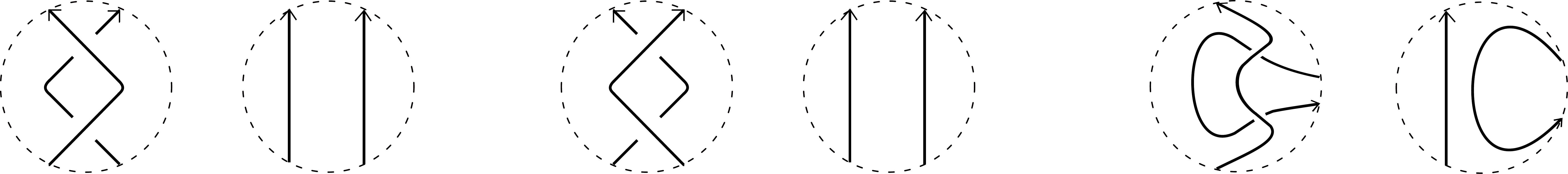}}
\end{equation*}
\begin{equation*}
\centre{
\labellist \small \hair 2pt
\pinlabel{$\leftrightsquigarrow$}  at 439 190
\pinlabel{{\scriptsize $\Omega 2c2$}}  at 445 260
\pinlabel{$\leftrightsquigarrow$}  at 1625 190
\pinlabel{{\scriptsize $\Omega 2d1$}}  at 1630 260
\pinlabel{$\leftrightsquigarrow$}  at 2865 190
\pinlabel{{\scriptsize $\Omega 2d2$}}  at 2869 260
\endlabellist
\centering
\includegraphics[width=0.9\textwidth]{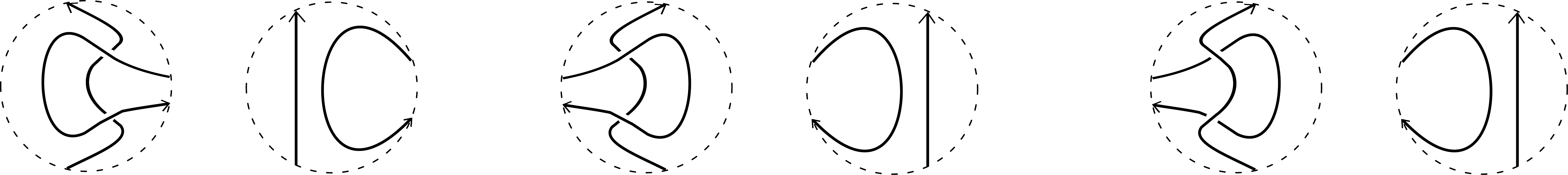}}
\end{equation*}
and eighteen different versions of the  $\Omega 3$ move,
\begin{equation*}
\centre{
\labellist \small \hair 2pt
\pinlabel{$\leftrightsquigarrow$}  at 439 190
\pinlabel{{\scriptsize $\Omega 3a1$}}  at 445 260
\pinlabel{$\leftrightsquigarrow$}  at 1625 190
\pinlabel{{\scriptsize $\Omega 3a2$}}  at 1630 260
\pinlabel{$\leftrightsquigarrow$}  at 2865 190
\pinlabel{{\scriptsize $\Omega 3a3$}}  at 2869 260
\endlabellist
\centering
\includegraphics[width=0.9\textwidth]{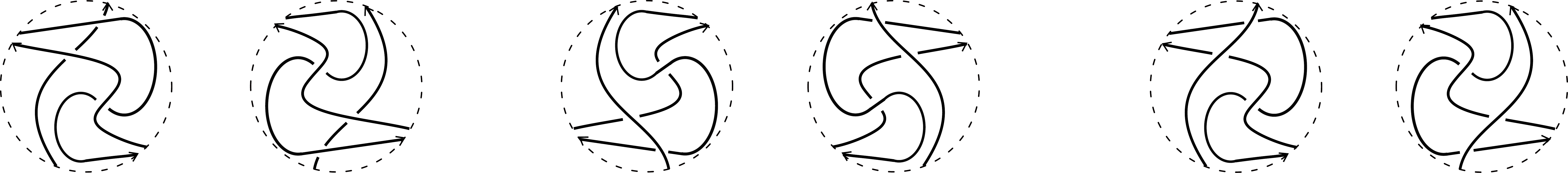}}
\end{equation*}
\begin{equation*}
\centre{
\labellist \small \hair 2pt
\pinlabel{$\leftrightsquigarrow$}  at 439 190
\pinlabel{{\scriptsize $\Omega 3a4$}}  at 445 260
\pinlabel{$\leftrightsquigarrow$}  at 1625 190
\pinlabel{{\scriptsize $\Omega 3a5$}}  at 1630 260
\pinlabel{$\leftrightsquigarrow$}  at 2865 190
\pinlabel{{\scriptsize $\Omega 3a6$}}  at 2869 260
\endlabellist
\centering
\includegraphics[width=0.9\textwidth]{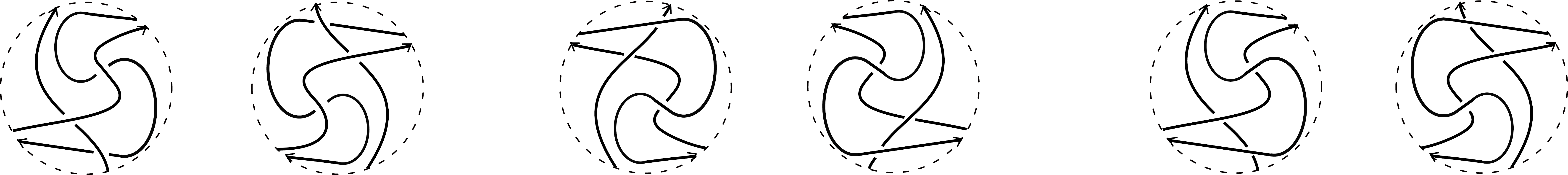}}
\end{equation*}
\begin{equation*}
\centre{
\labellist \small \hair 2pt
\pinlabel{$\leftrightsquigarrow$}  at 439 190
\pinlabel{{\scriptsize $\Omega 3b$}}  at 445 260
\pinlabel{$\leftrightsquigarrow$}  at 1625 190
\pinlabel{{\scriptsize $\Omega 3c$}}  at 1630 260
\pinlabel{$\leftrightsquigarrow$}  at 2865 190
\pinlabel{{\scriptsize $\Omega 3d$}}  at 2869 260
\endlabellist
\centering
\includegraphics[width=0.9\textwidth]{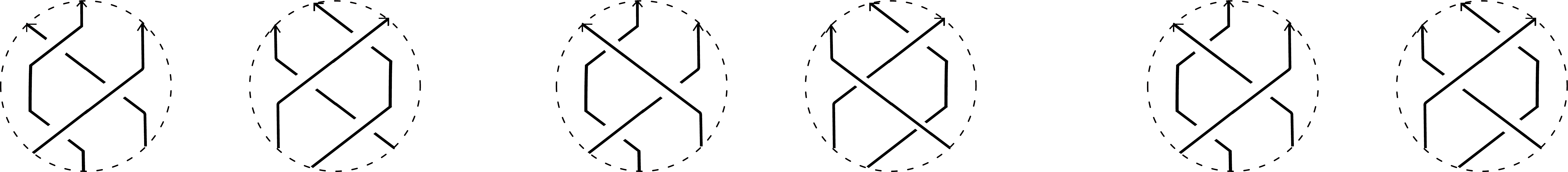}}
\end{equation*}
\begin{equation*}
\centre{
\labellist \small \hair 2pt
\pinlabel{$\leftrightsquigarrow$}  at 439 190
\pinlabel{{\scriptsize $\Omega 3e$}}  at 445 260
\pinlabel{$\leftrightsquigarrow$}  at 1625 190
\pinlabel{{\scriptsize $\Omega 3f$}}  at 1630 260
\pinlabel{$\leftrightsquigarrow$}  at 2865 190
\pinlabel{{\scriptsize $\Omega 3g$}}  at 2869 260
\endlabellist
\centering
\includegraphics[width=0.9\textwidth]{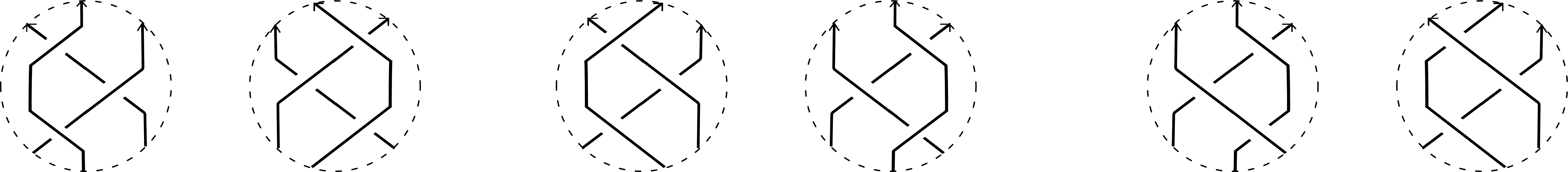}}
\end{equation*}
\begin{equation*}
\centre{
\labellist \small \hair 2pt
\pinlabel{$\leftrightsquigarrow$}  at 439 190
\pinlabel{{\scriptsize $\Omega 3h1$}}  at 445 260
\pinlabel{$\leftrightsquigarrow$}  at 1625 190
\pinlabel{{\scriptsize $\Omega 3h2$}}  at 1630 260
\pinlabel{$\leftrightsquigarrow$}  at 2865 190
\pinlabel{{\scriptsize $\Omega 3h3$}}  at 2869 260
\endlabellist
\centering
\includegraphics[width=0.9\textwidth]{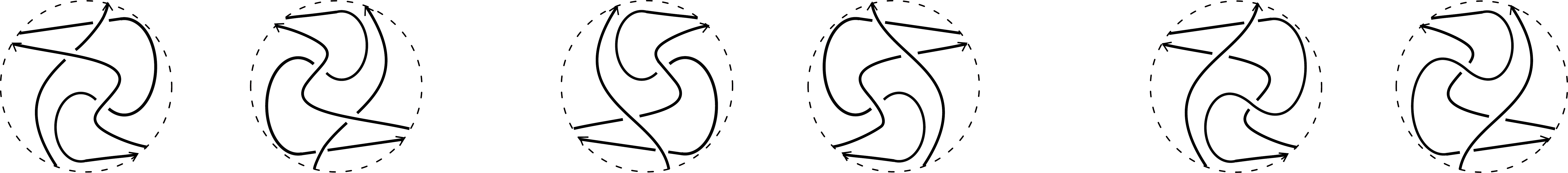}}
\end{equation*}
\begin{equation*}
\centre{
\labellist \small \hair 2pt
\pinlabel{$\leftrightsquigarrow$}  at 439 190
\pinlabel{{\scriptsize $\Omega 3h4$}}  at 445 260
\pinlabel{$\leftrightsquigarrow$}  at 1625 190
\pinlabel{{\scriptsize $\Omega 3h5$}}  at 1630 260
\pinlabel{$\leftrightsquigarrow$}  at 2865 190
\pinlabel{{\scriptsize $\Omega 3h6$}}  at 2869 260
\endlabellist
\centering
\includegraphics[width=0.9\textwidth]{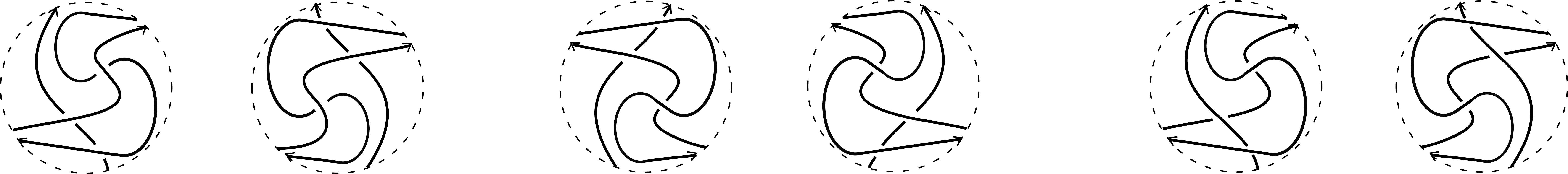}}.
\end{equation*}
The reason why we obtain so many more moves in the rotational setting is that not only we have to assign arbitrary orientations to the diagrams displayed in \eqref{eq:unfr_unor_Rmoves} and \eqref{eq:unfr_unor_Rmoves_2} but also  for each of the moves the rotational diagrams resulting from viewing the oriented \eqref{eq:unfr_unor_Rmoves} and \eqref{eq:unfr_unor_Rmoves_2} from all possible angles, given that we have a preferred, fixed vertical direction. For instance, Polyak's $\Omega 3a$ move splits into six different moves, each of them obtained from Polyak's picture rotated by $2 \pi k/6$ for $0 \leq k \leq 5$.

In order to obtain a generating set of rotational Reidemeister moves, one could simply make Polyak's minimal generating set of moves \eqref{eq:polyak_gen_1} and \eqref{eq:polyak_gen_2} rotational and then add whatever is needed to obtain planar isotopy from Morse planar isotopy (which, by \cref{cor:planar->Morse} below, amounts to four additional moves). The main downside is that, as explained in the previous paragraph, Polyak's $\Omega 3a$ move splits into six different moves, which yields a total of 2+1+6+4=13 moves, far more than desirable. Bar-Natan and van der Veen considered in \cite{barnatanveenpolytime} a set with 16 generators, six of them braid-like  Reidemeister 3 moves, which suffers from the same problem. For sets with fewer Reidemeister 3 moves, different sets of generators have appeared in the literature. In \cite{barnatanveenAPAI}, they claim that a set of six rotational Reidemeister moves suffices (including a single braid-like Reidemeister 3 move), but this is discarded by \cite[Figure 8]{polyak10} even in the non-rotational setting  (and essentially the same issue appears in \cite[Figure 10]{barnatanveengaussians}). The first author adds two extra generators to the former generating set (in the framed case) in \cite{becerra_gaussians}, but \cite[Figure 8]{polyak10} again rules out this set as generating. With the same purpose as Polyak's paper, we believe it is time for a careful treatment.

In this paper we present  a truly generating set of rotational Reidemeister moves with two Reidemeister 1 moves,  one Reidemeister 2 move, one Reidemeister 3 move and additionally four more moves to account for the restriction to Morse planar isotopy.

\begin{theorem}\label{thm:1}
Two rotational tangle diagrams represent the same oriented, unframed tangle if and only if they are related by Morse planar isotopy and a finite sequence of the rotational Reidemeister moves $\Omega 0 a$, $\Omega 0 b$, $\Omega 0 c$, $\Omega 0 d$, $\Omega 1 a$, $\Omega 1b$, $\Omega 2a$ and $\Omega 3 a1$, shown below:
\begin{equation*}
\centre{
\labellist \small \hair 2pt
\pinlabel{$\leftrightsquigarrow$}  at 439 190
\pinlabel{{\scriptsize $\Omega 0a$}}  at 445 260
\pinlabel{$\leftrightsquigarrow$}  at 1625 190
\pinlabel{{\scriptsize $\Omega 0b$}}  at 1630 260
\pinlabel{$\leftrightsquigarrow$}  at 2865 190
\pinlabel{{\scriptsize $\Omega 0c$}}  at 2869 260
\endlabellist
\centering
\includegraphics[width=0.9\textwidth]{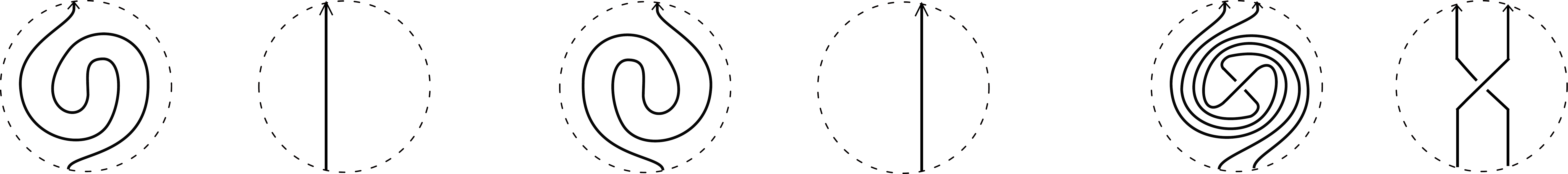}}
\end{equation*}
\begin{equation*}
\centre{
\labellist \small \hair 2pt
\pinlabel{$\leftrightsquigarrow$}  at 439 190
\pinlabel{{\scriptsize $\Omega 0d$}}  at 445 260
\pinlabel{$\leftrightsquigarrow$}  at 1625 190
\pinlabel{{\scriptsize $\Omega 1a$}}  at 1630 260
\pinlabel{$\leftrightsquigarrow$}  at 2865 190
\pinlabel{{\scriptsize $\Omega 1b$}}  at 2869 260
\endlabellist
\centering
\includegraphics[width=0.9\textwidth]{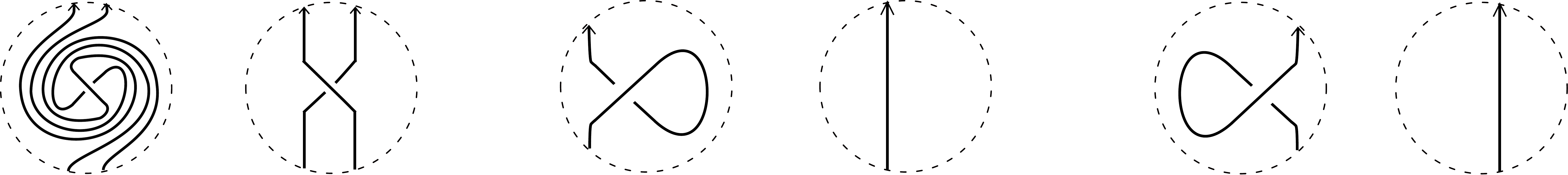}}
\end{equation*}
\begin{equation*}
\centre{
\labellist \small \hair 2pt
\pinlabel{$\leftrightsquigarrow$}  at 439 190
\pinlabel{{\scriptsize $\Omega 2a$}}  at 445 260
\pinlabel{$\leftrightsquigarrow$}  at 1675 190
\pinlabel{{\scriptsize $\Omega 3a1$}}  at 1680 260
\endlabellist
\centering
\includegraphics[width=0.6\textwidth]{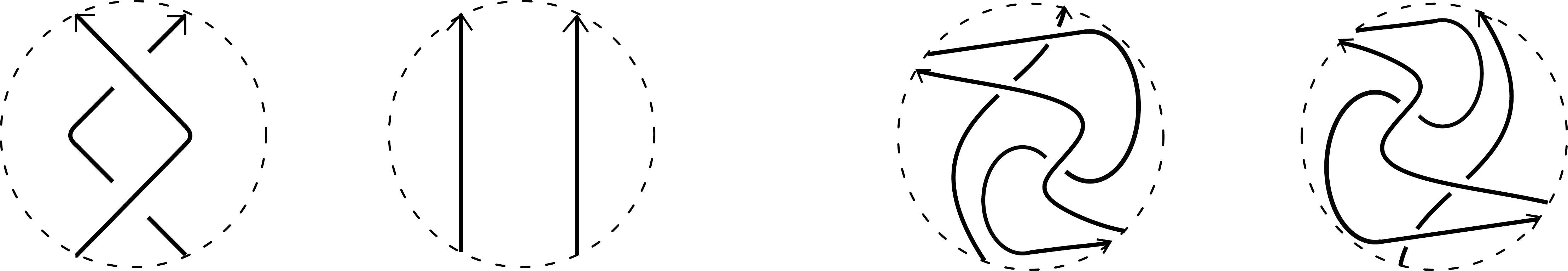}}. \phantom{.}
\end{equation*}
\end{theorem}

This set of rotational Reidemeister moves is in fact minimal. Indeed, we include rotational versions of Polyak's minimal set of generators, and besides we show in \cref{cor:planar->Morse} that the change from planar isotopy to Morse planar isotopy amounts to adding the four $\Omega 0$ moves in the theorem. Furthermore,  we make use of \cite{CS} to describe all minimal (i.e. 8-element) sets of rotational Reidemeister moves, see \cref{thm:all_min_sets}. If one insists in using the braid-like $\Omega 3b$ move as the only Reidemeister 3 move, then  \cite[Theorem 1.2]{polyak10} admits a straightforward rotational analogue, see \cref{thm:polyak_1.2_rot}.

Another case of interest is to consider rotational tangle diagrams to represent isotopy classes of oriented, framed tangles in $D^2 - \{(0, \pm 1) \}$. In this situation, the tangle diagram is endowed with the blackboard framing as it represents the core of a band in $D^3$. Therefore,  the Reidemeister moves $\Omega 1a$ -- $\Omega 1d$ need to be replaced by the framed Reidemeister one moves $\Omega 1\text{f}a$ -- $\Omega 1\text{f}f$ depicted below:
\begin{equation*}
\centre{
\labellist \small \hair 2pt
\pinlabel{$\leftrightsquigarrow$}  at 439 190
\pinlabel{{\scriptsize $\Omega 1\text{f}a$}}  at 445 260
\pinlabel{$\leftrightsquigarrow$}  at 1860 190
\pinlabel{{\scriptsize $\Omega 1\text{f}b$}}  at 1865 260
\endlabellist
\centering
\includegraphics[width=0.6\textwidth]{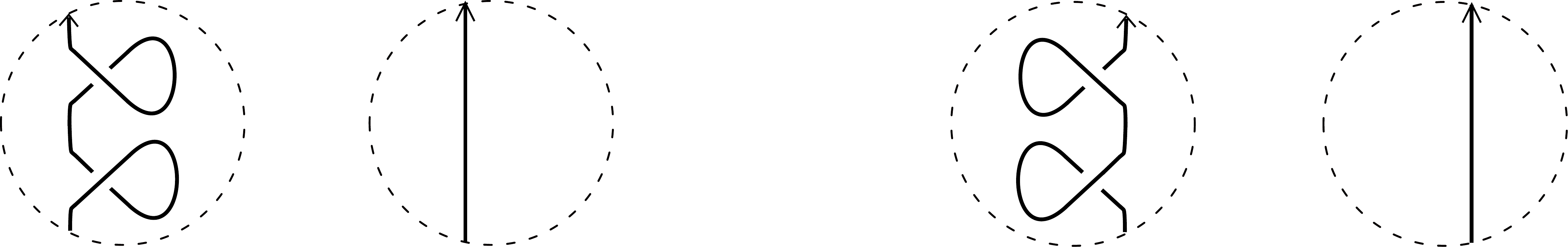}}
\end{equation*}
\begin{equation*}
\centre{
\labellist \small \hair 2pt
\pinlabel{$\leftrightsquigarrow$}  at 439 190
\pinlabel{{\scriptsize $\Omega 1\text{f}c$}}  at 445 260
\pinlabel{$\leftrightsquigarrow$}  at 1860 190
\pinlabel{{\scriptsize $\Omega 1\text{f}d$}}  at 1865 260
\endlabellist
\centering
\includegraphics[width=0.6\textwidth]{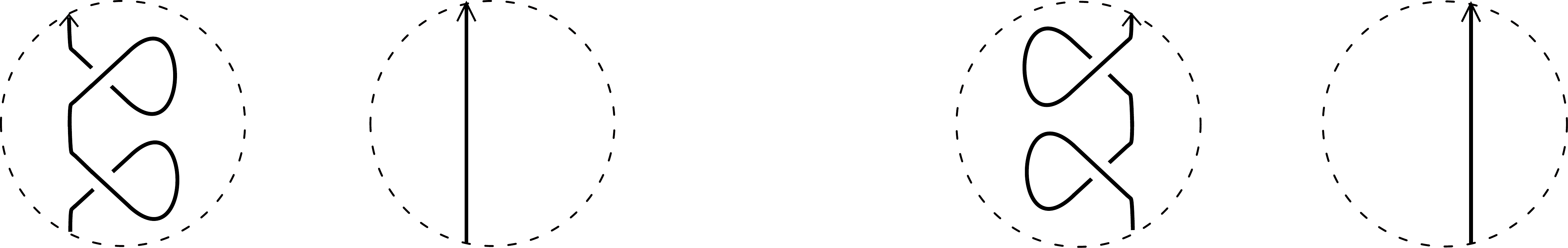}}
\end{equation*}
\begin{equation*}
\centre{
\labellist \small \hair 2pt
\pinlabel{$\leftrightsquigarrow$}  at 439 190
\pinlabel{{\scriptsize $\Omega 1\text{f}e$}}  at 445 260
\pinlabel{$\leftrightsquigarrow$}  at 1860 190
\pinlabel{{\scriptsize $\Omega 1\text{f}f$}}  at 1865 260
\endlabellist
\centering
\includegraphics[width=0.6\textwidth]{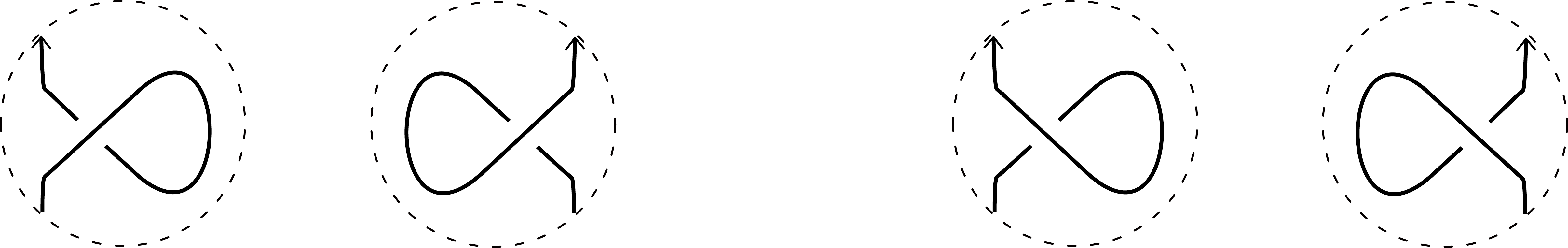}}.
\end{equation*}

Our second main result is to give a family of genuine generating sets of rotational Reidemeister moves for framed, oriented, rotational tangles:

\begin{theorem}\label{thm:2}
Let $S$ be a set of rotational Reidemeister moves formed by all $\Omega 0$, one of the moves $\Omega 2a$ or $\Omega 2b$, one of the pairs $$(\Omega 1 \textup{f}a, \Omega 2di)  \qquad , \qquad (\Omega 1 \textup{f}c, \Omega 2di)  \qquad , \qquad  (\Omega 1 \textup{f}b, \Omega 2ci)  \qquad , \qquad  (\Omega 1 \textup{f}d, \Omega 2ci)  $$ with $i=1,2$,
and  one move of each of the types $\Omega 3a$ and $\Omega 3h$.

Then $S$ is generating, that is, two rotational tangle diagrams represent the same oriented, framed tangle if and only if they are related by Morse planar isotopy and a finite sequence of the rotational Reidemeister moves in $S$ (an instance is shown below):
\begin{equation*}
\centre{
\labellist \small \hair 2pt
\pinlabel{$\leftrightsquigarrow$}  at 439 190
\pinlabel{{\scriptsize $\Omega 0a$}}  at 445 260
\pinlabel{$\leftrightsquigarrow$}  at 1625 190
\pinlabel{{\scriptsize $\Omega 0b$}}  at 1630 260
\pinlabel{$\leftrightsquigarrow$}  at 2865 190
\pinlabel{{\scriptsize $\Omega 0c$}}  at 2869 260
\endlabellist
\centering
\includegraphics[width=0.9\textwidth]{figures/gens_1}}
\end{equation*}
\begin{equation*}
\centre{
\labellist \small \hair 2pt
\pinlabel{$\leftrightsquigarrow$}  at 439 190
\pinlabel{{\scriptsize $\Omega 0d$}}  at 445 260
\pinlabel{$\leftrightsquigarrow$}  at 1625 190
\pinlabel{{\scriptsize $\Omega 1 \textup{f}a$}}  at 1630 260
\pinlabel{$\leftrightsquigarrow$}  at 2865 190
\pinlabel{{\scriptsize $\Omega 2a$}}  at 2869 260
\endlabellist
\centering
\includegraphics[width=0.9\textwidth]{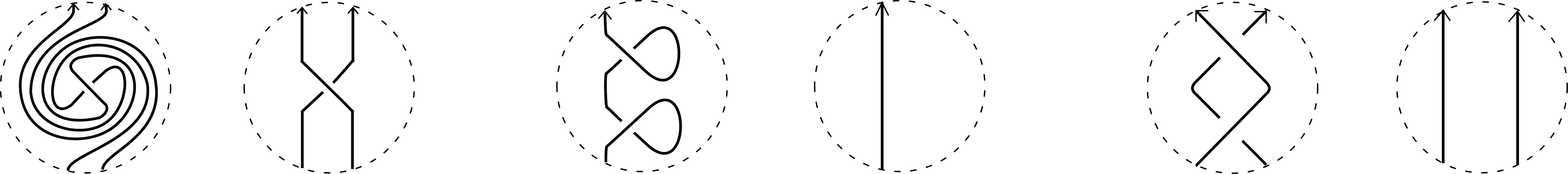}}
\end{equation*}
\begin{equation*}
\centre{
\labellist \small \hair 2pt
\pinlabel{$\leftrightsquigarrow$}  at 439 190
\pinlabel{{\scriptsize $\Omega 2d1$}}  at 445 260
\pinlabel{$\leftrightsquigarrow$}  at 1635 190
\pinlabel{{\scriptsize $\Omega 3a1$}}  at 1650 260
\pinlabel{$\leftrightsquigarrow$}  at 2885 190
\pinlabel{{\scriptsize $\Omega 3h1$}}  at 2895 260
\endlabellist
\centering
\includegraphics[width=0.9\textwidth]{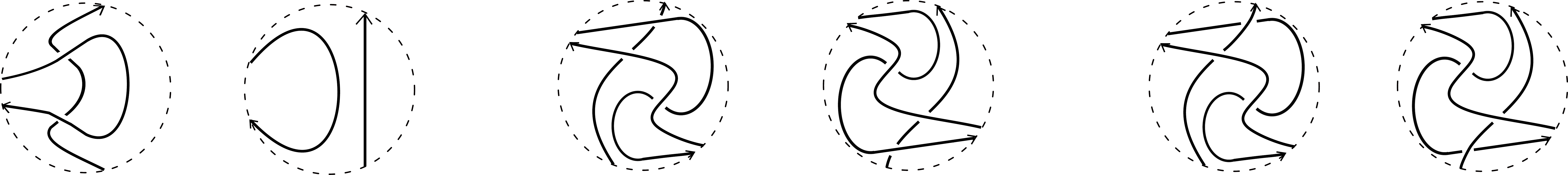}}.
\end{equation*}
\end{theorem}

It is surprising that there is a striking difference between the framed and unframed cases. Whereas in the unframed case at least two Reidemeister 1 moves are needed, cf. \cite[Lemma 3.1]{polyak10}, in the framed situation we find that at least two Reidemeister 2 moves are needed, see \cref{lem:2_O2}.

Moreover, it turns out that any of the generating 9-element sets described in the previous theorem is in fact minimal, we prove this in  \cref{thm:1.2_minimal}. This automatically gives a family of minimal generating sets of framed, oriented Reidemeister moves, see \cref{cor:pol_fr}. To the authors' knowledge, no such minimal sets were known in the framed setting. In particular it implies that such minimal generating set has 5 elements. If one wants to include a single Reidemeister 3 move (e.g. the braid-like $\Omega 3b$), then at least one extra generator is needed, see \cref{thm:another_fr}.

One important consequence of these results, and really the main motivation of this paper, is that we obtain a minimal set of relations that needs to be imposed to define a genuine isotopy invariant of framed, oriented tangles built as  the universal tangle invariant \cite{barnatanveengaussians,becerra_gaussians,becerra_thesis}. In the rotational setting, each diagram can be interpreted as an algebraic expression, and each of the rotational Reidemeister moves can be translated as an algebraic equation, see \cref{sec:application}.


\subsection*{Acknowledgments} The authors would like to thank Carmen Caprau for a helpful conversation. The first author was supported by the ARN project CPJ number ANR-22-CPJ1-0001-0  at the Institut de Mathématiques de Bourgogne (IMB). The IMB receives support from the EIPHI Graduate School (contract ANR-17-EURE-0002).

\section{Rotational diagrams}

In this section we introduce rotational diagrams of link and tangle diagrams. This type of diagrams appeared first in \cite{barnatanveenpolytime} in the virtual setting and in \cite{becerra_thesis,becerra_refined} in the upwards tangle setting.

\subsection{Tangle diagrams in discs}

Let $D^n$ denote the $n$-disc in $\R^n$. A  \textit{tangle diagram} in $D^2$ is an oriented immersion $$D: \left( \coprod_n D^1 \right) \amalg \left( \coprod_m  S^1 \right) \to D^2 - \{ (0,\pm 1) \}$$   which has only finitely many double points of transversal self-intersections and which is such that the image of $\amalg_n \partial D^1$ lies on $\partial (D^2- \{ (0,\pm 1) \})$. Double points are called  crossings and  carry additionally  the under/over-pass information. If $n=0$, we say that the tangle diagram is a link diagram. Note that we opt to use the bigon $D^2 - \{ (0,\pm 1)\}$ in our definition of a tangle diagram instead of a disc, as this is necessary for the upcoming discussion on rotational diagrams. We will, however, refer to those bigons as discs in this paper.

Typically, we will regard tangles diagrams up to planar isotopy. More precisely, two tangle diagrams $D_0, D_1$ as above are \textit{planar isotopic} if there exists a smooth map
\begin{equation}\label{eq:isotopy}
 H: \left[ \left( \coprod_n D^1 \right) \amalg \left( \coprod_m  S^1 \right) \right] \times [0,1] \to D^2 - \{ (0,\pm 1) \}   
\end{equation}
with $H(-,0)=D_0$,  $H(-,1)=D_1$, and  $H(-,t)$ is a tangle diagram for all $t \in [0,1]$ and the crossing type of the double points is preserved throughout the isotopy.

Note that our definition does not fix endpoints of the open components of the tangle diagram, they can freely move within each of the connected components of $\partial (D^2 - \{ (0,\pm 1) \})$ in an orderly fashion.

\subsection{Rotational diagrams}

We now introduce the class of tangle diagrams that concerns this paper.  A tangle diagram $D$ is said to be  \textit{rotational} if all endpoints and crossings of $D$ point upwards and all maxima and minima appear in pairs of the following two forms,
\begin{equation}\label{eq:full_rotations}
\centre{
\centering
\includegraphics[width=0.30\textwidth]{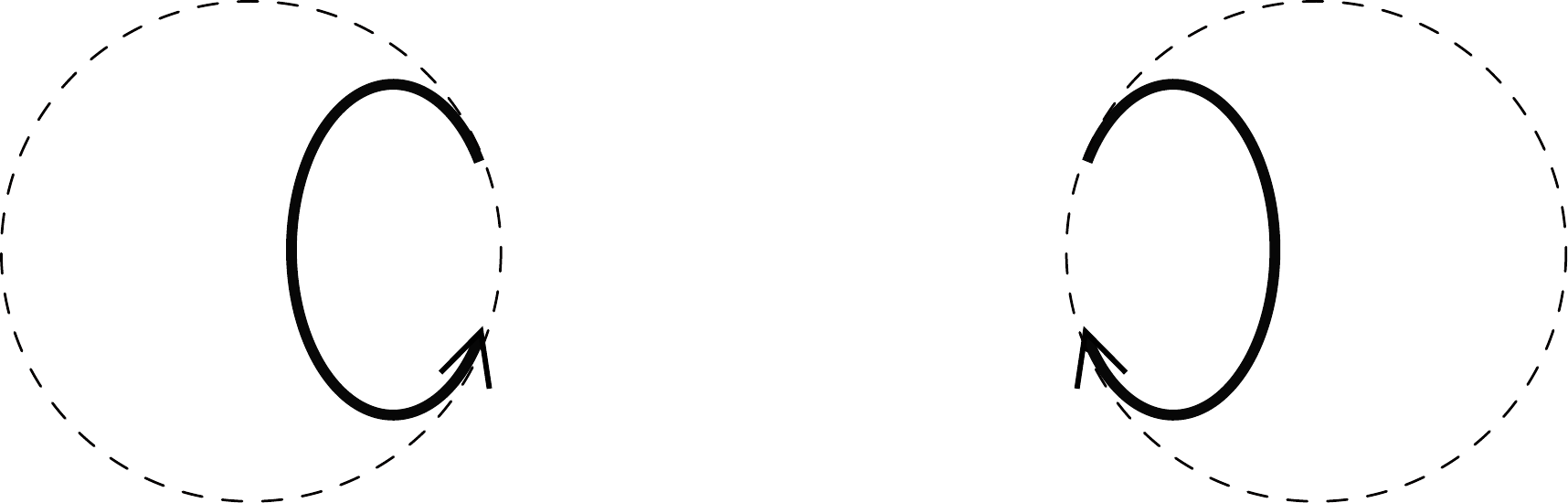}},
\end{equation}
where the dashed discs denote some neighbourhoods of that piece of strand in the tangle diagram.  In essence, these are diagrams that have a special form with respect to the vertical direction. (Partial) closures of braid diagrams are particular examples of rotational diagrams, but the former are far more general.  In this definition, ``pointing upwards'' at a point means that the tangent vector at the given point has positive $y$-component.

Of course, we need a refined notion of planar isotopy that respects the extra properties. Given $D_0, D_1$ rotational tangle diagrams, a planar isotopy $H$ as in \eqref{eq:isotopy} between $D_0$ and $D_1$ is called to be \textit{Morse} if $H(-,t)$ is a rotational diagram for all $t \in [0,1]$ and the number of maxima and minima (with respect to the vertical direction) is constant throughout the isotopy (that is, no maxima or minima can be created or destroyed).

Rotational diagrams are abundant. It was shown in  \cite[Lemma 3.2]{becerra_refined} and  \cite[Lemma 2.2.1]{becerra_thesis} that any upwards tangle (that is, a tangle in a cube with only open components oriented from bottom to top) has a rotational diagram. This is also true in the more general tangle setting that we are dealing with:

\begin{lemma}\label{lem:rot_diagram}
    Any tangle diagram  is planar isotopic to a rotational diagram.
\end{lemma}
\begin{proof}
This follows almost verbatim from the algorithm described in \cite[Lemma 3.2]{becerra_refined} only with minor modifications that we describe now. Firstly, we must slightly isotope the endpoints of the open components so that the head and the tail of each component point up; this is always possible as the points $\{ (0,\pm 1) \}$ are removed. If the tail of a given component is higher than the head, then one of the full rotations from \eqref{eq:full_rotations} will necessarily have to be included in the algorithm at the beginning or the end of the component. Finally we change if necessary the height of the endpoints to remove undesired isolated cups and caps.

For the closed components of the tangle diagram, it suffices to choose an arbitrary point in each of the components, follow the algorithm described in  described in \cite[Lemma 3.2]{becerra_refined}, and finally join the endpoints with a full rotation.
\end{proof}

\begin{example}
The tangle diagrams
\begin{equation*}
\centre{
\centering
\includegraphics[width=0.3\textwidth]{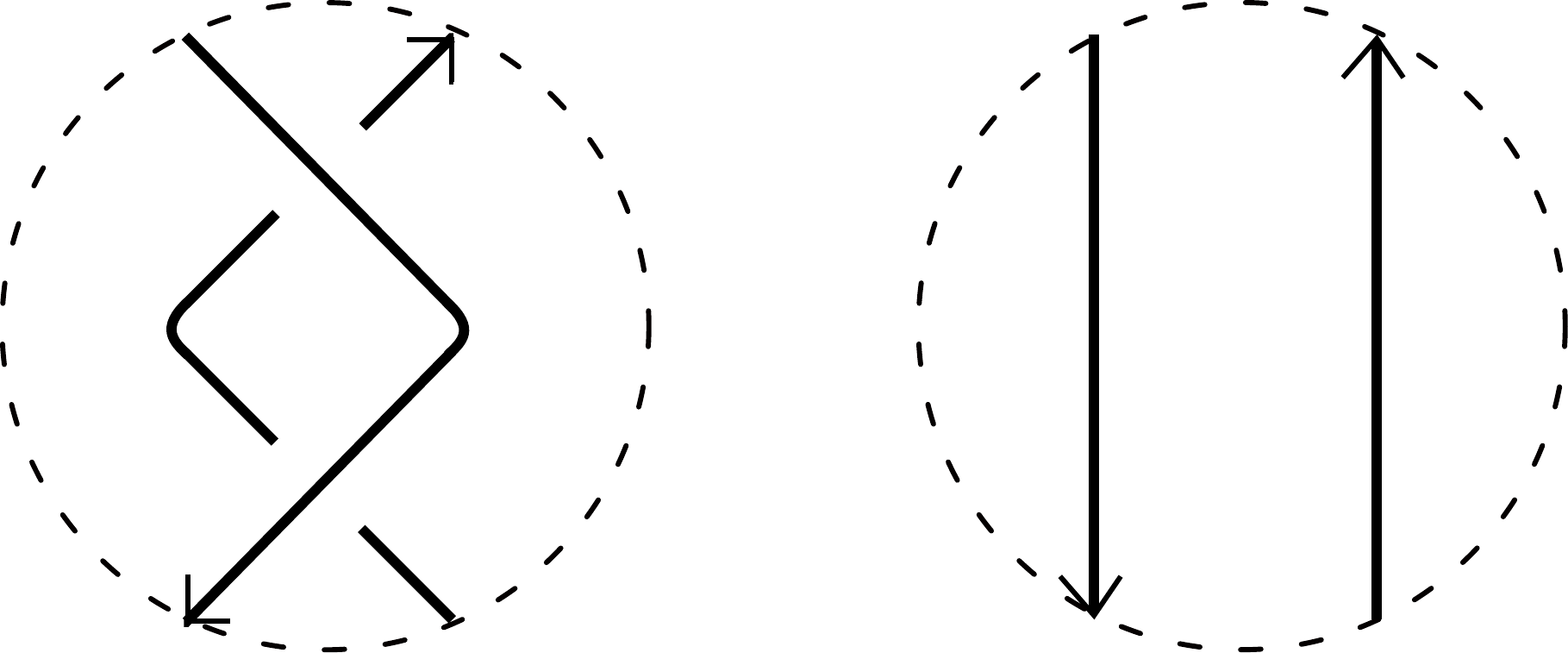}}
\end{equation*}
(which can be viewed as one of the Reidemeister 2 moves) can be planar isotoped, following the algorithm from \cref{lem:rot_diagram}, to the following rotational tangle diagrams:
\begin{equation*}
\centre{
\centering
\includegraphics[width=0.3\textwidth]{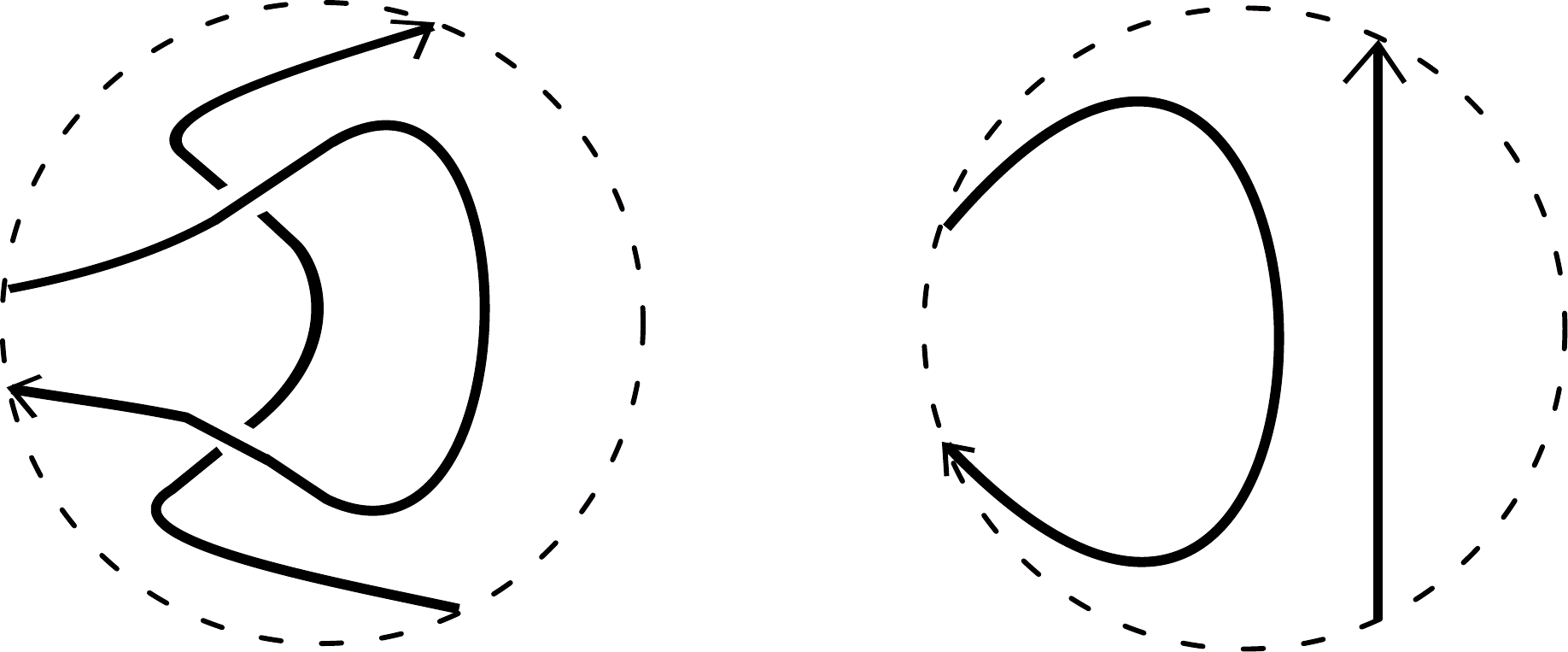}}.
\end{equation*}
Of course, the rotational diagram to which a given tangle diagram can be planar isotoped to is far from unique. For example, the left-hand side tangle diagrams can be planar isotoped to the following rotational diagram as well:
\begin{equation*}
\centre{
\centering
\includegraphics[width=0.14\textwidth]{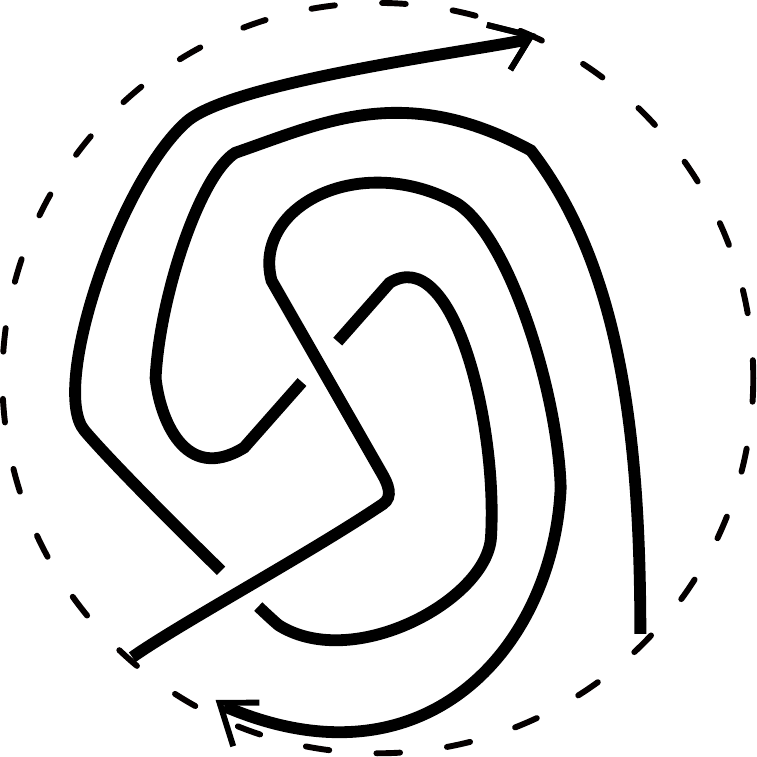}}
\end{equation*}
(which is in fact the diagram resulting from applying the algorithm from \cref{lem:rot_diagram} starting from the other strand).

For a step-wise example of the algorithm, see \cite[Example 3.3]{becerra_refined}.
\end{example}

\subsection{Planar isotopy vs. Morse planar isotopy}

Now we would like to compare the two notions of isotopy we have at hand. It follows from \cref{lem:rot_diagram} that there is a bijection 
$$\begin{tikzcd}[column sep=3em]
\frac{ \left\{ \parbox[c][2.5em]{7em}{\centering {\small  \textnormal{tangle diagrams in $D^2$ }}} \right\} }{  \parbox[c][1.5em]{8em}{\centering  {\small   \textnormal{planar isotopy} }}}  \ar[-,double line with arrow={-,-}]{r} &    \frac{ \left\{ \parbox[c][2.5em]{7em}{\centering     {\small  \textnormal{rotational tangle diagrams in $D^2$}}} \right\} }{ \parbox[c][1.5em]{7em}{\centering  {\small  \textnormal{planar isotopy} }}}
\end{tikzcd}.$$

We would like to mod out by Morse planar isotopy in the right-hand side of this bijection. Among rotational diagrams, every Morse isotopy is a planar isotopy; but the converse is obviously not true. The following proposition fully describes what is the additional data that one needs to add to obtain arbitrary planar isotopies from Morse ones.

\begin{proposition}\label{prop:swirls}
Let $D,D'$ be rotational diagrams. Then $D,D'$ are planar isotopic if and only if they are related by a sequence of  Morse planar isotopies and the swirl moves $Sw^\pm_T$ below for arbitrary $n$-tangle diagrams $T$,
\begin{equation*}
\centre{
\labellist \small \hair 2pt
\pinlabel{$ T$} at 410 410
\pinlabel{$ T$} at 1450 410
\pinlabel{$ T$} at 2795 410
\pinlabel{$ T$} at 3854 410
\pinlabel{$ \leftrightsquigarrow$} at 3320 410
\pinlabel{$ Sw_T^-$} at 3330 520
\pinlabel{$ \leftrightsquigarrow$} at 932 410
\pinlabel{$  Sw_T^+$} at 932 520
\pinlabel{$ ,$} at 2150 410
\pinlabel{$ \cdots$} at 430 50
\pinlabel{$ \cdots$} at 407 303
\pinlabel{$ \cdots$} at 428 521
\pinlabel{$ \cdots$} at 418 762
\pinlabel{$ \cdots$} at 1465 164
\pinlabel{$ \cdots$} at 1465 636
\pinlabel{$ \cdots$} at 2800 40
\pinlabel{$ \cdots$} at 2800 300
\pinlabel{$ \cdots$} at 2785 516
\pinlabel{$ \cdots$} at 2800 768
\pinlabel{$ \cdots$} at 3873 152
\pinlabel{$ \cdots$} at 3873 654
\endlabellist
\centering
\includegraphics[width=0.85\textwidth]{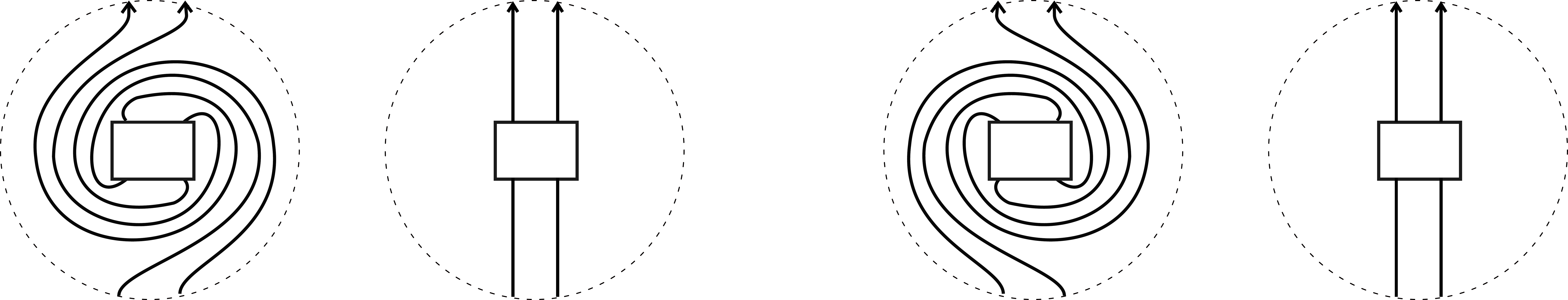}}.
\end{equation*}
\end{proposition}
In the previous statement, an \textit{$n$-tangle diagram} is an upwards tangle diagram with possibly closed components included.
\begin{proof}
 The ``if'' is evident, so we only have to prove the ``only if''. Without loss of generality we can assume that $D$ and $D'$ are $n$-tangles, that $H(-,t)$ is not a rotational diagram only in some interval $(a,b) \subset [0,1]$ and that $D$ and $D'$ are not Morse planar isotopic. Indeed we could rearrange the isotopy to be the identity outside a neighbourhood of a given point during some subinterval of $[0,1]$.

Since $D$ and $D'$ are rotational diagrams, each of them must have the same number of maxima and minima. In particular, the number of maxima created or destroyed during the isotopy must be equal to the number of minima created or destroyed. This means that $D$ and $D'$ are related by adding or removing the same number of cups and caps.  Since isolated zig-zags of cups and caps (with perhaps a tangle diagram in between)  cannot happen as they would destroy the property of being  rotational, if one such zig-zag is created then it must be accompanied by an extra pair of cup and cap making  a pair of clockwise and counterclockwise full rotations   as in \eqref{eq:full_rotations}.
\end{proof}

As an immediate consequence, we have a bijection
$$\begin{tikzcd}[column sep=3em]
\frac{ \left\{ \parbox[c][2.5em]{7em}{\centering {\small  \textnormal{tangle diagrams in $D^2$ }}} \right\} }{  \parbox[c][1.5em]{8em}{\centering  {\small   \textnormal{planar isotopy} }}}  \ar[-,double line with arrow={-,-}]{r} &    \frac{ \left\{ \parbox[c][2.5em]{7em}{\centering     {\small  \textnormal{rotational tangle diagrams in $D^2$}}} \right\} }{ \parbox[c][3.5em]{7em}{\centering  {\small  \textnormal{Morse planar isotopy and swirl moves} }}}
\end{tikzcd}.$$

A downside of the previous result is that one needs to add infinitely many swirl moves to upgrade Morse planar isotopy to general planar isotopy. Fortunately, only the swirl moves of three tangles are needed to obtain the rest.

\begin{corollary}\label{cor:planar->Morse}
Let $D,D'$ be rotational diagrams. Then $D,D'$ are planar isotopic if and only if they are related by a sequence of  Morse planar isotopies and the swirl moves $Sw^\pm_T$ from \cref{prop:swirls} for $T$ equals the positive crossing $\PC$, the negative crossing $\NC$ and the trivial one-component tangle $\uparrow$. 

In particular,  $Sw^{-}(\PC)$ (resp. $Sw^{-}(\NC)$) is consequence of $Sw^{+}(\PC)$ (resp. $Sw^{+}(\NC)$) and $Sw^{\pm}(\uparrow)$.
\end{corollary}
\begin{proof}
The first observation is that given two (rotational) $n$-tangle diagrams $T,T'$, if $T \otimes T'$ denotes the tangle resulting from placing $T'$ to the right of $T$, then the swirl moves $Sw_T^\pm$ and $Sw_{T'}^\pm$ imply $Sw_{T \otimes T'}^\pm$. This implies the swirl moves $Sw_T^\pm$ for $T =\uparrow_n := \uparrow \overset{n}{\cdots} \uparrow$ the trivial $n$-tangle diagram.

Next,  given two $n$-tangle diagrams $T,T'$ with the same number of components, let $T' \circ T$ be the upwards tangle diagram resulting from placing $T'$ on top of $T$. Then the swirl moves $Sw_T^\pm$ and $Sw_{T'}^\pm$ imply $Sw_{T' \circ T}^\pm$. Indeed, $T' \circ T$ is equivalent to the diagram resulting from placing the left-hand side of $Sw_{T'}^\pm$ on top of the left-hand side of $Sw_{T}^\pm$, and then applying $Sw_{\uparrow_n}^\pm$ in the middle.

Lastly, if $T$ is an $(n+m)$-tangle, let $\mathrm{cl}_m(T)$ the the $n$-tangle resulting from closing up the $m$ rightmost components of $T$. Then $Sw_T^\pm$ implies $Sw^\pm (\mathrm{cl}_m(T))$. Indeed
\begin{equation*}
\centre{
\labellist \small \hair 2pt
\pinlabel{$\leftrightsquigarrow$}  at 600 500
\pinlabel{{\scriptsize $Sw_T^+$}}  at 605 570
\pinlabel{$T$}  at 107 495
\pinlabel{$T$}  at 1126 495
\pinlabel{$T$}  at 2264 495
\pinlabel{$\leftrightsquigarrow$}  at 1733 500
\pinlabel{{\scriptsize $Sw^-(\uparrow_n)$}}  at 1738 570
\endlabellist
\centering
\includegraphics[width=0.6\textwidth]{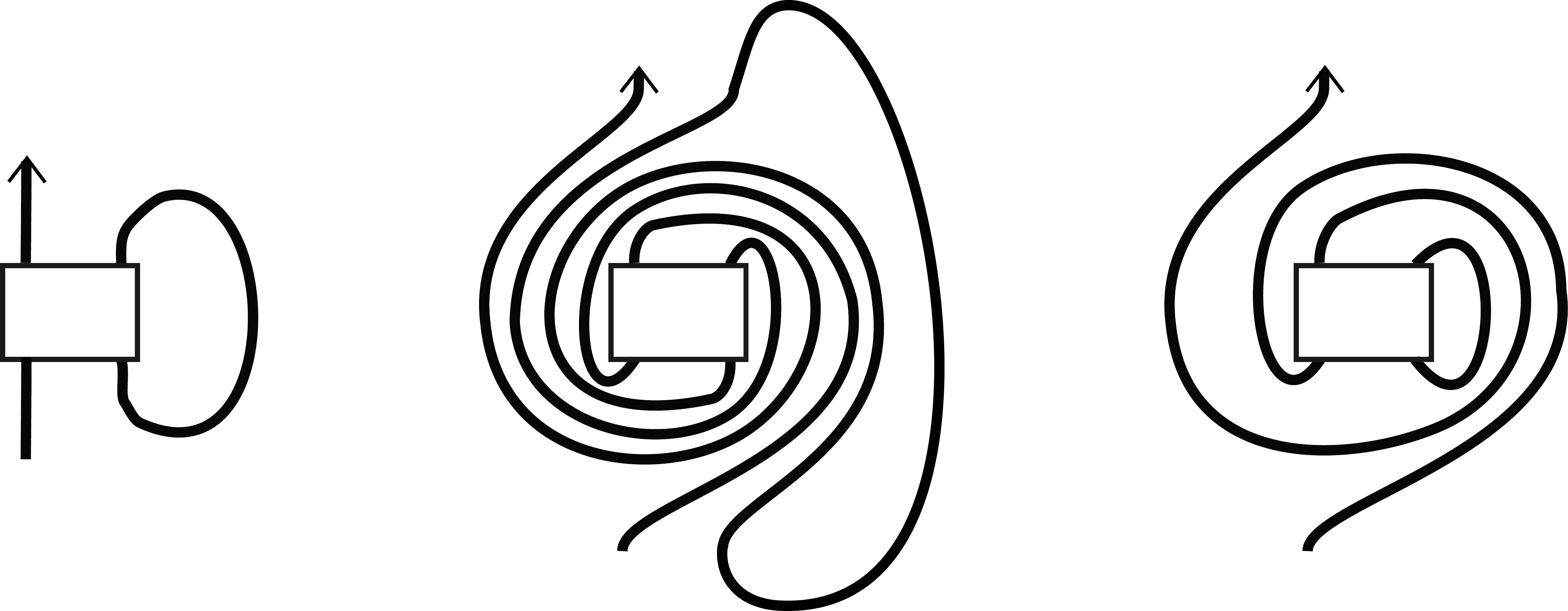}},
\end{equation*}
where the thick lines denote a bundle of strands. 
A similar argument follows when we close up the $n$ leftmost components of $T$. Combining horizontal and vertical stacking with left and right partial closure, we exhaust all possible $n$-tangle diagrams.

For the last part, it suffices to note that
\begin{equation*}
\centre{
\labellist \small \hair 2pt
\pinlabel{$\leftrightsquigarrow$}  at 400 330
\pinlabel{{\scriptsize $Sw^+(\PC)$}}  at 405 400
\pinlabel{$\leftrightsquigarrow$}  at 1110 330
\pinlabel{{\scriptsize $Sw^-(\uparrow_2)$}}  at 1115 400
\pinlabel{{\scriptsize $(\times 2)$}}  at 1115 250
\endlabellist
\centering
\includegraphics[width=0.45\textwidth]{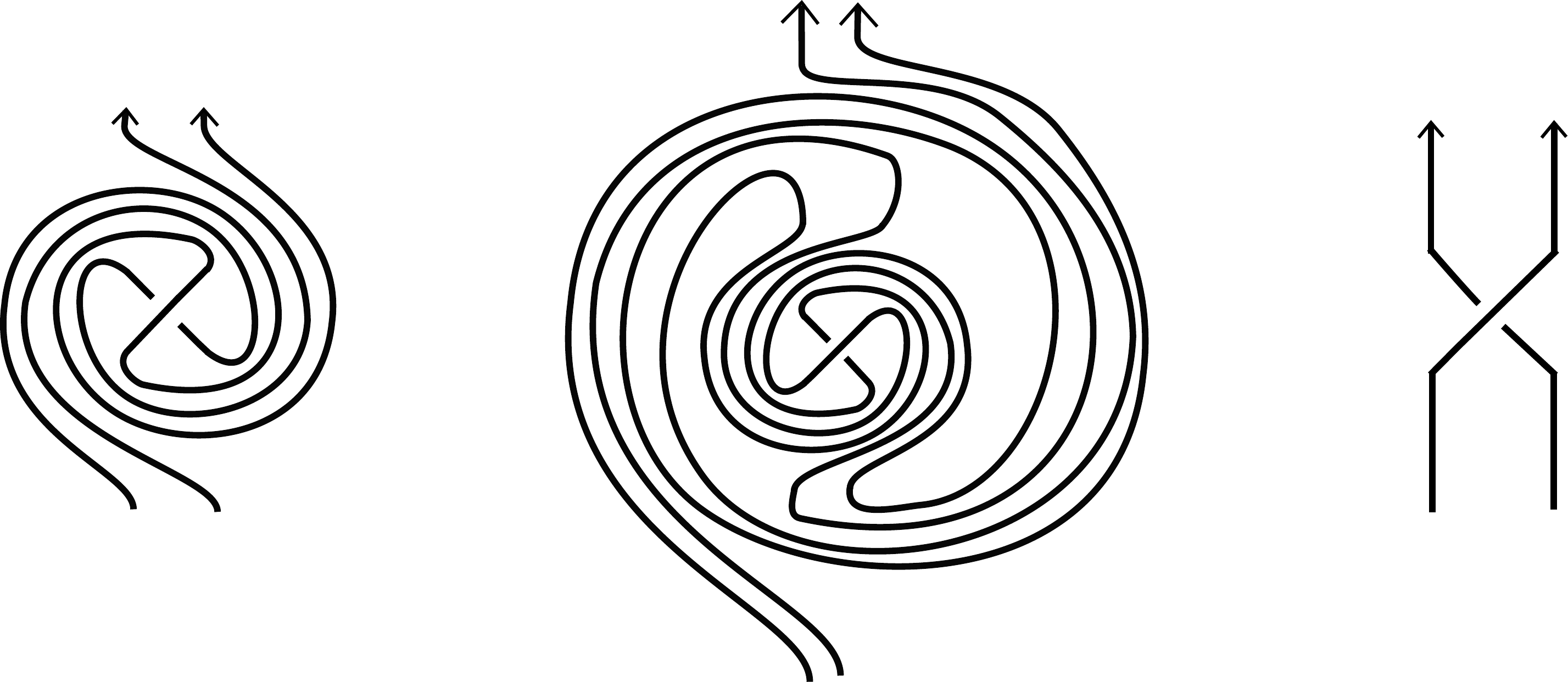}}
\end{equation*}
and similarly for the negative crossing.
\end{proof}

Let us call for convenience 
$$\Omega 0 a := Sw^+(\uparrow) \qquad , \qquad \Omega 0 b := Sw^-(\uparrow)  \qquad , \qquad \Omega 0 c := Sw^+(\PC)  \qquad , \qquad \Omega 0 d := Sw^+(\NC).$$ 

Then, the previous bijection can be rewritten as
\begin{equation}
\begin{tikzcd}[column sep=3em]
\frac{ \left\{ \parbox[c][2.5em]{7em}{\centering {\small  \textnormal{tangle diagrams in $D^2$ }}} \right\} }{  \parbox[c][1.5em]{8em}{\centering  {\small   \textnormal{planar isotopy} }}}  \ar[-,double line with arrow={-,-}]{r} &    \frac{ \left\{ \parbox[c][2.5em]{7em}{\centering     {\small  \textnormal{rotational tangle diagrams in $D^2$}}} \right\} }{ \parbox[c][4.5em]{6em}{\centering  {\small  \textnormal{Morse planar isotopy and $\Omega 0 a$ -- $\Omega 0 d$} }}}
\end{tikzcd}
\end{equation}
Note that this set of generating moves for the swirls is minimal, as one clearly cannot realise any of them from the other three.

\section{A minimal set of rotational Reidemeister moves}\label{sec:3}

\subsection{Proof of \cref{thm:1}}\label{sec:proof_1.1}
The goal of this subsection is to give a complete proof of \cref{thm:1}. Essentially, it consists of a careful adaptation of Polyak's arguments to the rotational setting. Also, by means of clarity, from now on we will not explicitly draw the disc where the tangle sits.

We start by generating $\Omega 3a2$ -- $\Omega 3a6$ from $\Omega 3a1$ and the swirls. To this end, and to obtain the moves $\Omega 2c2$ and $\Omega 2d2$, we will make use of an auxiliary

\begin{lemma}\label{lem:the_lemma}
The moves $\Omega 0a$ and $Sw^{-}(\PC)$ imply the following move
\begin{equation}\label{eq:lemma1}
\centre{
\labellist \small \hair 2pt
\pinlabel{$\leftrightsquigarrow$}  at 275 130
\endlabellist
\centering
\includegraphics[width=0.2\textwidth]{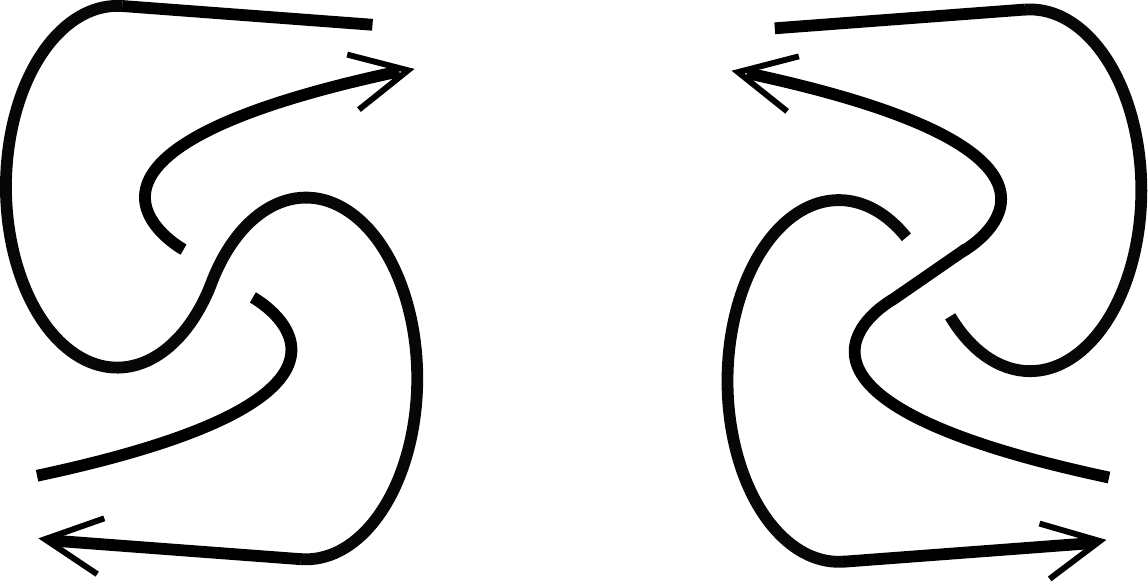}}.
\end{equation}
Similarly, the moves $\Omega 0b$ and $\Omega 0d$ imply the mirror image of \eqref{eq:lemma1}:
\begin{equation}\label{eq:lemma2}
\centre{
\labellist \small \hair 2pt
\pinlabel{$\leftrightsquigarrow$}  at 275 130
\endlabellist
\centering
\includegraphics[width=0.2\textwidth]{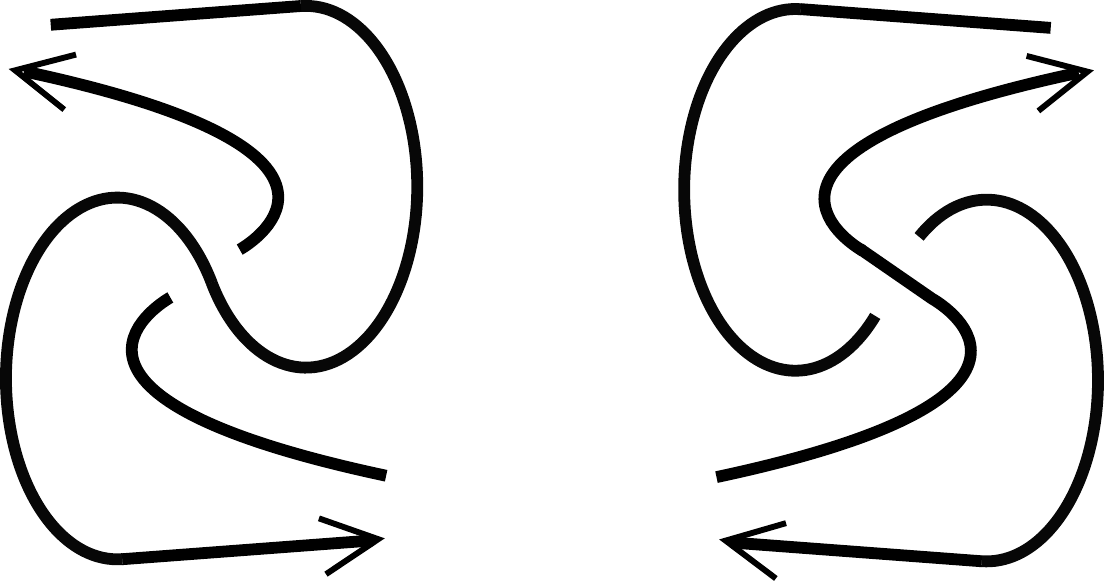}}.
\end{equation}
\end{lemma}
\begin{proof}
    We have
    \begin{equation*}
\centre{
\labellist \small \hair 2pt
\pinlabel{$\leftrightsquigarrow$}  at 300 300
\pinlabel{{\scriptsize $Sw^{-}(\PC)$}}  at 305 370
\pinlabel{$\leftrightsquigarrow$}  at 960 300
\pinlabel{{\scriptsize \begin{tabular}{c}Morse \\ isotopy\end{tabular} }}  at 970 370
\pinlabel{$\leftrightsquigarrow$}  at 1500 300
\pinlabel{{\scriptsize $\Omega 0a$}}  at 1505 370
\pinlabel{$\leftrightsquigarrow$}  at 2000 300
\pinlabel{{\scriptsize $\Omega 0a$}}  at 2005 370
\endlabellist
\centering
\includegraphics[width=0.9\textwidth]{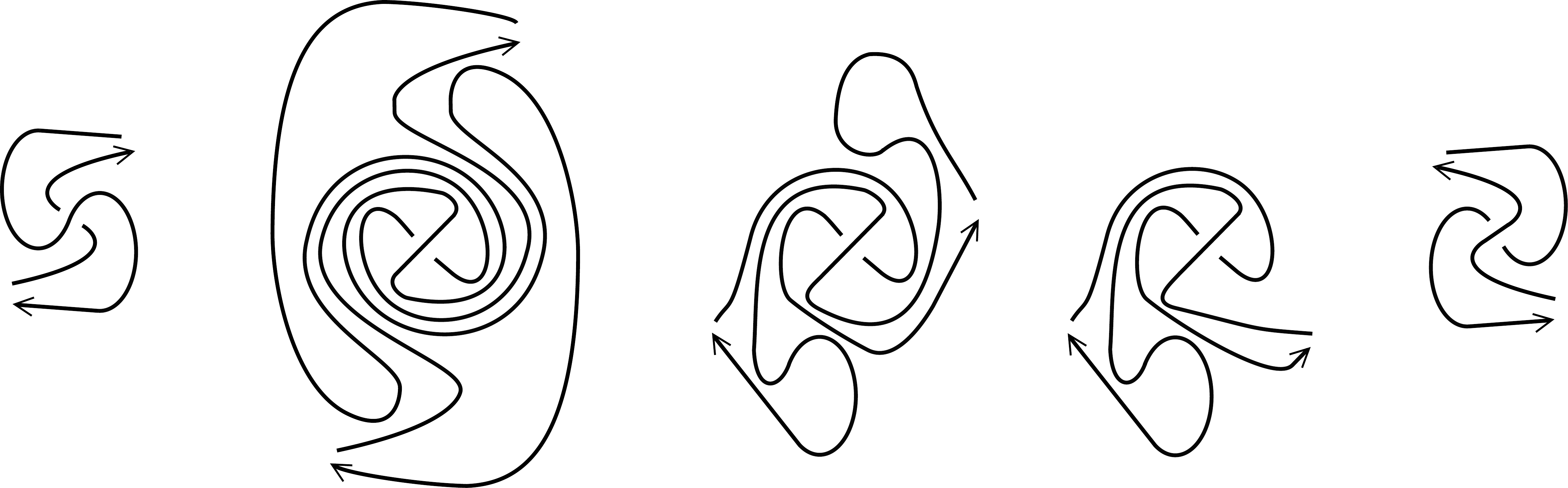}}
\end{equation*}
and its mirror image
    \begin{equation*}
\centre{
\labellist \small \hair 2pt
\pinlabel{$\leftrightsquigarrow$}  at 300 300
\pinlabel{{\scriptsize $\Omega 0d$}}  at 305 370
\pinlabel{$\leftrightsquigarrow$}  at 960 300
\pinlabel{{\scriptsize \begin{tabular}{c}Morse \\ isotopy\end{tabular} }}  at 970 370
\pinlabel{$\leftrightsquigarrow$}  at 1500 300
\pinlabel{{\scriptsize $\Omega 0b$}}  at 1505 370
\pinlabel{$\leftrightsquigarrow$}  at 2000 300
\pinlabel{{\scriptsize $\Omega 0b$}}  at 2005 370
\endlabellist
\centering
\includegraphics[width=0.9\textwidth]{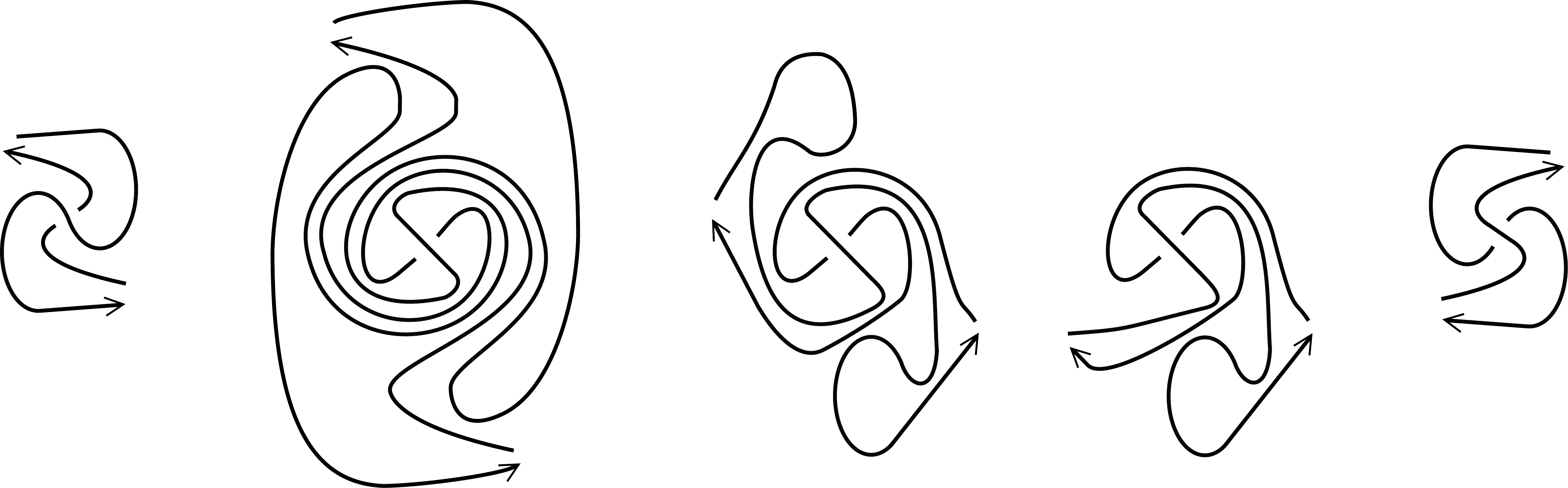}}.
\end{equation*}
\end{proof}

\begin{lemma}\label{lem:O3a}
Each of the moves  $\Omega 3a2$ -- $\Omega 3a6$ can be obtained from $\Omega 3a1$, $\Omega 0a$ and the move in \eqref{eq:lemma1}.
\end{lemma}
\begin{proof}
We realise $\Omega 3a2$ from $\Omega 3a1$ and $\Omega 0a$ via the moves
     \begin{equation*}
\centre{
\labellist \small \hair 2pt
\pinlabel{$\leftrightsquigarrow$}  at 460 330
\pinlabel{{\scriptsize $\Omega 0a$}}  at 465 400
\pinlabel{$\leftrightsquigarrow$}  at 980 330
\pinlabel{{\scriptsize $\Omega 3a1$}}  at 985 400
\pinlabel{$\leftrightsquigarrow$}  at 1500 330
\pinlabel{{\scriptsize $\Omega 0a$}}  at 1505 400
\endlabellist
\centering
\includegraphics[width=0.7\textwidth]{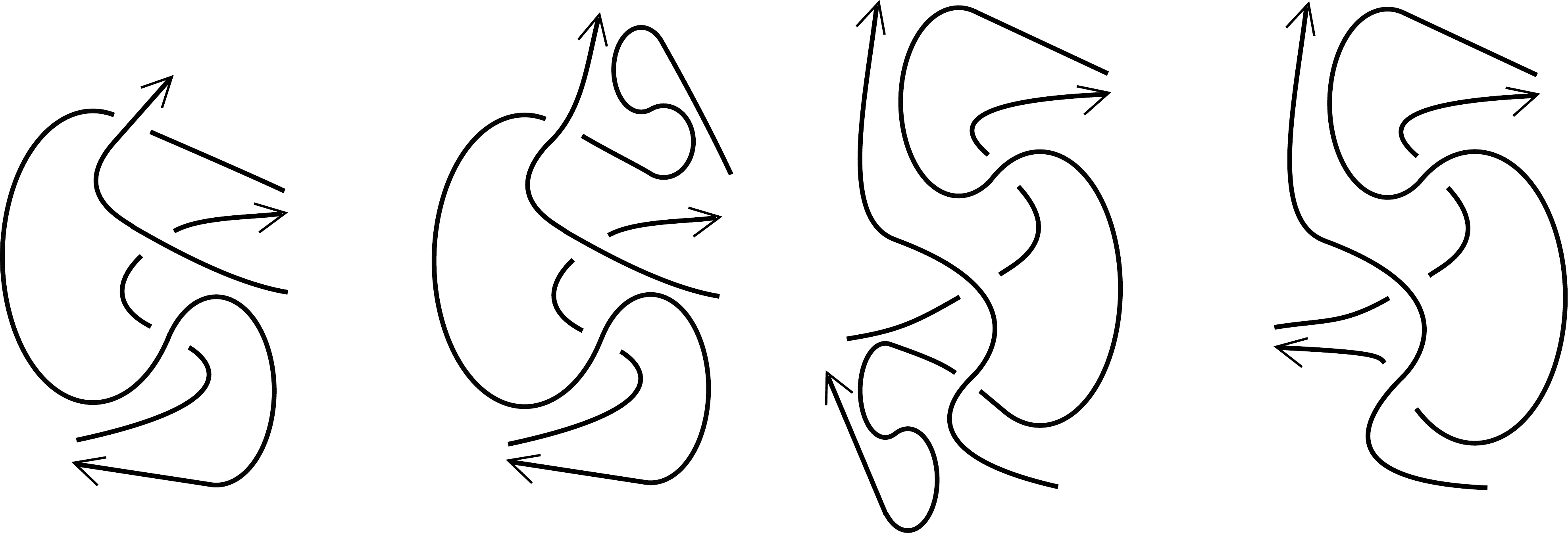}}.
\end{equation*}

We realise $\Omega 3a3$ from $\Omega 3a2$ and \eqref{eq:lemma1} via the moves
     \begin{equation*}
\centre{
\labellist \small \hair 2pt
\pinlabel{$\leftrightsquigarrow$}  at 460 330
\pinlabel{{\scriptsize \eqref{eq:lemma1}}}  at 465 400
\pinlabel{$\leftrightsquigarrow$}  at 980 330
\pinlabel{{\scriptsize $\Omega 3a2$}}  at 985 400
\pinlabel{$\leftrightsquigarrow$}  at 1500 330
\pinlabel{{\scriptsize \eqref{eq:lemma1}}}  at 1505 400
\endlabellist
\centering
\includegraphics[width=0.7\textwidth]{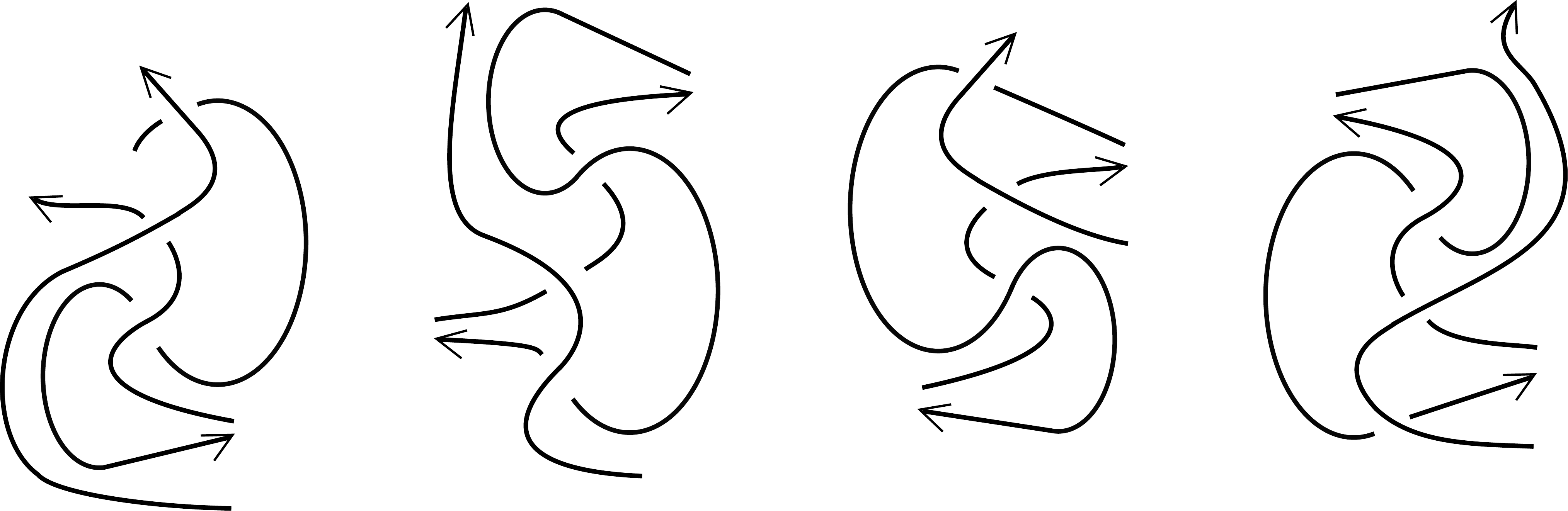}}.
\end{equation*}

We realise $\Omega 3a4$ from $\Omega 3a3$ and $\Omega 0a$ via the moves
     \begin{equation*}
\centre{
\labellist \small \hair 2pt
\pinlabel{$\leftrightsquigarrow$}  at 430 330
\pinlabel{{\scriptsize $\Omega 0a$}}  at 435 400
\pinlabel{$\leftrightsquigarrow$}  at 950 330
\pinlabel{{\scriptsize $\Omega 3a3$}}  at 955 400
\pinlabel{$\leftrightsquigarrow$}  at 1500 330
\pinlabel{{\scriptsize $\Omega 0a$}}  at 1505 400
\endlabellist
\centering
\includegraphics[width=0.7\textwidth]{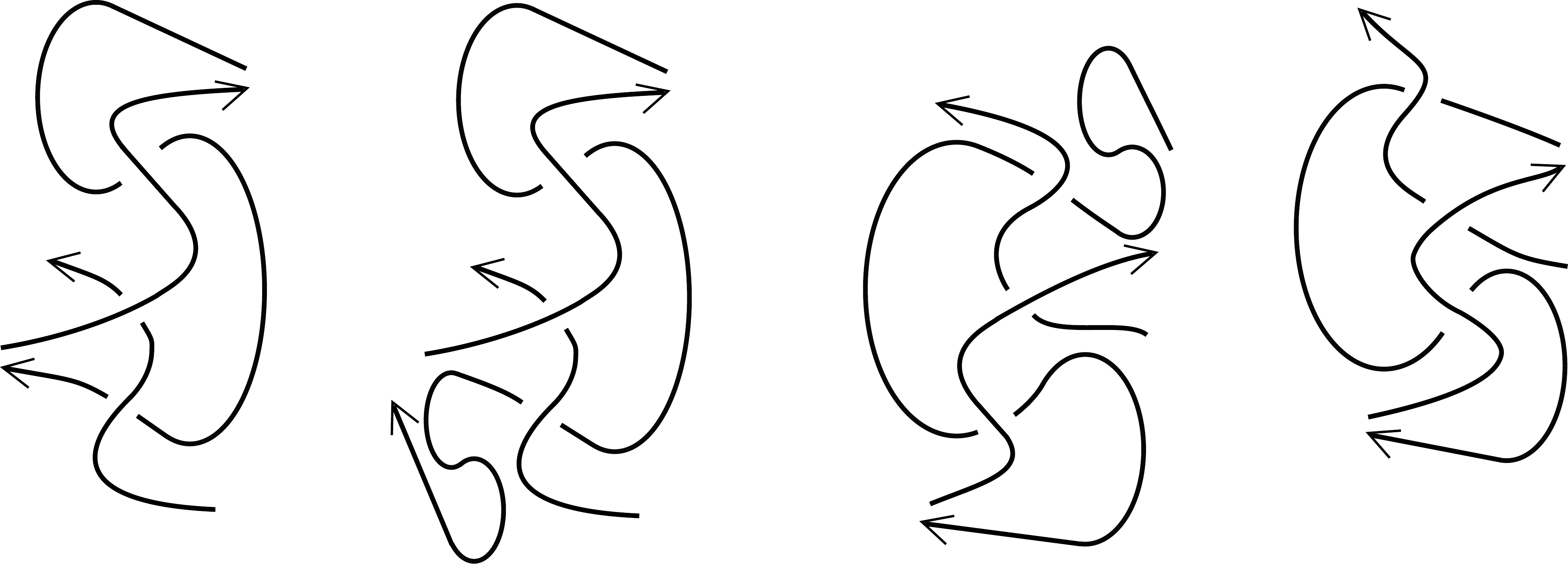}}.
\end{equation*}

We realise $\Omega 3a5$ from $\Omega 3a4$ and \eqref{eq:lemma1} via the moves
     \begin{equation*}
\centre{
\labellist \small \hair 2pt
\pinlabel{$\leftrightsquigarrow$}  at 460 330
\pinlabel{{\scriptsize \eqref{eq:lemma1}}}  at 465 400
\pinlabel{$\leftrightsquigarrow$}  at 980 330
\pinlabel{{\scriptsize $\Omega 3a4$}}  at 985 400
\pinlabel{$\leftrightsquigarrow$}  at 1500 330
\pinlabel{{\scriptsize \eqref{eq:lemma1}}}  at 1505 400
\endlabellist
\centering
\includegraphics[width=0.7\textwidth]{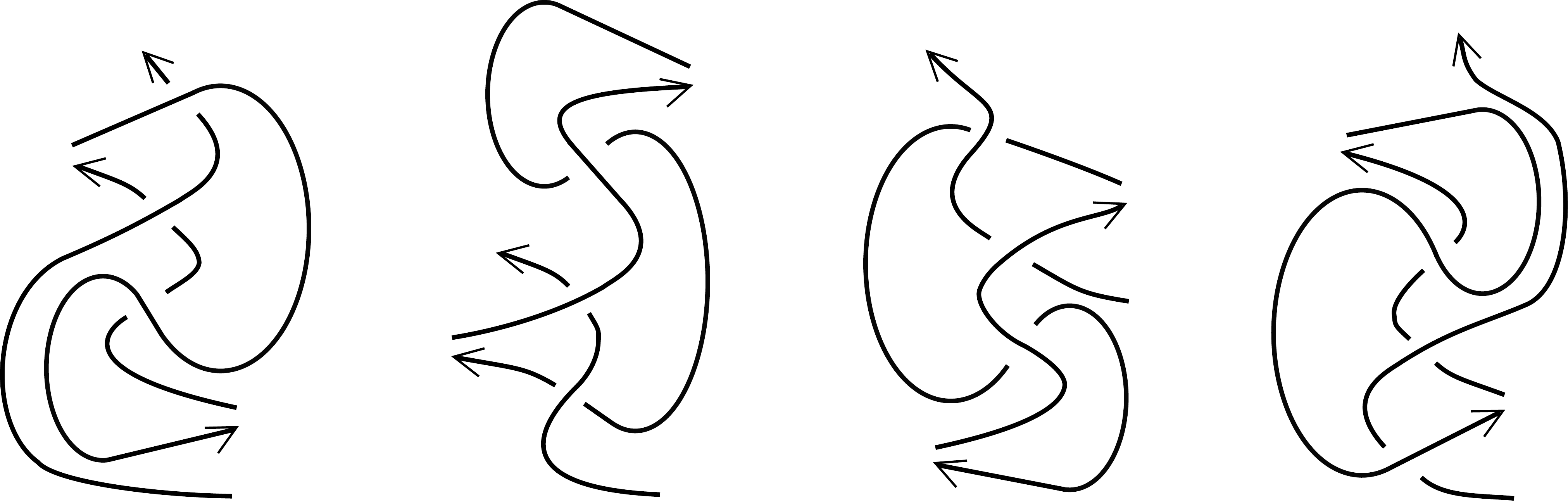}}.
\end{equation*}

We realise $\Omega 3a6$ from $\Omega 3a5$ and $\Omega 0a$ via the moves
     \begin{equation*}
\centre{
\labellist \small \hair 2pt
\pinlabel{$\leftrightsquigarrow$}  at 430 330
\pinlabel{{\scriptsize $\Omega 0a$}}  at 435 400
\pinlabel{$\leftrightsquigarrow$}  at 960 330
\pinlabel{{\scriptsize $\Omega 3a5$}}  at 965 400
\pinlabel{$\leftrightsquigarrow$}  at 1470 330
\pinlabel{{\scriptsize $\Omega 0a$}}  at 1475 400
\endlabellist
\centering
\includegraphics[width=0.7\textwidth]{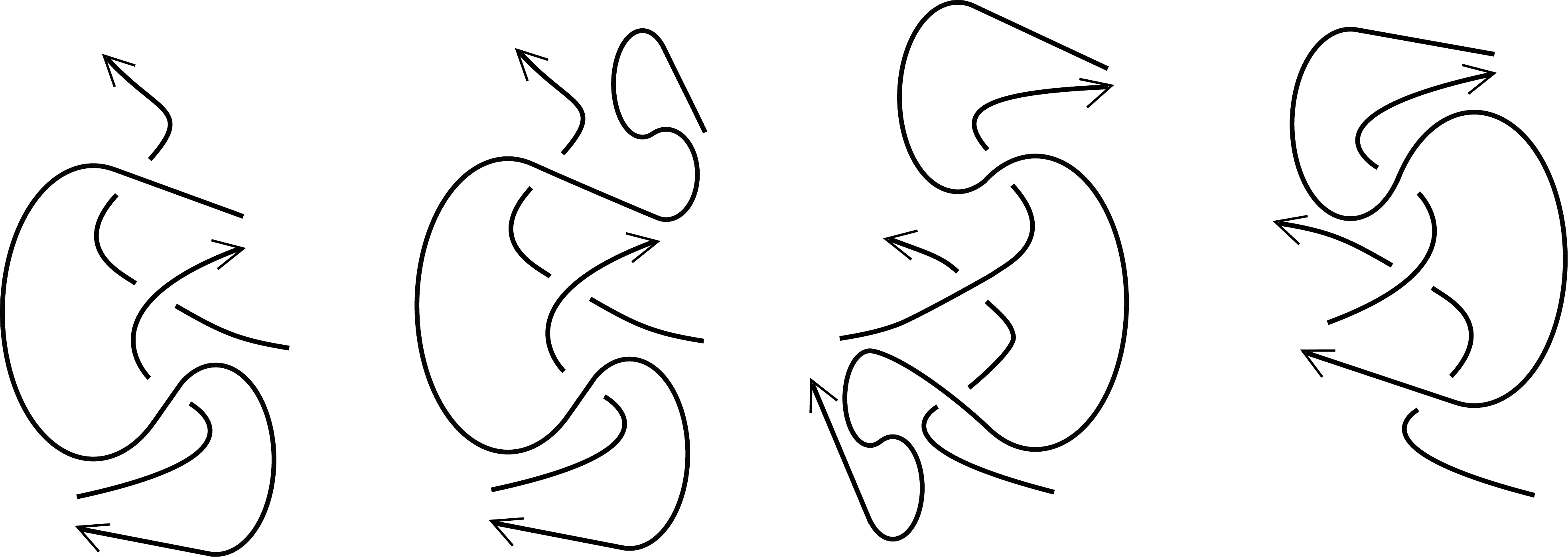}}.
\end{equation*}

\end{proof}

\begin{remark}\label{rem:O3a1_from_O3a6}
    Note that we can realize $\Omega 3a1$ from $\Omega 3a6$ and \eqref{eq:lemma1} via the moves
        \begin{equation*}
\centre{
\labellist \small \hair 2pt
\pinlabel{$\leftrightsquigarrow$}  at 460 330
\pinlabel{{\scriptsize \eqref{eq:lemma1}}}  at 465 400
\pinlabel{$\leftrightsquigarrow$}  at 950 330
\pinlabel{{\scriptsize $\Omega 3a6$}}  at 955 400
\pinlabel{$\leftrightsquigarrow$}  at 1500 330
\pinlabel{{\scriptsize \eqref{eq:lemma1}}}  at 1505 400
\endlabellist
\centering
\includegraphics[width=0.7\textwidth]{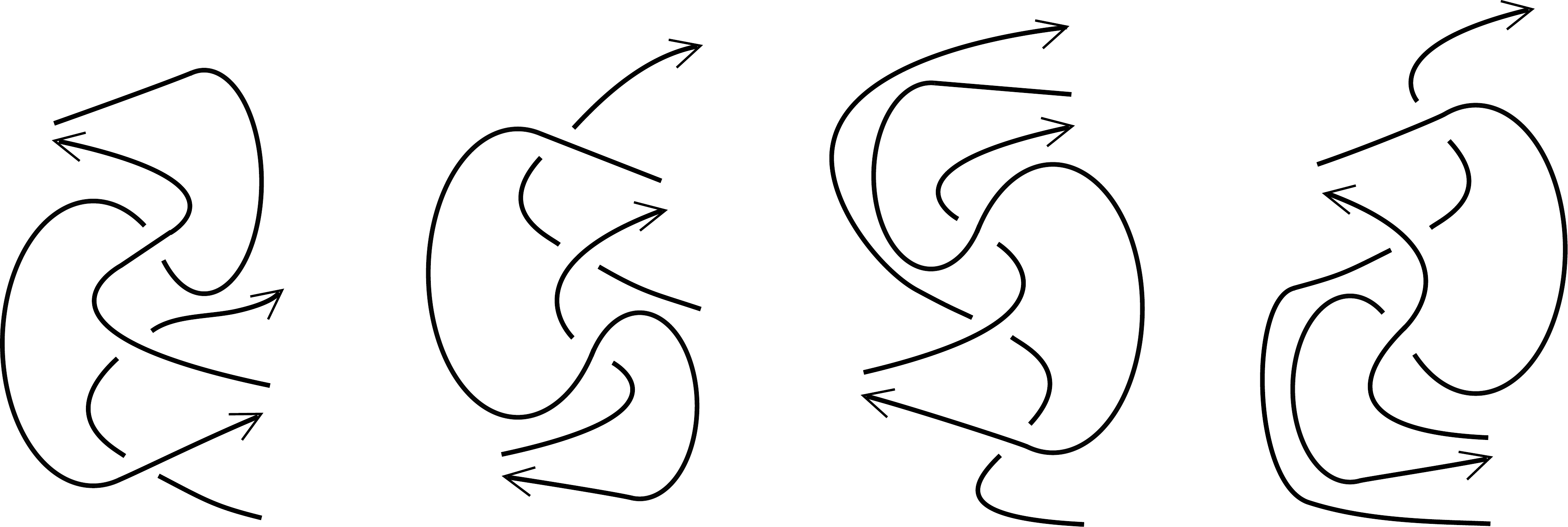}},
\end{equation*}
so given $\Omega 0a$ and \eqref{eq:lemma1}, we can realize all of the $\Omega 3ak$ ($1 \leq k \leq 6$) from just one of them.
\end{remark}

\begin{lemma}
The move $\Omega 2c1$ can be realised as a combination of the moves $\Omega 1a$, $\Omega 2a$ and $\Omega 3a3$. Similarly, the move $\Omega 2d1$ can be realised as a combination of the moves $\Omega 1b$, $\Omega 2a$ and $\Omega 3a6$.
\end{lemma}
\begin{proof}
On the one hand
\begin{equation*}
\centre{
\labellist \small \hair 2pt
\pinlabel{$\leftrightsquigarrow$}  at 380 300
\pinlabel{{\scriptsize $\Omega 1a$}}  at 385 370
\pinlabel{$=$}  at 920 300
\pinlabel{$\leftrightsquigarrow$}  at 1450 300
\pinlabel{{\scriptsize $\Omega 2a$}}  at 1455 370
\pinlabel{$\leftrightsquigarrow$}  at 1960 300
\pinlabel{{\scriptsize $\Omega 3a3$}}  at 1965 370
\pinlabel{$\leftrightsquigarrow$}  at 2450 300
\pinlabel{{\scriptsize $\Omega 1a$}}  at 2455 370
\endlabellist
\centering
\includegraphics[width=0.9\textwidth]{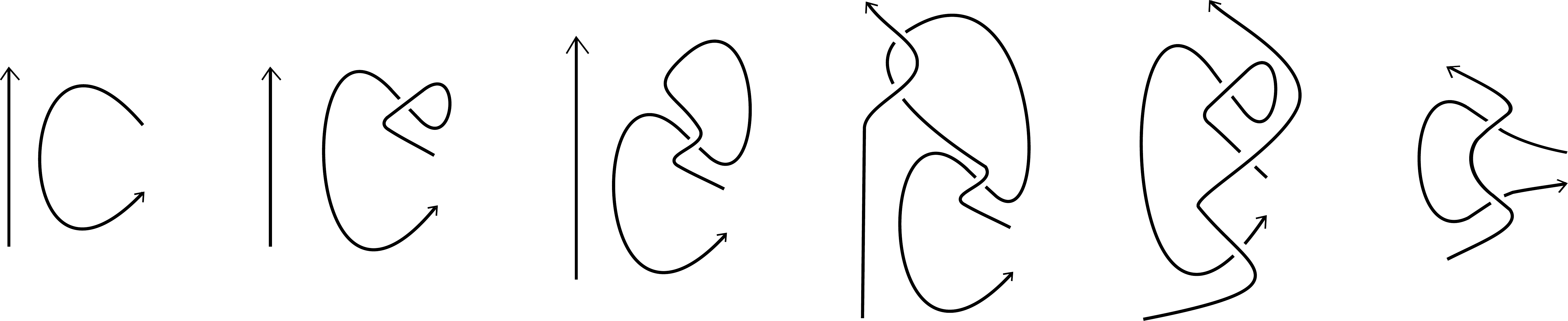}},
\end{equation*}
and on the other hand
\begin{equation*}
\centre{
\labellist \small \hair 2pt
\pinlabel{$\leftrightsquigarrow$}  at 380 300
\pinlabel{{\scriptsize $\Omega 1b$}}  at 385 370
\pinlabel{$=$}  at 920 300
\pinlabel{$\leftrightsquigarrow$}  at 1450 300
\pinlabel{{\scriptsize $\Omega 2a$}}  at 1455 370
\pinlabel{$\leftrightsquigarrow$}  at 1960 300
\pinlabel{{\scriptsize $\Omega 3a6$}}  at 1965 370
\pinlabel{$\leftrightsquigarrow$}  at 2450 300
\pinlabel{{\scriptsize $\Omega 1b$}}  at 2455 370
\endlabellist
\centering
\includegraphics[width=0.9\textwidth]{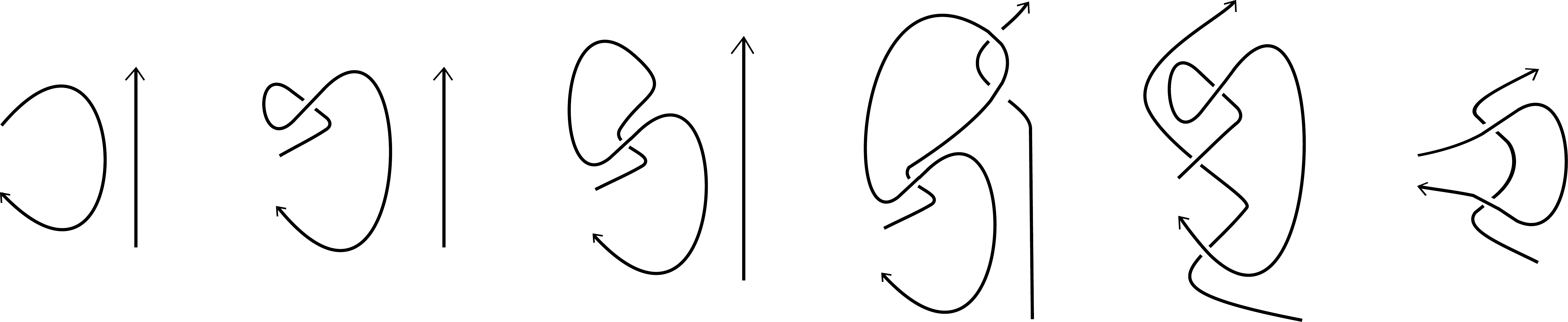}}.
\end{equation*}
\end{proof}

\begin{lemma}\label{lem:O1c_O1d}
The move $\Omega 1c$ can be obtained by a combination of $\Omega 0a$, $\Omega 2d1$ and $\Omega 1b$. Similarly, the move $\Omega 1d$ can be obtained by a combination of $\Omega 0b$, $\Omega 2c1$ and $\Omega 1a$.
\end{lemma}
\begin{proof}
We have
\begin{equation*}
\centre{
\labellist \small \hair 2pt
\pinlabel{$\leftrightsquigarrow$}  at 82 185
\pinlabel{{\scriptsize $\Omega 0a$}}  at 87 255
\pinlabel{$\leftrightsquigarrow$}  at 510 185
\pinlabel{{\scriptsize $\Omega 2d1$}}  at 515 255
\pinlabel{$\leftrightsquigarrow$}  at 880 185
\pinlabel{{\scriptsize $\Omega 1b$}}  at 885 255
\endlabellist
\centering
\includegraphics[width=0.4\textwidth]{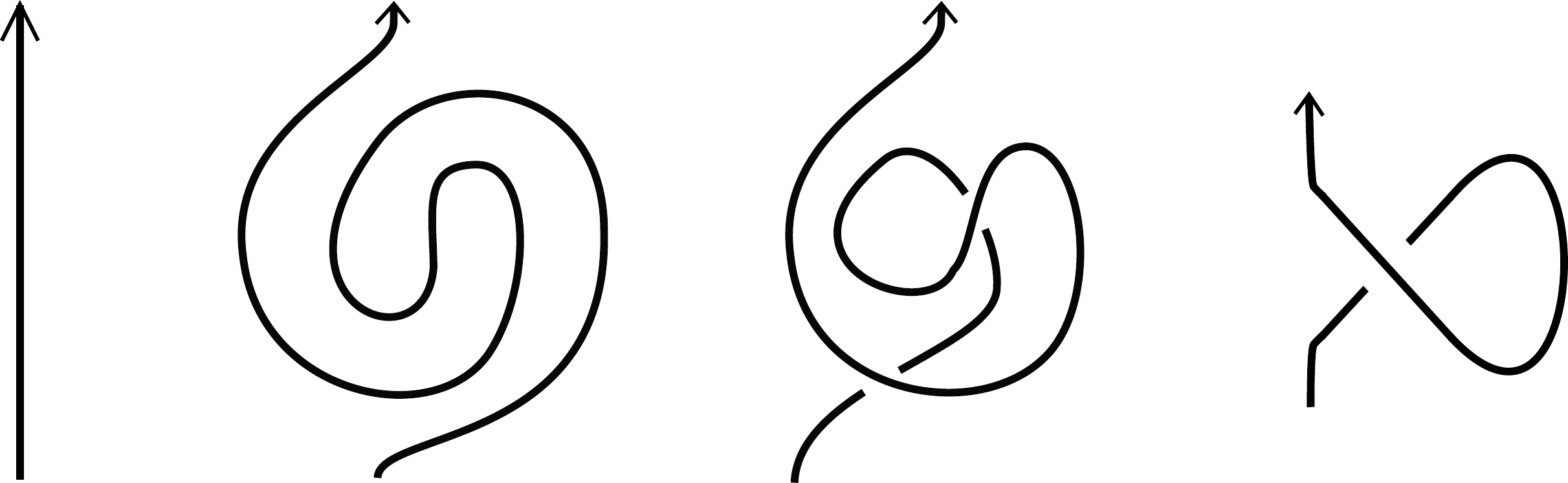}}
\end{equation*}
and likewise
\begin{equation*}
\centre{
\labellist \small \hair 2pt
\pinlabel{$\leftrightsquigarrow$}  at 82 185
\pinlabel{{\scriptsize $\Omega 0b$}}  at 87 255
\pinlabel{$\leftrightsquigarrow$}  at 510 185
\pinlabel{{\scriptsize $\Omega 2c1$}}  at 515 255
\pinlabel{$\leftrightsquigarrow$}  at 880 185
\pinlabel{{\scriptsize $\Omega 1a$}}  at 885 255
\endlabellist
\centering
\includegraphics[width=0.4\textwidth]{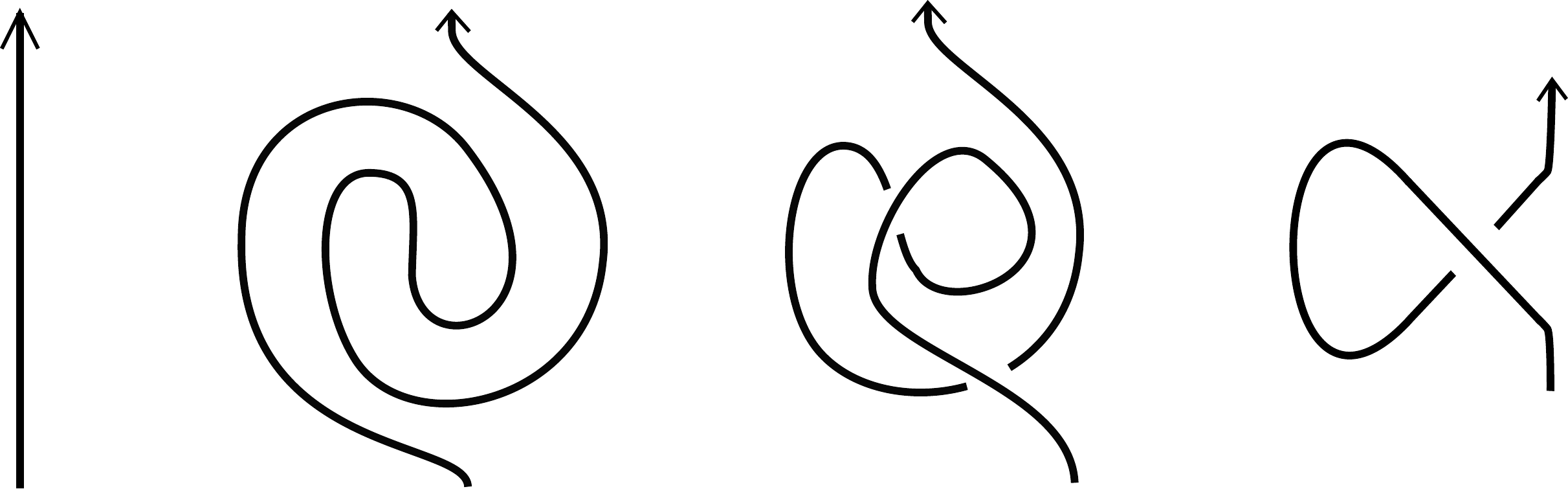}}.
\end{equation*}
\end{proof}

\begin{lemma}\label{lem:O2c2_O2d2}
The move $\Omega 2c2$ can be realised as a combination of the moves $\Omega 0a$, $\Omega 2c1$ and  \eqref{eq:lemma2}. Similarly, the move $\Omega 2d2$ can be realised as a combination of the moves $\Omega 0b$, $\Omega 2d1$ and  \eqref{eq:lemma2}.
\end{lemma}
\begin{proof}
    We have
    \begin{equation*}
\centre{
\labellist \small \hair 2pt
\pinlabel{$\leftrightsquigarrow$}  at 350 300
\pinlabel{{\scriptsize $\Omega 0a$}}  at 355 370
\pinlabel{$\leftrightsquigarrow$}  at 1060 300
\pinlabel{{\scriptsize $\Omega 2c1$}}  at 1065 370
\pinlabel{$\leftrightsquigarrow$}  at 1760 300
\pinlabel{{\scriptsize \eqref{eq:lemma2}}}  at 1765 370
\pinlabel{$\leftrightsquigarrow$}  at 2360 300
\pinlabel{{\scriptsize $\Omega 0a$}}  at 2365 370
\endlabellist
\centering
\includegraphics[width=0.9\textwidth]{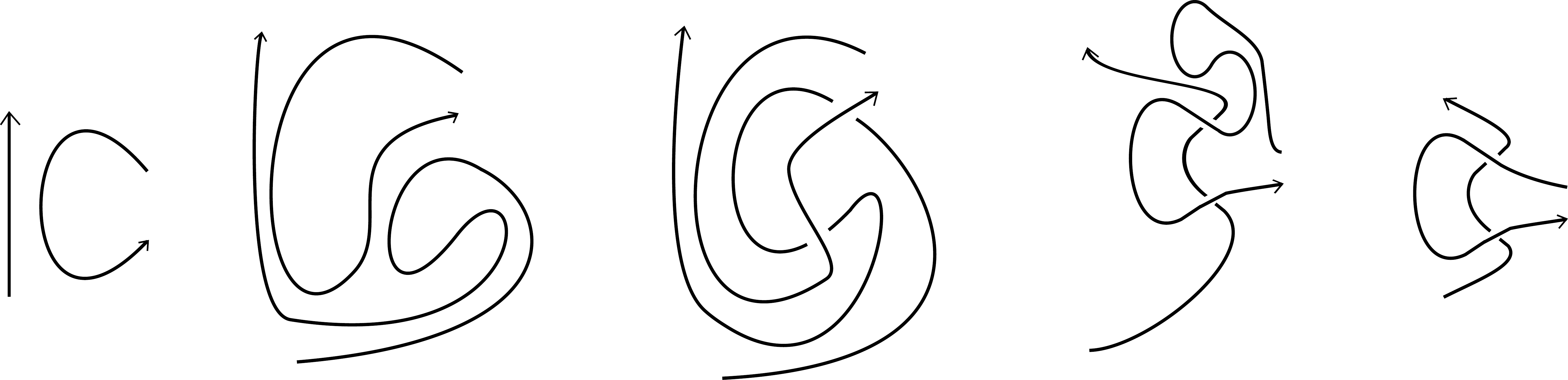}}
\end{equation*}
and likewise
   \begin{equation*}
\centre{
\labellist \small \hair 2pt
\pinlabel{$\leftrightsquigarrow$}  at 350 300
\pinlabel{{\scriptsize $\Omega 0b$}}  at 355 370
\pinlabel{$\leftrightsquigarrow$}  at 1060 300
\pinlabel{{\scriptsize $\Omega 2d1$}}  at 1065 370
\pinlabel{$\leftrightsquigarrow$}  at 1760 300
\pinlabel{{\scriptsize \eqref{eq:lemma2}}}  at 1765 370
\pinlabel{$\leftrightsquigarrow$}  at 2360 300
\pinlabel{{\scriptsize $\Omega 0b$}}  at 2365 370
\endlabellist
\centering
\includegraphics[width=0.9\textwidth]{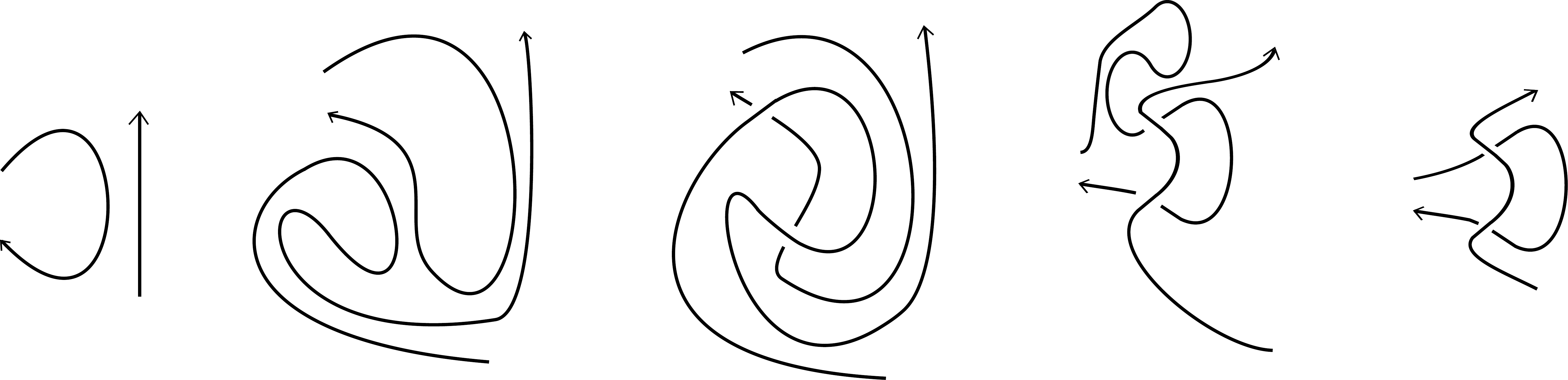}}.
\end{equation*}
\end{proof}

As in \cite{polyak10}, we will take care of the $\Omega 2b$ move later. Before we will obtain two braid-like Reidemeister 3 moves:

\begin{lemma}\label{lem:O3b}
The move $\Omega 3b$ can be realised by a sequence of the moves $\Omega 0a$, $\Omega 2c1$, $\Omega 3a4$ and $\Omega 2d2$.
\end{lemma}
\begin{proof}
We have
\begin{equation*}
\centre{
\labellist \small \hair 2pt
\pinlabel{$\leftrightsquigarrow$}  at 330 370
\pinlabel{{\scriptsize $\Omega 0a$}}  at 335 430
\pinlabel{$\leftrightsquigarrow$}  at 980 370
\pinlabel{{\scriptsize $\Omega 2c1$}}  at 985 430
\pinlabel{$\leftrightsquigarrow$}  at 1610 370
\pinlabel{{\scriptsize $\Omega 3a4$}}  at 1615 430
\endlabellist
\centering
\includegraphics[width=0.7\textwidth]{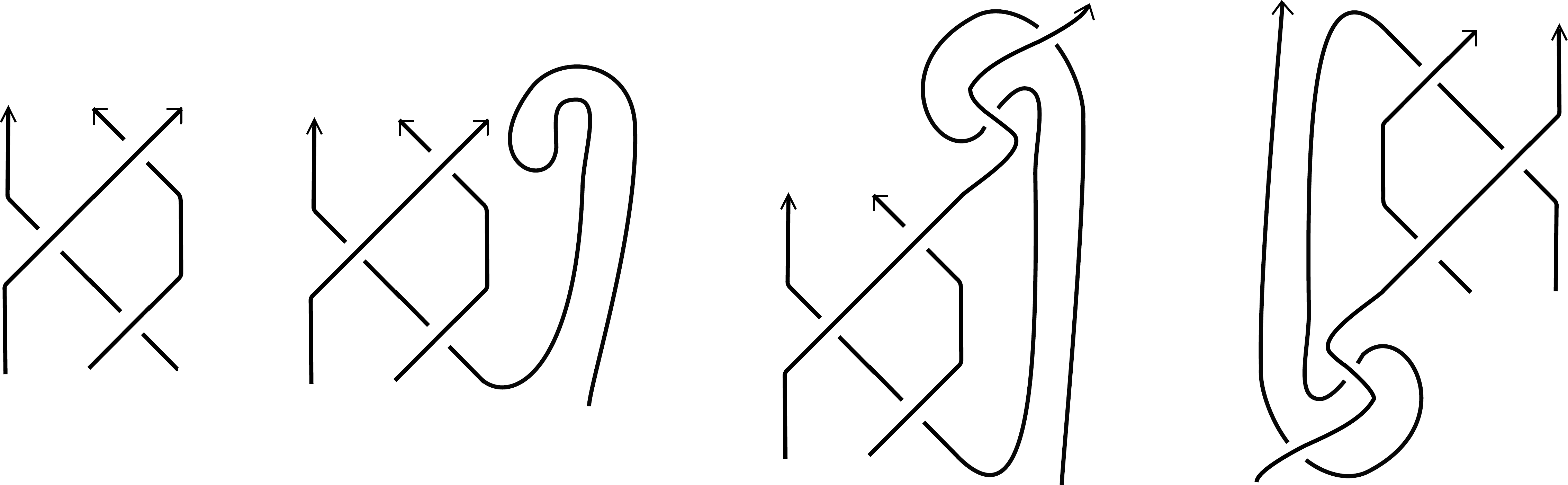}}
\end{equation*}
\begin{equation*}
\centre{
\labellist \small \hair 2pt
\pinlabel{$\leftrightsquigarrow$}  at -60 225
\pinlabel{{\scriptsize $\Omega 2d2$}}  at -55 295
\pinlabel{$\leftrightsquigarrow$}  at 580 225
\pinlabel{{\scriptsize $\Omega 0a$}}  at 585 295
\endlabellist
\centering
\includegraphics[width=0.3\textwidth]{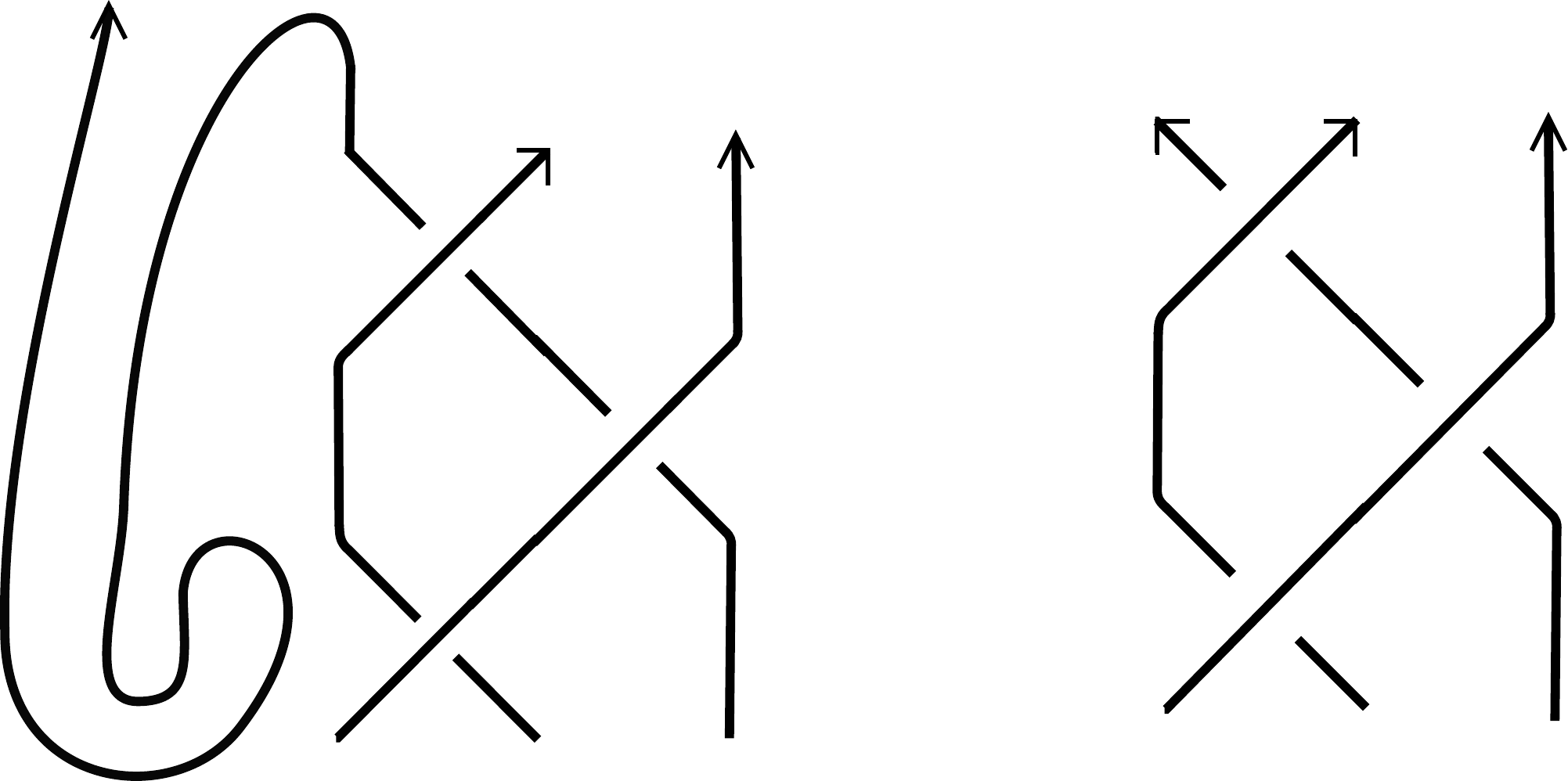}}. \phantom{-----}
\end{equation*}
\end{proof}

In order to realise the $\Omega 2b$ move we need to obtain the $\Omega 3c$ move:

\begin{lemma}\label{lem:O3c}
The move $\Omega 3c$ can be obtained by a sequence of the moves $\Omega 0a$, $\Omega 2c1$, $\Omega 3a6$ and $\Omega 2d1$.
\end{lemma}
\begin{proof}
We have
\begin{equation*}
\centre{
\labellist \small \hair 2pt
\pinlabel{$\leftrightsquigarrow$}  at 330 370
\pinlabel{{\scriptsize $\Omega 0a$}}  at 335 430
\pinlabel{$\leftrightsquigarrow$}  at 980 370
\pinlabel{{\scriptsize $\Omega 2c2$}}  at 985 430
\pinlabel{$\leftrightsquigarrow$}  at 1610 370
\pinlabel{{\scriptsize $\Omega 3a6$}}  at 1615 430
\endlabellist
\centering
\includegraphics[width=0.7\textwidth]{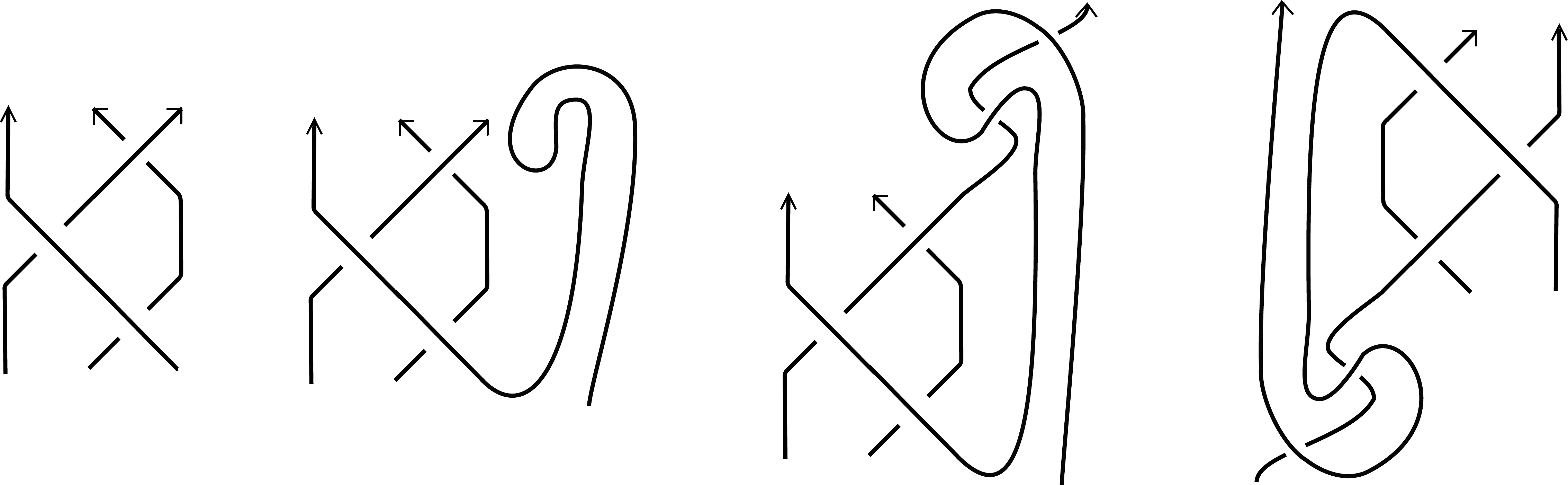}}
\end{equation*}
\begin{equation*}
\centre{
\labellist \small \hair 2pt
\pinlabel{$\leftrightsquigarrow$}  at -60 225
\pinlabel{{\scriptsize $\Omega 2d1$}}  at -55 295
\pinlabel{$\leftrightsquigarrow$}  at 580 225
\pinlabel{{\scriptsize $\Omega 0a$}}  at 585 295
\endlabellist
\centering
\includegraphics[width=0.3\textwidth]{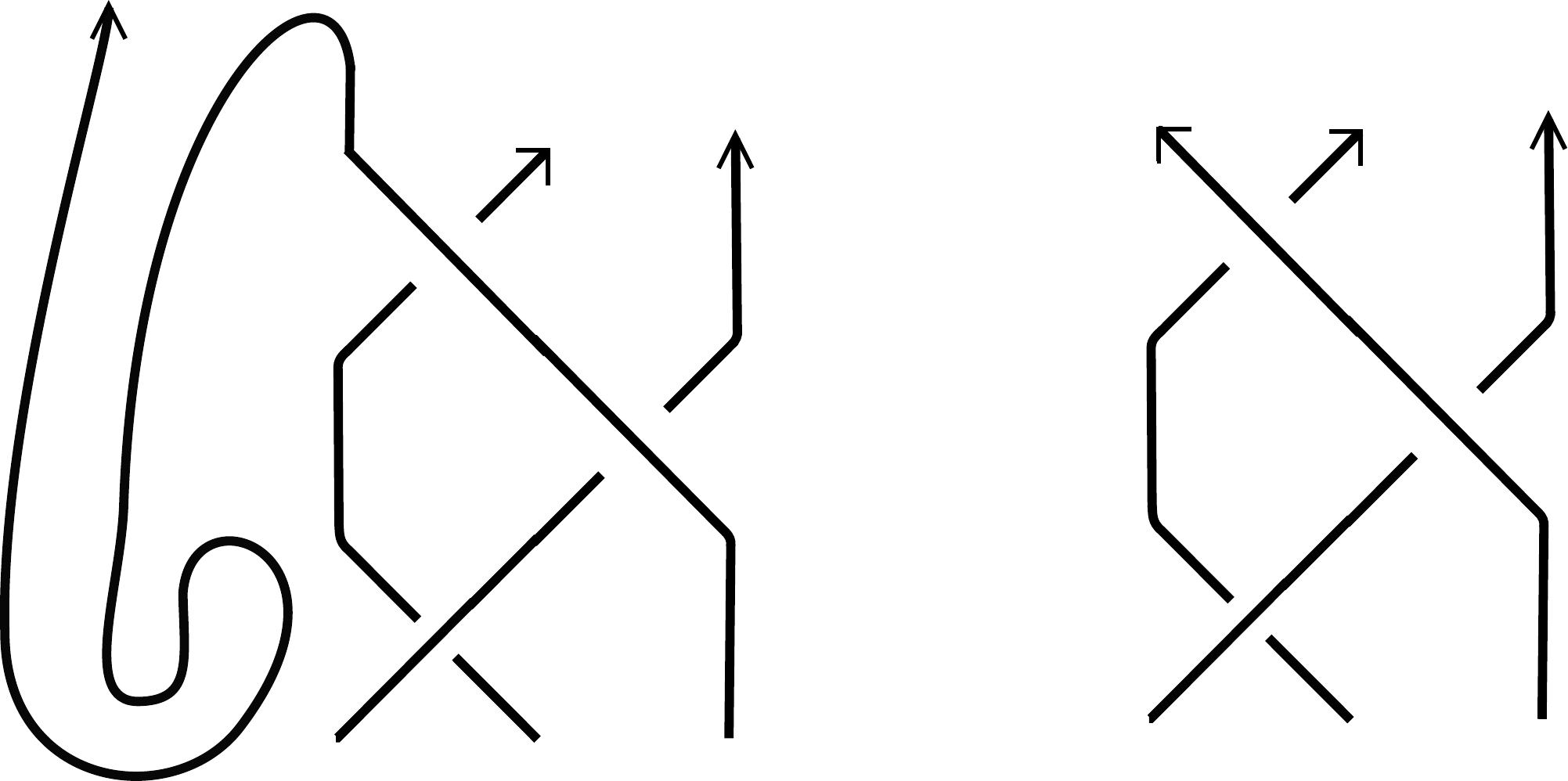}}. \phantom{-----}
\end{equation*}
\end{proof}

We can finally obtain the last Reidemeister 2 move:

\begin{lemma}\label{lem:O2b}
The move $\Omega 2b$ can be realised as a sequence of the moves $\Omega 1d$, $\Omega 2c2$ and $\Omega 3c$.
\end{lemma}
\begin{proof}
We have
\begin{equation*}
\centre{
\labellist \small \hair 2pt
\pinlabel{$\leftrightsquigarrow$}  at 330 200
\pinlabel{{\scriptsize $\Omega 1d$}}  at 335 270
\pinlabel{$\leftrightsquigarrow$}  at 1000 200
\pinlabel{{\scriptsize $\Omega 2c2$}}  at 1005 270
\pinlabel{$\leftrightsquigarrow$}  at 1550 200
\pinlabel{{\scriptsize $\Omega 3c$}}  at 1555 270
\pinlabel{$\leftrightsquigarrow$}  at 2140 200
\pinlabel{{\scriptsize $\Omega 1d$}}  at 2145 270
\endlabellist
\centering
\includegraphics[width=0.8\textwidth]{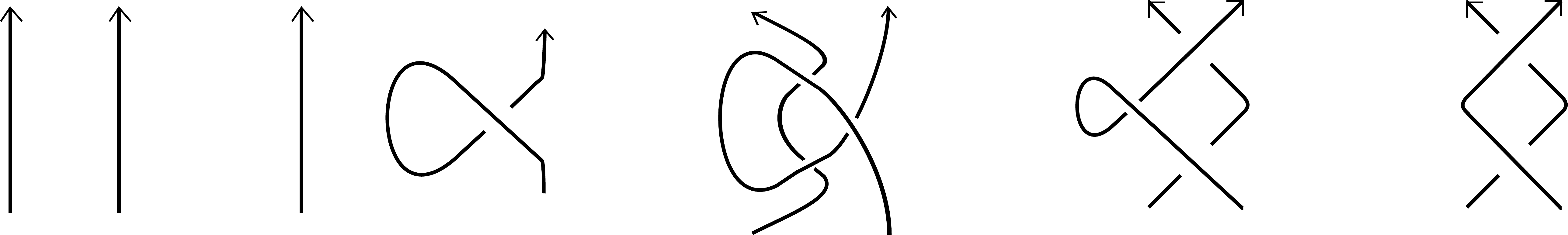}}
\end{equation*}
\end{proof}

It only remains to obtain the moves $\Omega 3d$ -- $\Omega 3g$ and $\Omega 3h1$ --$\Omega 3h6$. For the sake of clarity we treat these two groups separately.

\begin{lemma}\label{lem:O3d_g}
    The braid-like moves $\Omega 3d$ -- $\Omega 3g$ can be realised by a sequence of the moves $\Omega 3b$, $\Omega 2a$ and $\Omega 2b$.
\end{lemma}
\begin{proof}
Most sequences are analogous to the ones in \cite{polyak10}, but adapted to the rotational setting. We realise $\Omega 3d$ from $\Omega 3b$, $\Omega 2a$ and $\Omega 2b$ via the moves
\begin{equation*}
\centre{
\labellist \small \hair 2pt
\pinlabel{$\leftrightsquigarrow$}  at 500 220
\pinlabel{{\scriptsize $\Omega 2a$}}  at 505 270
\pinlabel{$\leftrightsquigarrow$}  at 1080 220
\pinlabel{{\scriptsize $\Omega 3b$}}  at 1085 270
\pinlabel{$\leftrightsquigarrow$}  at 1630 220
\pinlabel{{\scriptsize $\Omega 2b$}}  at 1635 270
\endlabellist
\centering
\includegraphics[width=0.7\textwidth]{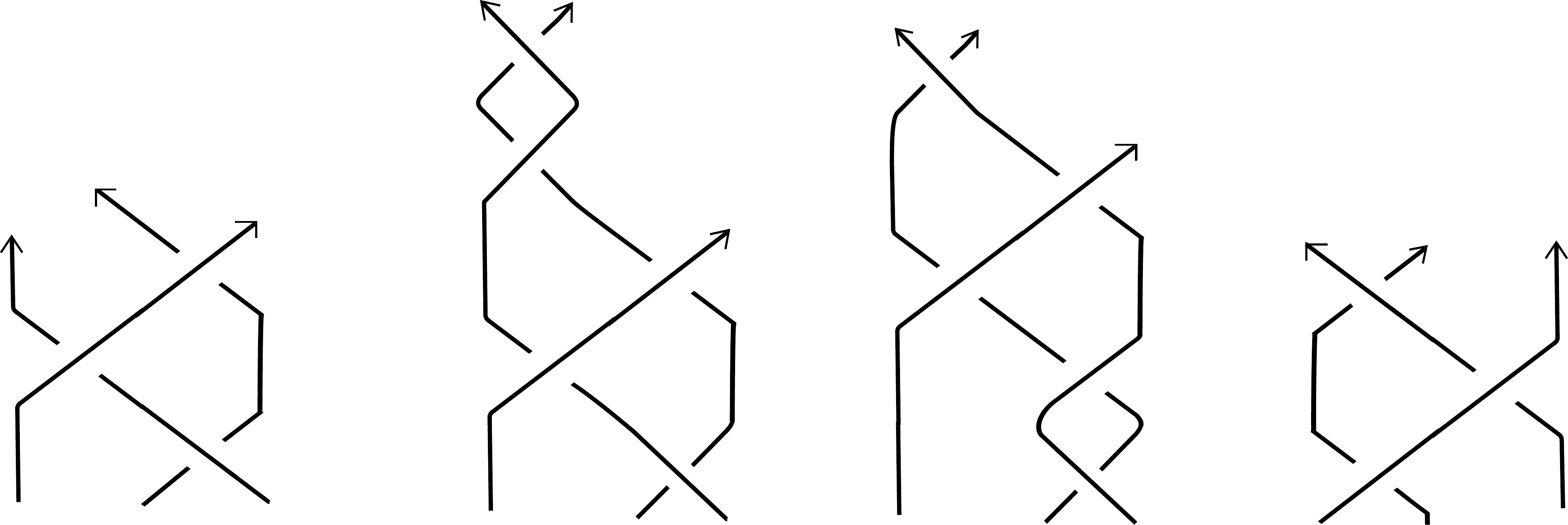}}.
\end{equation*}

We realise $\Omega 3e$ from $\Omega 3b$, $\Omega 2a$ and $\Omega 2b$ via the moves
\begin{equation*}
\centre{
\labellist \small \hair 2pt
\pinlabel{$\leftrightsquigarrow$}  at 500 220
\pinlabel{{\scriptsize $\Omega 2a$}}  at 505 270
\pinlabel{$\leftrightsquigarrow$}  at 1080 220
\pinlabel{{\scriptsize $\Omega 3b$}}  at 1085 270
\pinlabel{$\leftrightsquigarrow$}  at 1630 220
\pinlabel{{\scriptsize $\Omega 2b$}}  at 1635 270
\endlabellist
\centering
\includegraphics[width=0.7\textwidth]{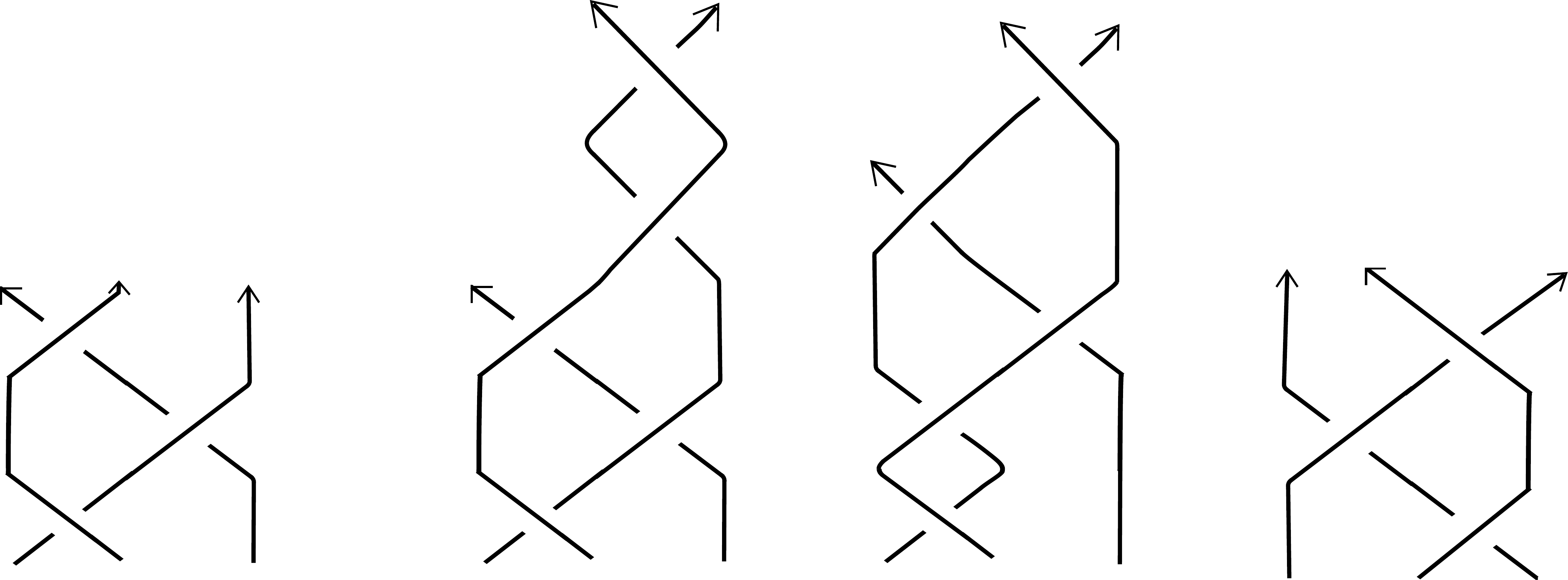}}.
\end{equation*}

We realise $\Omega 3f$ from $\Omega 3d$, $\Omega 2a$ and $\Omega 2b$ via the moves
\begin{equation*}
\centre{
\labellist \small \hair 2pt
\pinlabel{$\leftrightsquigarrow$}  at 500 220
\pinlabel{{\scriptsize $\Omega 2a$}}  at 505 270
\pinlabel{$\leftrightsquigarrow$}  at 1080 220
\pinlabel{{\scriptsize $\Omega 3d$}}  at 1085 270
\pinlabel{$\leftrightsquigarrow$}  at 1630 220
\pinlabel{{\scriptsize $\Omega 2b$}}  at 1635 270
\endlabellist
\centering
\includegraphics[width=0.7\textwidth]{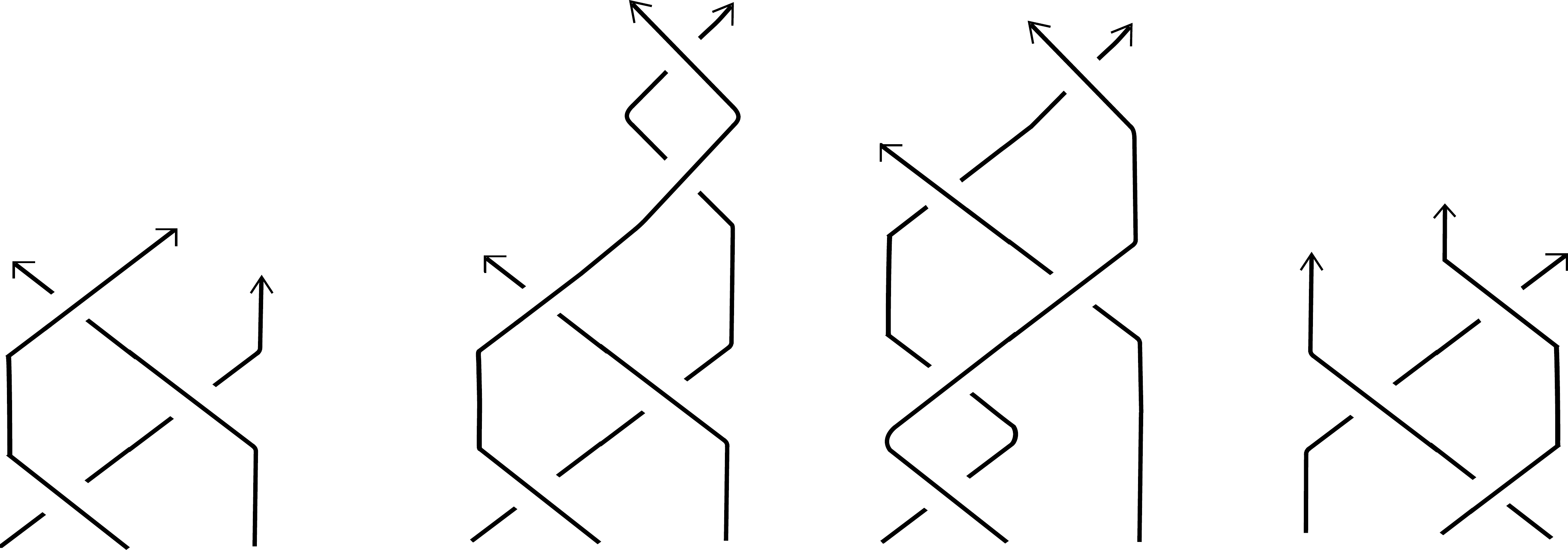}}.
\end{equation*}

We realise $\Omega 3g$ from $\Omega 3f$, $\Omega 2a$ and $\Omega 2b$ via the moves
\begin{equation*}
\centre{
\labellist \small \hair 2pt
\pinlabel{$\leftrightsquigarrow$}  at 500 220
\pinlabel{{\scriptsize $\Omega 2a$}}  at 505 270
\pinlabel{$\leftrightsquigarrow$}  at 1080 220
\pinlabel{{\scriptsize $\Omega 3f$}}  at 1085 270
\pinlabel{$\leftrightsquigarrow$}  at 1630 220
\pinlabel{{\scriptsize $\Omega 2b$}}  at 1635 270
\endlabellist
\centering
\includegraphics[width=0.7\textwidth]{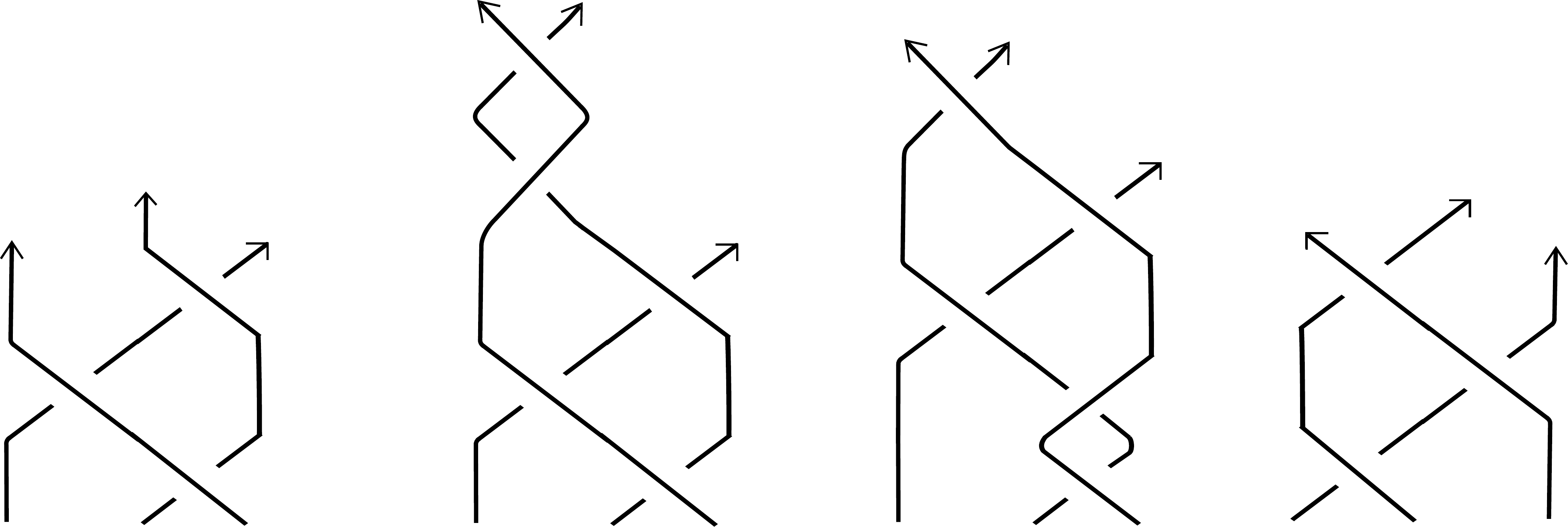}}.
\end{equation*}

\end{proof}

\begin{lemma}\label{lem:O3h1}
    The move $\Omega 3h1$ can be obtained  by a sequence of the moves $\Omega 0a$, $\Omega 2c2$, $\Omega 3g$, $\Omega 2d1$, $\Omega 0b$ and \eqref{eq:lemma1}.
\end{lemma}
\begin{proof}
    We have
\begin{equation*}
\centre{
\labellist \small \hair 2pt
\pinlabel{$\leftrightsquigarrow$}  at 420 180
\pinlabel{{\scriptsize $\Omega 0a$}}  at 425 250
\pinlabel{$\leftrightsquigarrow$}  at 1000 180
\pinlabel{{\scriptsize $\Omega 2c2$}}  at 1005 250
\pinlabel{$\leftrightsquigarrow$}  at 1600 180
\pinlabel{{\scriptsize $\Omega 3g$}}  at 1605 250
\endlabellist
\centering
\includegraphics[width=0.75\textwidth]{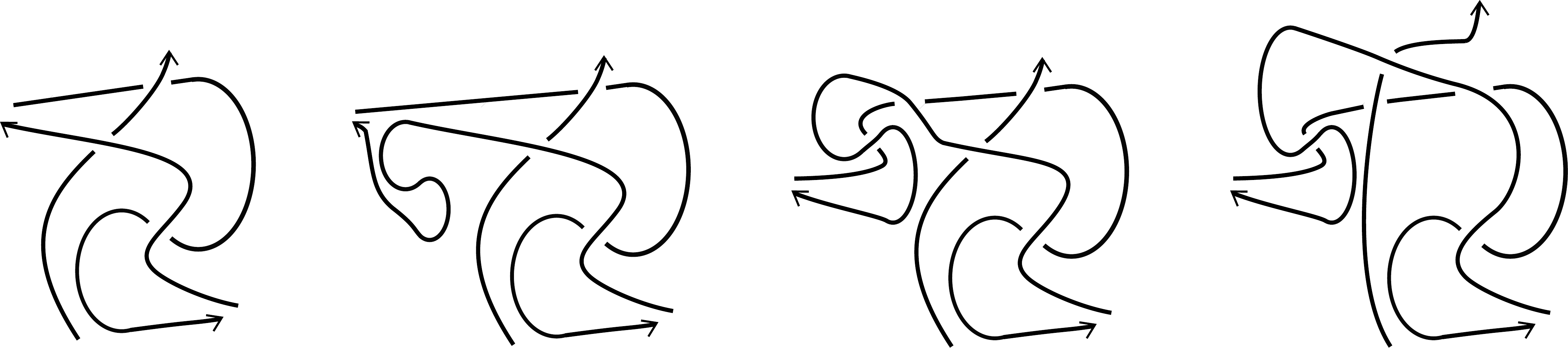}}
\end{equation*}
\begin{equation*}\phantom{------------}
\centre{
\labellist \small \hair 2pt
\pinlabel{$\leftrightsquigarrow$}  at -70 260
\pinlabel{{\scriptsize $\Omega 2d1$}}  at -65 330
\pinlabel{$\leftrightsquigarrow$}  at 570 260
\pinlabel{{\scriptsize $\Omega 0b$}}  at 575 330
\pinlabel{$\leftrightsquigarrow$}  at 1130 260
\pinlabel{{\scriptsize \eqref{eq:lemma1}}}  at 1135 330
\endlabellist
\centering
\includegraphics[width=0.6\textwidth]{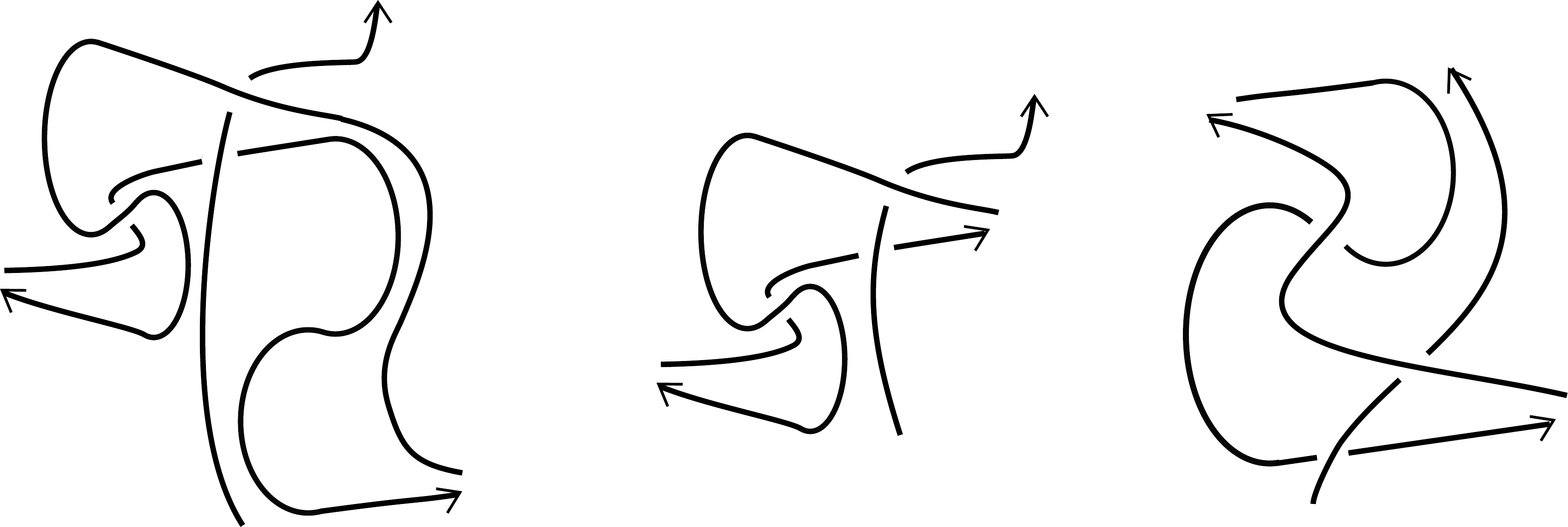}}.
\end{equation*}
\end{proof}


\begin{lemma}\label{lem:O3h}
Each of the moves  $\Omega 3h2$ -- $\Omega 3h6$ can be obtained from $\Omega 3h1$, $\Omega 0b$ and the move in \eqref{eq:lemma2}.
\end{lemma}
\begin{proof}
We realize $\Omega 3h6$ from $\Omega 3h1$ and \eqref{eq:lemma2} via the moves
        \begin{equation*}
\centre{
\labellist \small \hair 2pt
\pinlabel{$\leftrightsquigarrow$}  at 460 330
\pinlabel{{\scriptsize \eqref{eq:lemma2}}}  at 465 400
\pinlabel{$\leftrightsquigarrow$}  at 980 330
\pinlabel{{\scriptsize $\Omega 3h1$}}  at 985 400
\pinlabel{$\leftrightsquigarrow$}  at 1500 330
\pinlabel{{\scriptsize \eqref{eq:lemma2}}}  at 1505 400
\endlabellist
\centering
\includegraphics[width=0.7\textwidth]{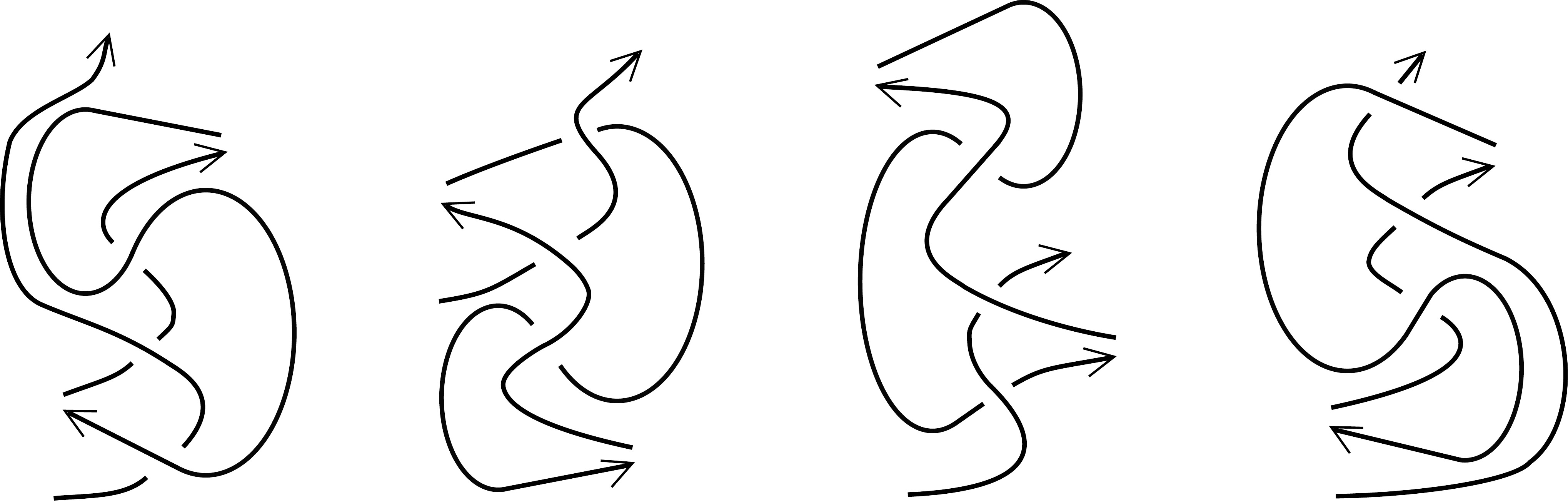}}.
\end{equation*}

We realise $\Omega 3h5$ from $\Omega 3h6$ and $\Omega 0b$ via the moves
     \begin{equation*}
\centre{
\labellist \small \hair 2pt
\pinlabel{$\leftrightsquigarrow$}  at 460 330
\pinlabel{{\scriptsize $\Omega 0b$}}  at 465 400
\pinlabel{$\leftrightsquigarrow$}  at 980 330
\pinlabel{{\scriptsize $\Omega 3h6$}}  at 985 400
\pinlabel{$\leftrightsquigarrow$}  at 1500 330
\pinlabel{{\scriptsize $\Omega 0b$}}  at 1505 400
\endlabellist
\centering
\includegraphics[width=0.7\textwidth]{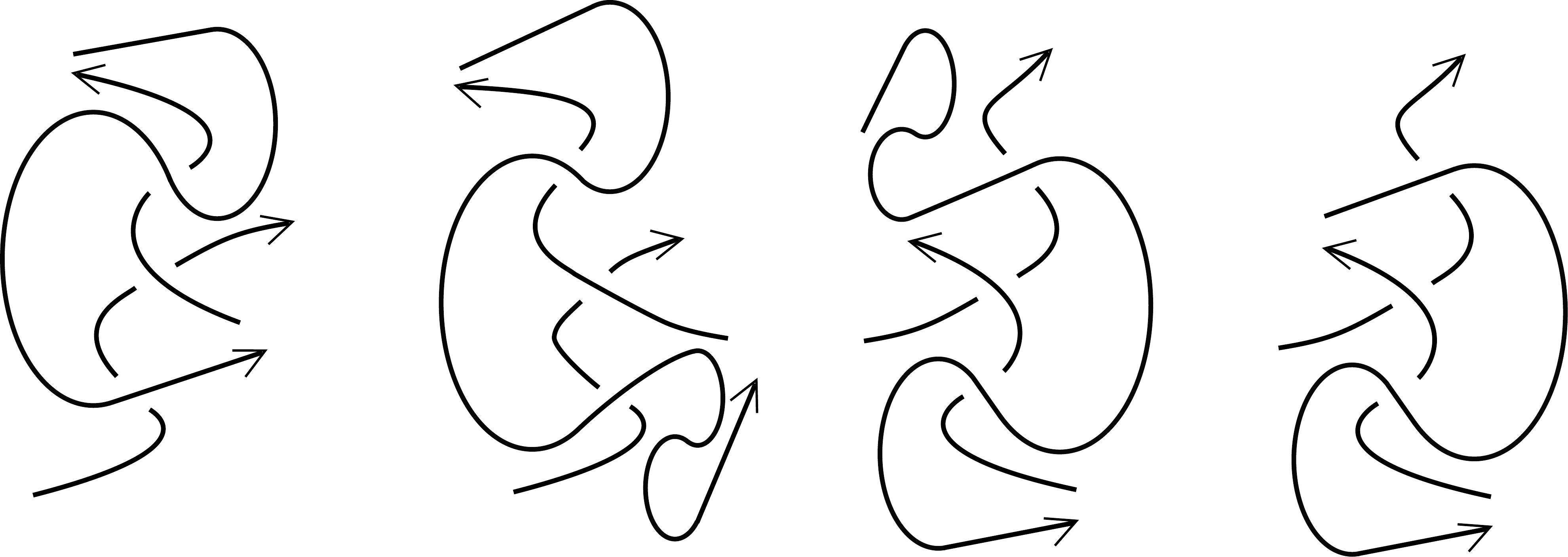}}.
\end{equation*}

We realise $\Omega 3h4$ from $\Omega 3h5$ and \eqref{eq:lemma2} via the moves
     \begin{equation*}
\centre{
\labellist \small \hair 2pt
\pinlabel{$\leftrightsquigarrow$}  at 460 330
\pinlabel{{\scriptsize \eqref{eq:lemma2}}}  at 465 400
\pinlabel{$\leftrightsquigarrow$}  at 980 330
\pinlabel{{\scriptsize $\Omega 3h5$}}  at 985 400
\pinlabel{$\leftrightsquigarrow$}  at 1500 330
\pinlabel{{\scriptsize \eqref{eq:lemma2}}}  at 1505 400
\endlabellist
\centering
\includegraphics[width=0.7\textwidth]{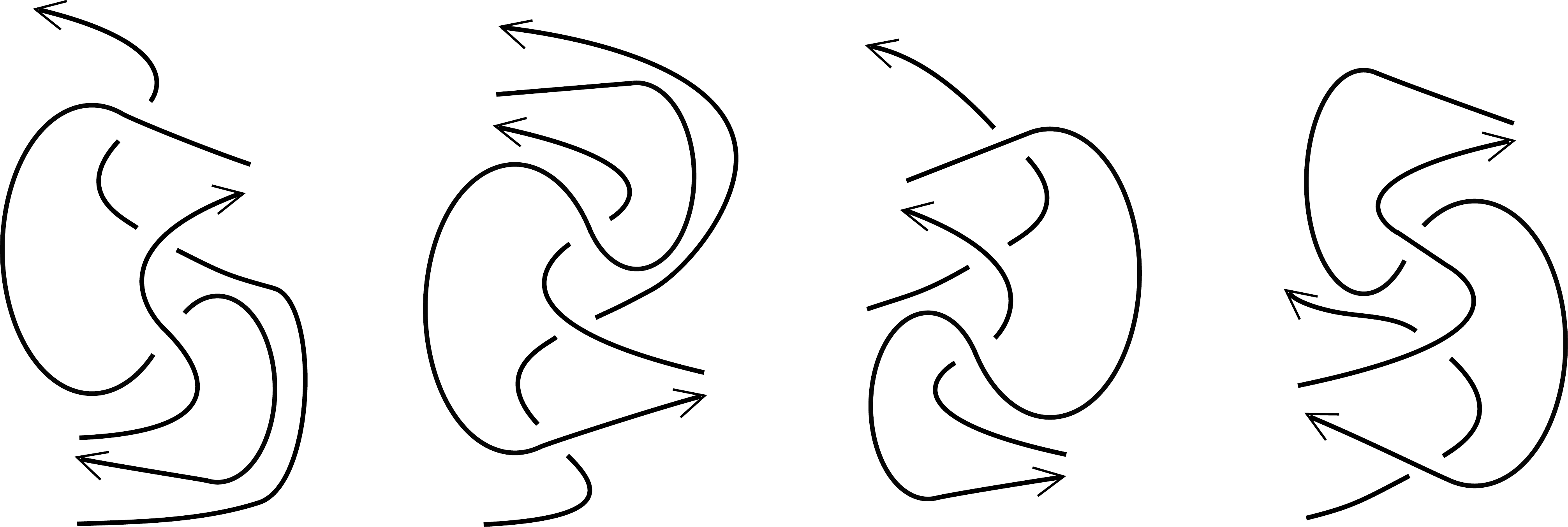}}.
\end{equation*}

We realise $\Omega 3h3$ from $\Omega 3h4$ and $\Omega 0b$ via the moves
     \begin{equation*}
\centre{
\labellist \small \hair 2pt
\pinlabel{$\leftrightsquigarrow$}  at 460 330
\pinlabel{{\scriptsize $\Omega 0b$}}  at 465 400
\pinlabel{$\leftrightsquigarrow$}  at 950 330
\pinlabel{{\scriptsize $\Omega 3h4$}}  at 955 400
\pinlabel{$\leftrightsquigarrow$}  at 1500 330
\pinlabel{{\scriptsize $\Omega 0b$}}  at 1505 400
\endlabellist
\centering
\includegraphics[width=0.7\textwidth]{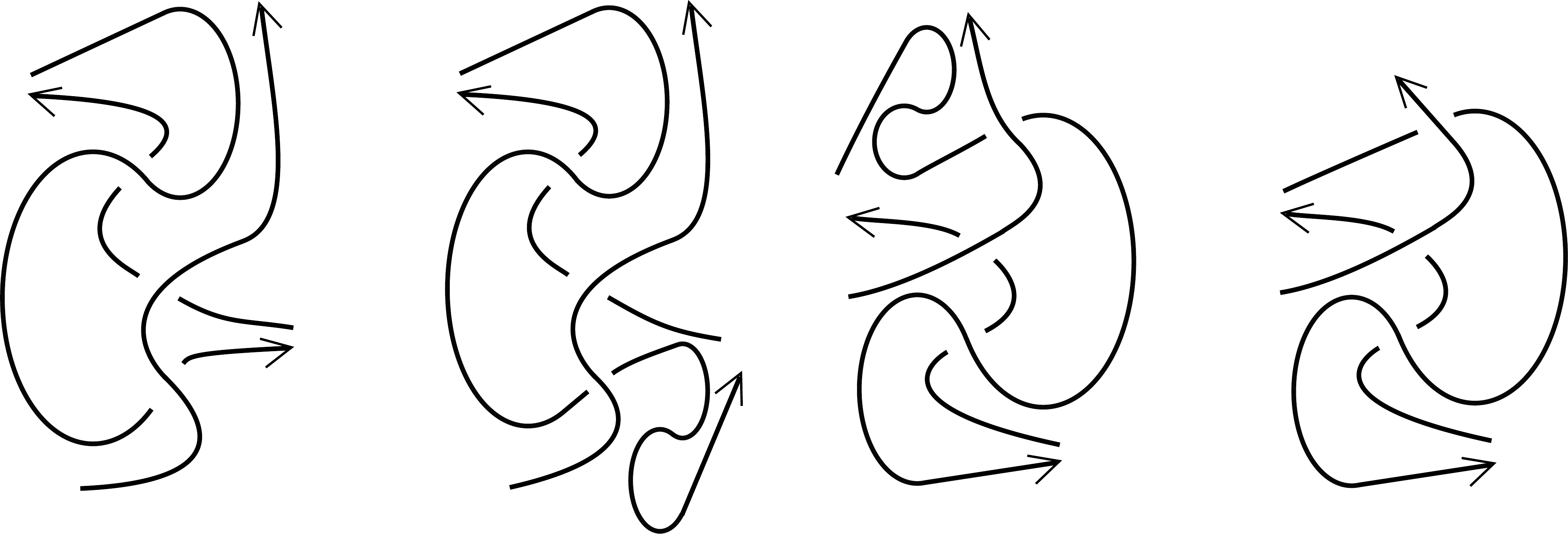}}.
\end{equation*}

We realise $\Omega 3h2$ from $\Omega 3h3$ and \eqref{eq:lemma1} via the moves
     \begin{equation*}
\centre{
\labellist \small \hair 2pt
\pinlabel{$\leftrightsquigarrow$}  at 460 330
\pinlabel{{\scriptsize \eqref{eq:lemma2}}}  at 465 400
\pinlabel{$\leftrightsquigarrow$}  at 980 330
\pinlabel{{\scriptsize $\Omega 3h3$}}  at 985 400
\pinlabel{$\leftrightsquigarrow$}  at 1500 330
\pinlabel{{\scriptsize \eqref{eq:lemma2}}}  at 1505 400
\endlabellist
\centering
\includegraphics[width=0.7\textwidth]{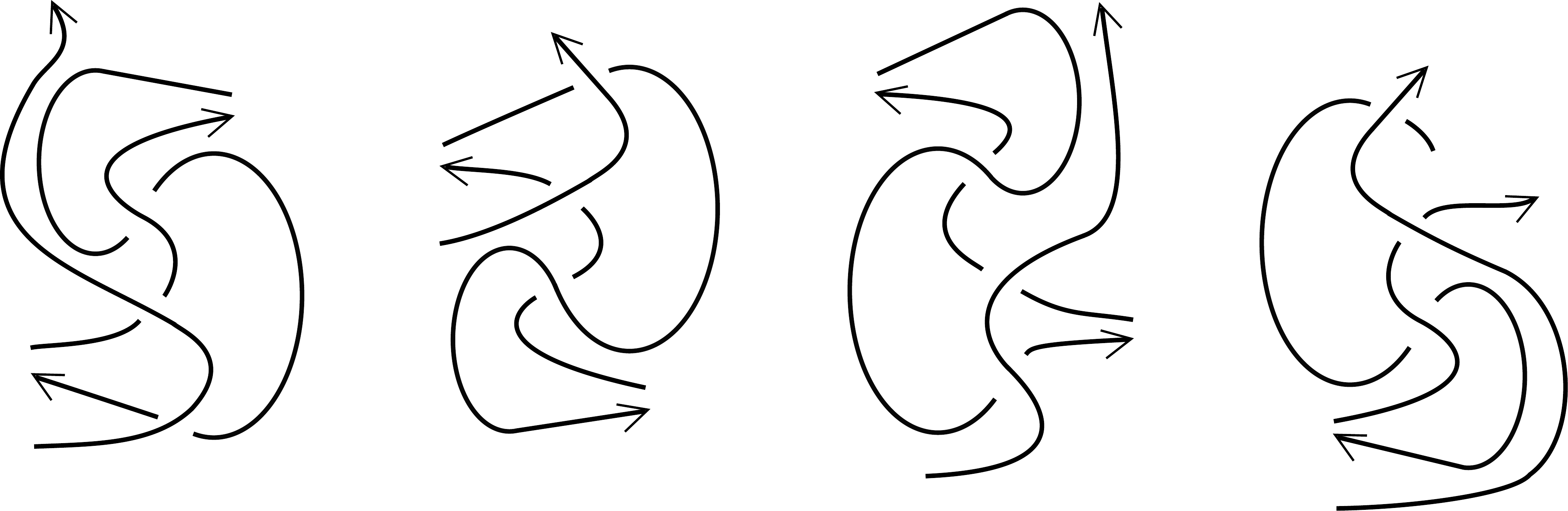}}.
\end{equation*}

\end{proof}

\begin{remark}\label{rem:O3h1_from_O3h2}
    Note that we can realize $\Omega 3h1$ from $\Omega 3h2$ and $\Omega 0b$ via the moves
        \begin{equation*}
\centre{
\labellist \small \hair 2pt
\pinlabel{$\leftrightsquigarrow$}  at 420 330
\pinlabel{{\scriptsize $\Omega 0b$}}  at 425 400
\pinlabel{$\leftrightsquigarrow$}  at 980 330
\pinlabel{{\scriptsize $\Omega 3h2$}}  at 985 400
\pinlabel{$\leftrightsquigarrow$}  at 1500 330
\pinlabel{{\scriptsize $\Omega 0b$}}  at 1505 400
\endlabellist
\centering
\includegraphics[width=0.7\textwidth]{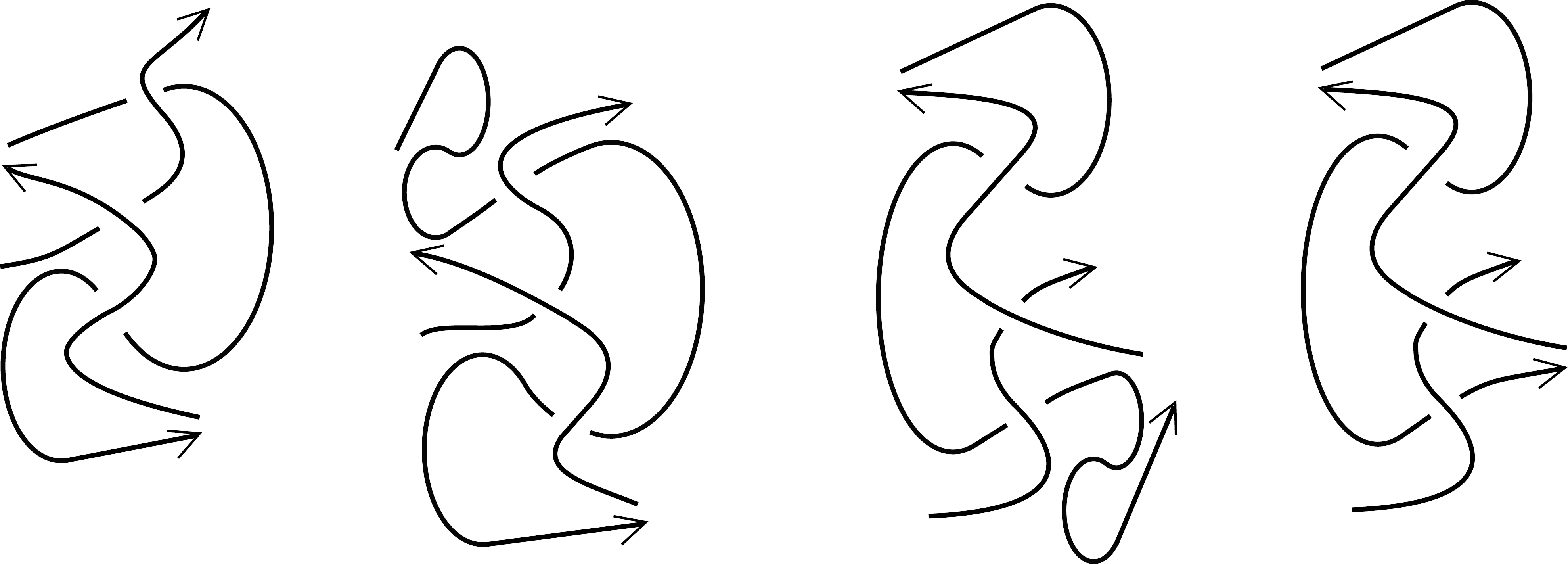}},
\end{equation*}
so given $\Omega 0b$ and \eqref{eq:lemma1}, we can realize all of the $\Omega 3hk$ ($1 \leq k \leq 6$) from just one of them.
\end{remark}

These lemmas, together with \cref{cor:planar->Morse}, conclude the proof of \cref{thm:1}.












\subsection{Other generating sets of rotational Reidemeister moves}

We now pass to describe different generating sets of rotational Reidemeister moves. As it was mentioned in the introduction, Caprau and Scott explicitly described in \cite{CS} all minimal sets of oriented Reidemeister moves. We can easily adapt their arguments in the rotational setting.

Consider the following three sets,
$$\begin{aligned}
\mathcal{O} := &\{ \Omega 0 a, \Omega 0 b, \Omega 0 c, \Omega 0 d   \}, \\
\mathcal{A} :=  &\{\{\Omega1a,\Omega1c,\Omega2a,\Omega3ak\},\{\Omega1a,\Omega1c,\Omega2b,\Omega3ak\},\{\Omega1b,\Omega1d,\Omega2a,\Omega3ak\},\\
&\{\Omega1b,\Omega1d,\Omega2b,\Omega3ak\},\{\Omega1a,\Omega1b,\Omega2a,\Omega3ak\},\{\Omega1a,\Omega1b,\Omega2b,\Omega3ak\}\}_{1 \leq k \leq 6}\\
\mathcal{H}:= &\{\{\Omega1a,\Omega1c,\Omega2a,\Omega3hk\},\{\Omega1a,\Omega1c,\Omega2b,\Omega3hk\},\{\Omega1b,\Omega1d,\Omega2a,\Omega3hk\},\\
&\{\Omega1b,\Omega1d,\Omega2b,\Omega3hk\},\{\Omega1c,\Omega1d,\Omega2a,\Omega3hk\},\{\Omega1c,\Omega1d,\Omega2b,\Omega3hk\}\}_{1 \leq k \leq 6},
\end{aligned}$$
where we remark that $\mathcal{A}$ and $\mathcal{H}$ are 36-element sets each of them.

\begin{theorem}\label{thm:all_min_sets}
Let $S$ be a minimal (i.e. 8-element) generating set of rotational Reidemeister moves. Then $S=\mathcal{O} \amalg X$ where $X \in \mathcal{A} \amalg \mathcal{H}$.
\end{theorem}

This follows directly from \cite[Theorem 2.24]{CS} and the minimality of the set $\mathcal{O}$ derived in \cref{cor:planar->Morse}, carefully adapting the realisations of the moves to the rotational setting as in \cref{sec:proof_1.1}.

In a similar fashion, we can adapt \cite[Theorem 1.2]{polyak10} with little effort to the rotational setting.

\begin{theorem}\label{thm:polyak_1.2_rot}
Let $S = \mathcal{O} \amalg X$ where $X$ is a set of at most five rotational Reidemeister moves of type 1,2 and 3  which contains $\Omega 3b$ as the only rotational Reidemeister move of type 3.

Then $S$ is generating if and only if $X$ contains one move of type $\Omega 2c$ (i.e. $\Omega 2c1$ or $\Omega 2c2$), one move of type $\Omega 2d$, and 
one of the pairs $$ (\Omega 1a, \Omega 1b) \quad , \quad (\Omega 1a, \Omega 1c) \quad , \quad (\Omega 1b, \Omega 1d) \quad , \quad (\Omega 1c, \Omega 1d). $$
\end{theorem}

Please note that we require all four $\Omega 0$ moves to be in the generating set in order to mimic Polyak's argument. However, there are other 9-element generating sets that do not include all four   $\Omega 0$ moves. For instance, according to the proof of \cref{lem:O1c_O1d}, the move $\Omega 0a$ can be obtained by a sequence of the moves $\Omega 1b$, $\Omega 1c$ and $\Omega 2d1$, and similarly the move $\Omega 0b$ can be obtained by a sequence of the moves $\Omega 1a$, $\Omega 1d$ and $\Omega 2c1$. Therefore, the 9-element set $$ S= \{ \Omega 0c, \Omega 0d, \Omega 1 a, \Omega 1b, \Omega 1c, \Omega 1d, \Omega 2c1, \Omega 2d1, \Omega 3b   \}$$ is also generating.

\begin{proof}
The proof runs in parallel to \cite[Theorem 1.2]{polyak10}, in particular the necessity part goes verbatim taking into account the four $\Omega 0$ moves. By \cref{lem:O2c2_O2d2} we can take any move of type $\Omega 2c$ and $\Omega 2d$ to get the other one. In order to obtain the $\Omega 3a$ moves, we first realise $\Omega 3a6 $ from $\Omega 3b$, $\Omega 2c2$ and $\Omega 2d1$ as follows:
\begin{equation*}
\centre{
\labellist \small \hair 2pt
\pinlabel{$\leftrightsquigarrow$}  at 420 210
\pinlabel{{\scriptsize $\Omega 2c2$}}  at 425 280
\pinlabel{$\leftrightsquigarrow$}  at 950 210
\pinlabel{{\scriptsize $\Omega 3b$}}  at 955 280
\pinlabel{$\leftrightsquigarrow$}  at 1520 210
\pinlabel{{\scriptsize $\Omega 2d1$}}  at 1525 280
\endlabellist
\centering
\includegraphics[width=0.7\textwidth]{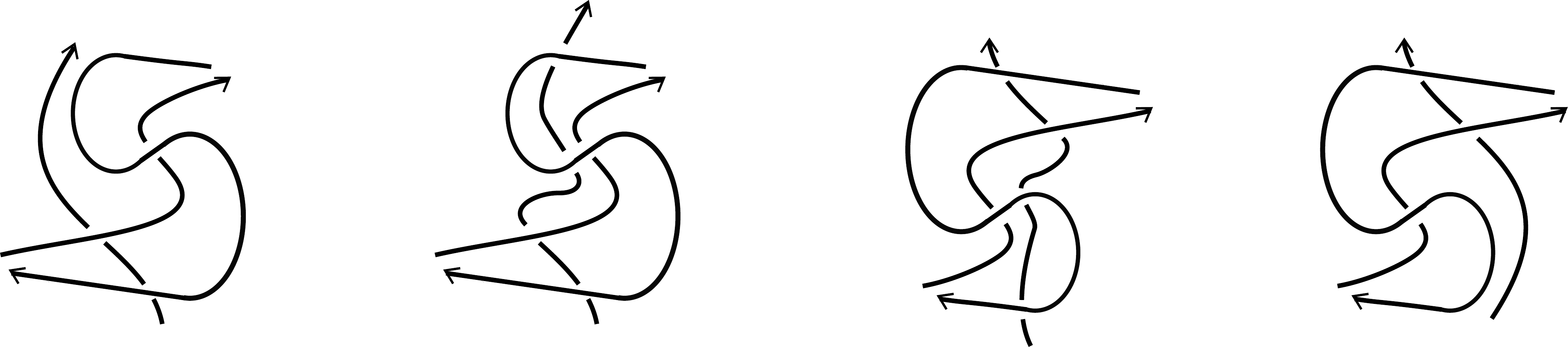}}.
\end{equation*}
Then the move $\Omega 3a1$ can be obtained from $\Omega 3a6$ and $\Omega 0$ moves as in \cref{rem:O3a1_from_O3a6}, and the rest of moves of type $\Omega 3a$ follows from \cref{lem:O3a}. Also $\Omega 3c$ is realised as in \cref{lem:O3c}. Independently on the pair of Reidemeister 1 moves that we picked, we can get the other two using the types $\Omega 2c$ and $\Omega 2d$. Then $\Omega 2b$ can be realised as in \cref{lem:O2b} and $\Omega 2a$ as
\begin{equation*}
\centre{
\labellist \small \hair 2pt
\pinlabel{$\leftrightsquigarrow$}  at 330 200
\pinlabel{{\scriptsize $\Omega 1c$}}  at 335 270
\pinlabel{$\leftrightsquigarrow$}  at 1000 200
\pinlabel{{\scriptsize $\Omega 2d1$}}  at 1005 270
\pinlabel{$\leftrightsquigarrow$}  at 1550 200
\pinlabel{{\scriptsize $\Omega 3c$}}  at 1555 270
\pinlabel{$\leftrightsquigarrow$}  at 2140 200
\pinlabel{{\scriptsize $\Omega 1c$}}  at 2145 270
\endlabellist
\centering
\includegraphics[width=0.8\textwidth]{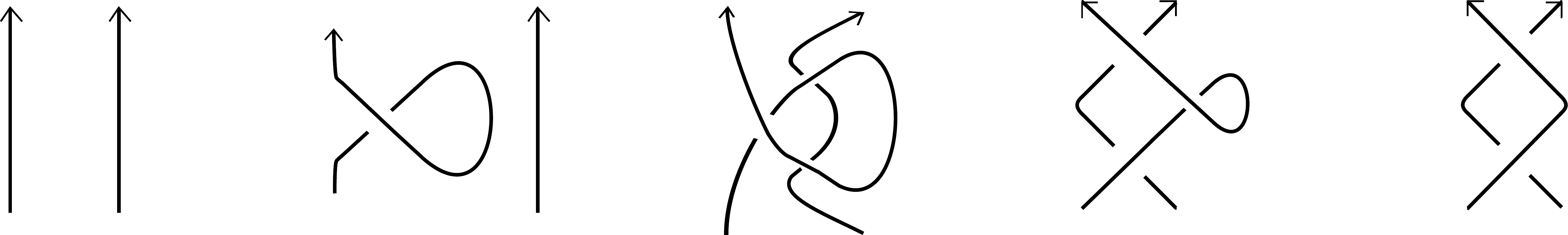}}.
\end{equation*}
The remaining Reidemeister 3 moves are obtained as in Lemmas \ref{lem:O3d_g} -- \ref{lem:O3h}.
\end{proof}

\section{Rotational Reidemeister moves for framed, oriented tangles}

In this section, all tangles will be assumed to be framed and oriented.

\subsection{Proof of \cref{thm:2} }

The aim of this subsection is to prove \cref{thm:2}. We give a detailed proof adapting the (necessary) lemmas of \cref{sec:3} to the framed setting. It follows from Lemmas \ref{lem:O3a}, \ref{lem:O2c2_O2d2} and \ref{lem:O3h} that, in the presence of all $\Omega 0$ moves, a single rotational Reidemeister move of type  $\Omega 2c$, $\Omega 2d$, $\Omega 3a$ or $\Omega 3h$ is enough to realise the rest of the given type.

We start by realising a third Reidemeister 2 move.

\begin{lemma}
The move $\Omega 2c1$ can be realised as a combination of the moves $\Omega 0a$, $\Omega 2a$, $\Omega 2d2$, $\Omega 3a3$, $\Omega 3h3$ and either $\Omega 1\textup{f}a$ or $\Omega 1\textup{f}c$. Similarly, $\Omega 2c2$ can be realised as a combination of the moves $\Omega 0a$, $\Omega 2b$, $\Omega 2d1$, $\Omega 3a1$, $\Omega 3h5$ and either $\Omega 1\textup{f}a$ or $\Omega 1\textup{f}c$.

On the other hand, the move $\Omega 2d1$ can be realised as a combination of the moves  $\Omega 0b$, $\Omega 2a$, $\Omega 2c2$, $\Omega 3a6$, $\Omega 3h4$ and either $\Omega 1\textup{f}b$ or $\Omega 1\textup{f}d$. Similarly, the  move $\Omega 2d2$ can be obtained as a combination of the moves  $\Omega 0b$, $\Omega 2b$, $\Omega 2c1$, $\Omega 3a2$, $\Omega 3h2$ and either $\Omega 1\textup{f}b$ or $\Omega 1\textup{f}d$.
\end{lemma}
\begin{proof}
Let us realise first $\Omega 2c1$. In the presence of $\Omega 1\textup{f}a$, we have
        \begin{equation*}
\centre{
\labellist \small \hair 2pt
\pinlabel{$\leftrightsquigarrow$}  at 380 430
\pinlabel{{\scriptsize $\Omega 1\text{f}a$}}  at 385 500
\pinlabel{$=$}  at 960 430
\pinlabel{$\leftrightsquigarrow$}  at 1580 430
\pinlabel{{\scriptsize $\Omega 2a$}}  at 1585 500
\pinlabel{{\scriptsize $(\times 2)$}}  at 1585 370
\pinlabel{$\leftrightsquigarrow$}  at 2260 430
\pinlabel{{\scriptsize $\Omega 3h3$}}  at 2265 500
\endlabellist
\centering
\includegraphics[width=0.9\textwidth]{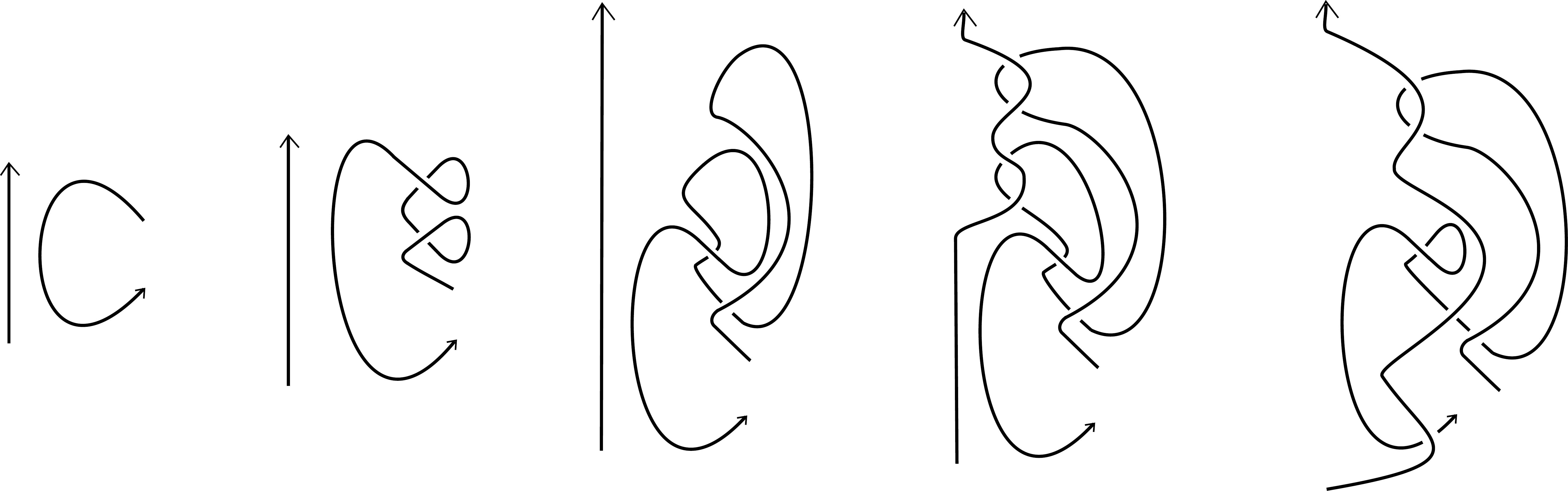}}
\end{equation*}
\begin{equation*} \phantom{------}   
\centre{
\labellist \small \hair 2pt
\pinlabel{$\leftrightsquigarrow$}  at -90 350
\pinlabel{{\scriptsize $\Omega 0a$}}  at -85 430
\pinlabel{$\leftrightsquigarrow$}  at 820 350
\pinlabel{{\scriptsize $\Omega 3a3$}}  at 825 430
\pinlabel{$\leftrightsquigarrow$}  at 1790 350
\pinlabel{{\scriptsize $\Omega 2d2$}}  at 1795 430
\endlabellist
\centering
\includegraphics[width=0.8\textwidth]{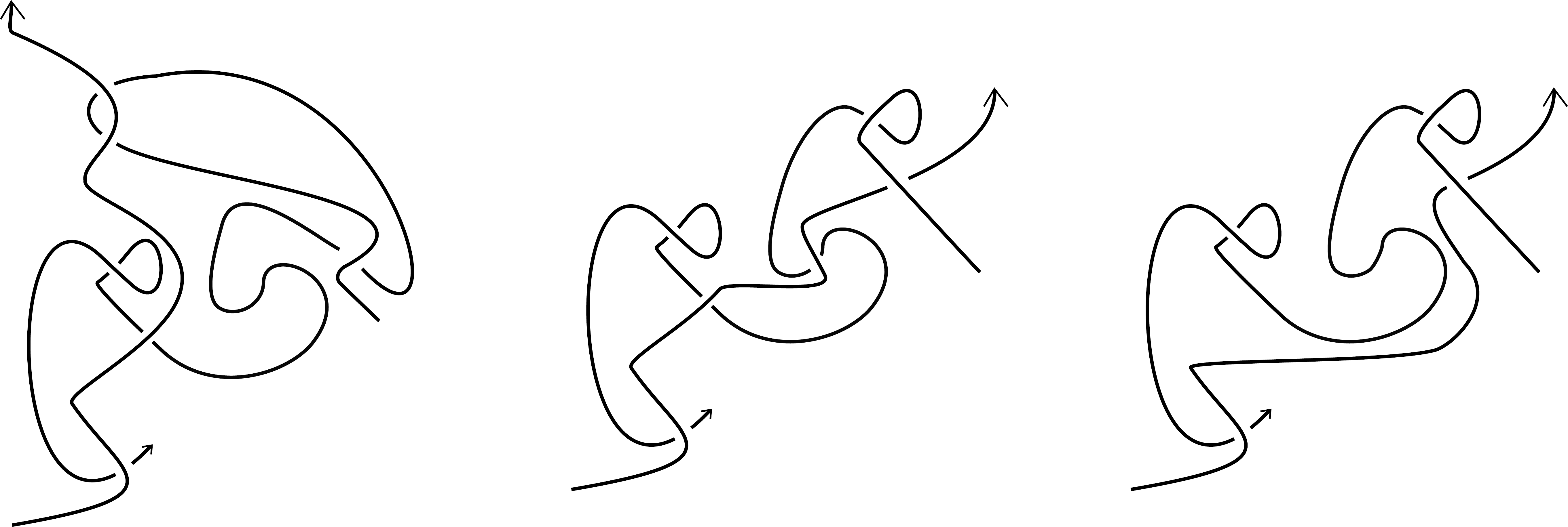}}
\end{equation*}
\begin{equation*}   
\centre{
\labellist \small \hair 2pt
\pinlabel{$\leftrightsquigarrow$}  at -90 400
\pinlabel{{\scriptsize $\Omega 0a$}}  at -85 470
\pinlabel{$\leftrightsquigarrow$}  at 500 400
\pinlabel{{\scriptsize $\Omega 1\text{f}a$}}  at 505 470
\endlabellist
\centering
\includegraphics[width=0.3\textwidth]{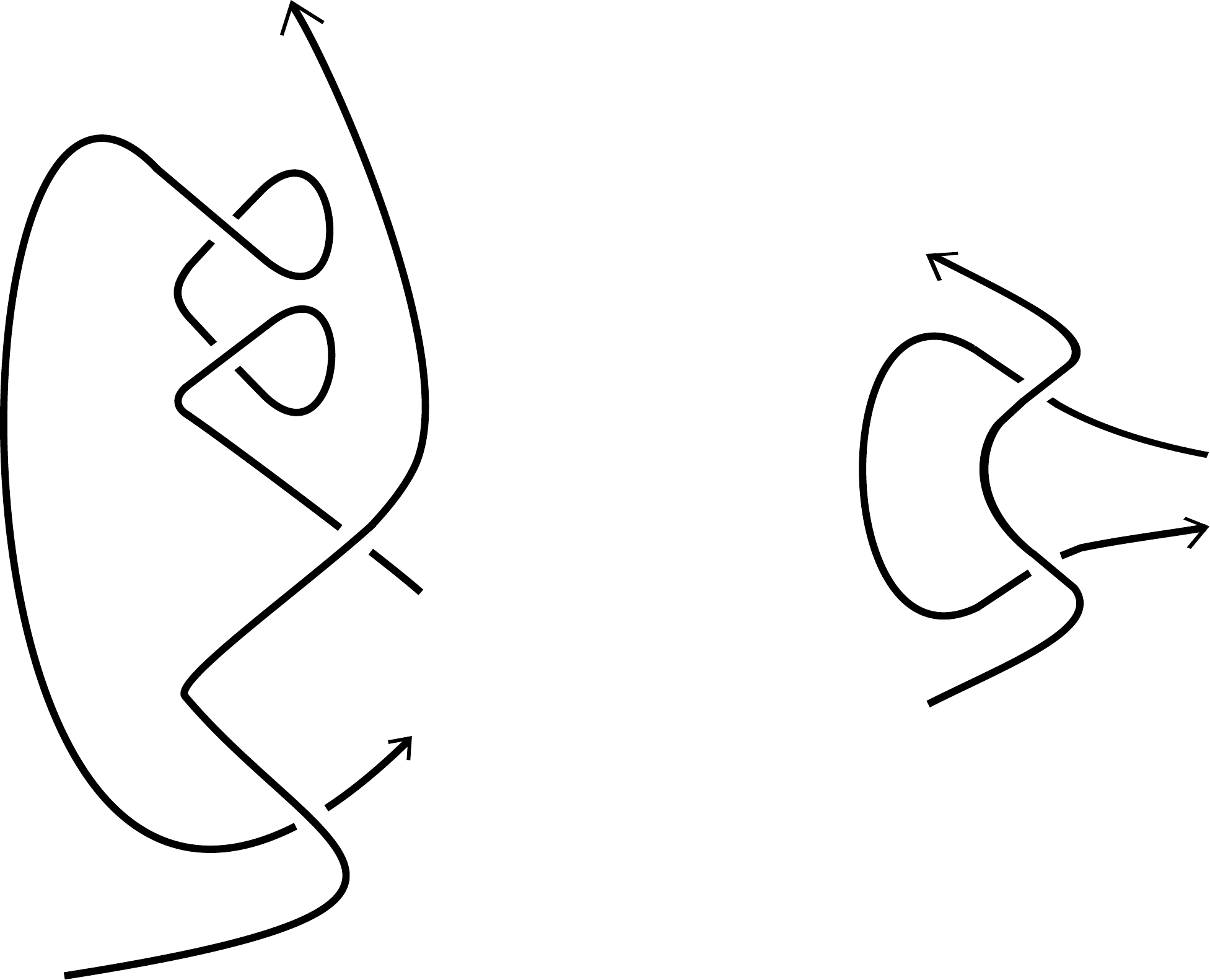}},
\phantom{-----------------} 
\end{equation*}
and if we had included $\Omega 1\textup{f}c$ instead then the only change would be that the order of the moves $\Omega 3a3$ and $\Omega 3h3$ would be exchanged. The realisation of $\Omega 2c2$ goes in parallel   slicing the unknotted, upwards component under the $C$-shaped component.

Similarly, in the presence of $\Omega 2c2$ and $\Omega 1\textup{f}b$ we can realise $\Omega 2d1$ as follows:
        \begin{equation*}
\centre{
\labellist \small \hair 2pt
\pinlabel{$\leftrightsquigarrow$}  at 380 430
\pinlabel{{\scriptsize $\Omega 1\text{f}b$}}  at 385 500
\pinlabel{$=$}  at 960 430
\pinlabel{$\leftrightsquigarrow$}  at 1580 430
\pinlabel{{\scriptsize $\Omega 2a$}}  at 1585 500
\pinlabel{{\scriptsize $(\times 2)$}}  at 1585 370
\pinlabel{$\leftrightsquigarrow$}  at 2260 430
\pinlabel{{\scriptsize $\Omega 3h4$}}  at 2265 500
\endlabellist
\centering
\includegraphics[width=0.9\textwidth]{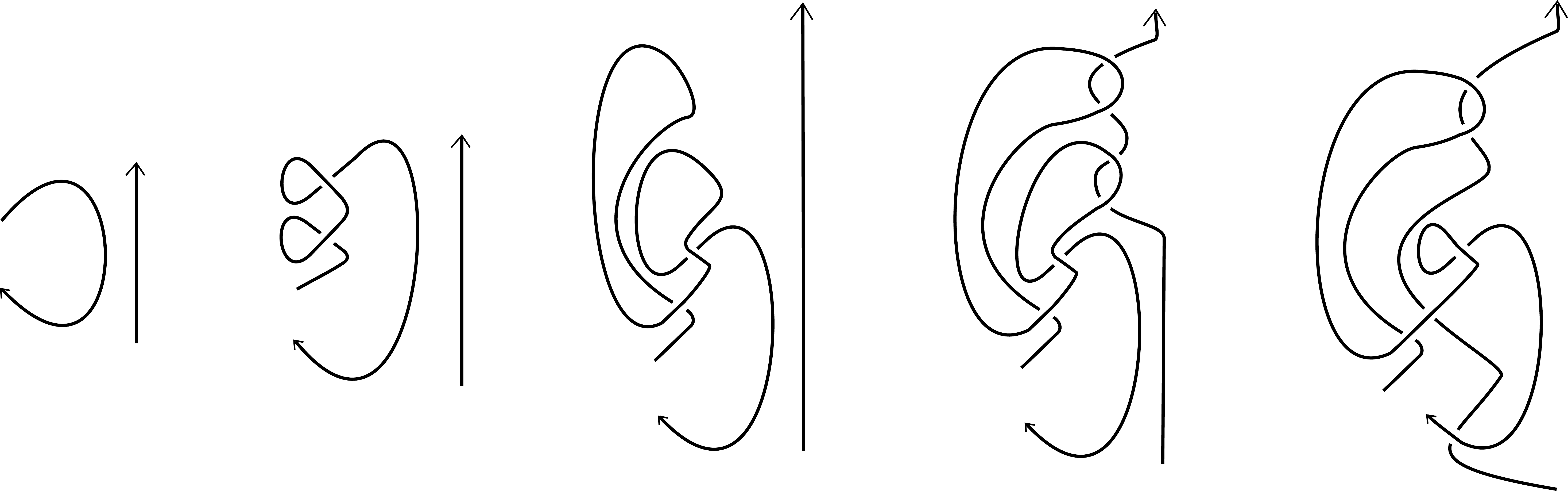}}
\end{equation*}
\begin{equation*} \phantom{------}   
\centre{
\labellist \small \hair 2pt
\pinlabel{$\leftrightsquigarrow$}  at -90 350
\pinlabel{{\scriptsize $\Omega 0b$}}  at -85 430
\pinlabel{$\leftrightsquigarrow$}  at 820 350
\pinlabel{{\scriptsize $\Omega 3a6$}}  at 825 430
\pinlabel{$\leftrightsquigarrow$}  at 1790 350
\pinlabel{{\scriptsize $\Omega 2c2$}}  at 1795 430
\endlabellist
\centering
\includegraphics[width=0.8\textwidth]{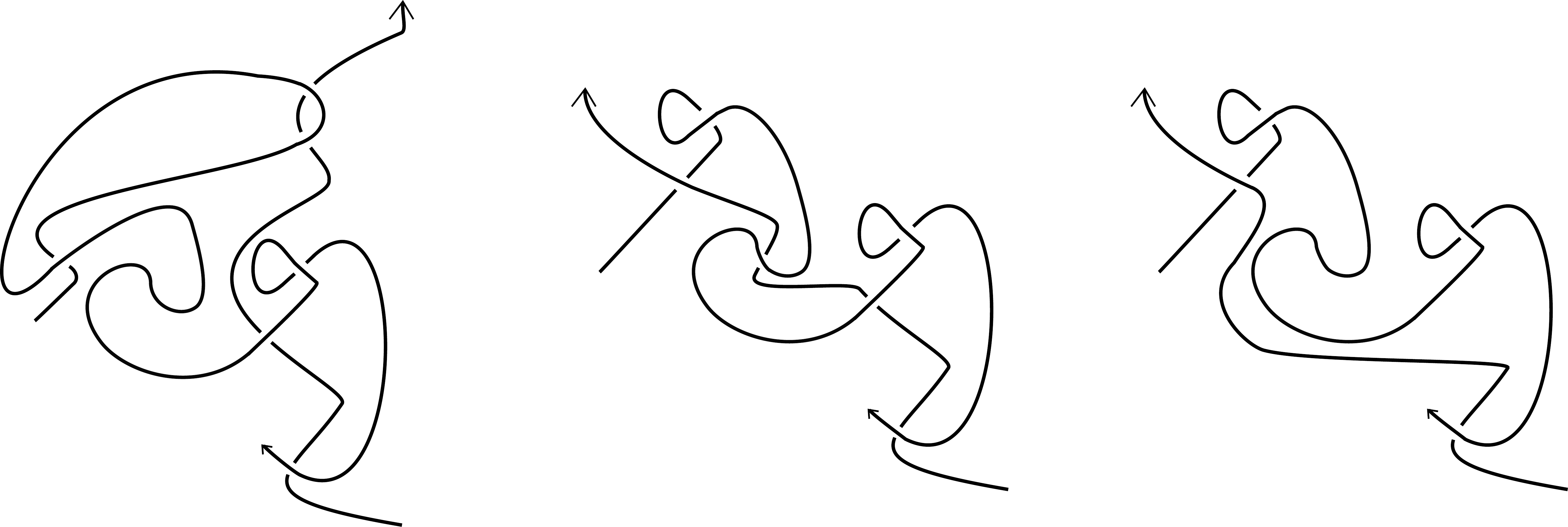}}
\end{equation*}
\begin{equation*}   
\centre{
\labellist \small \hair 2pt
\pinlabel{$\leftrightsquigarrow$}  at -90 400
\pinlabel{{\scriptsize $\Omega 0b$}}  at -85 470
\pinlabel{$\leftrightsquigarrow$}  at 500 400
\pinlabel{{\scriptsize $\Omega 1\text{f}b$}}  at 505 470
\endlabellist
\centering
\includegraphics[width=0.3\textwidth]{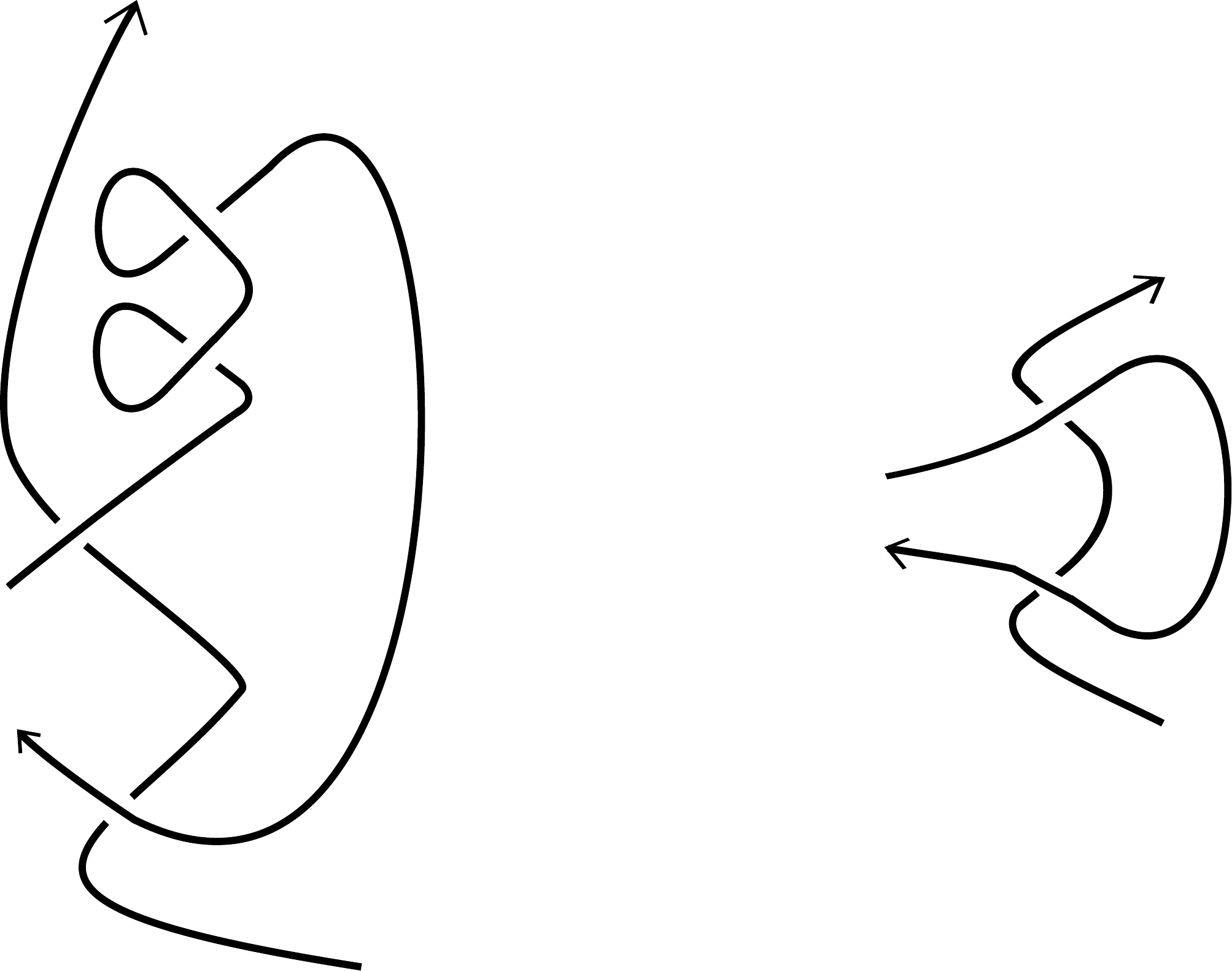}},
\phantom{-----------------} 
\end{equation*}
and if we had included $\Omega 1\textup{f}d$ instead then the only change would be that the order of the moves $\Omega 3a6$ and $\Omega 3h4$ would be exchanged. Realising $\Omega 2d2$ is very similar, slicing the right-hand side component under (it is in the fact the mirror image of the realisation of $\Omega 2c1$ above, and likewise  the realisation of $\Omega 2c2$ is the mirror image of the realisation of $\Omega 2d1$).
\end{proof}

Once we are equipped with the moves of type $\Omega 2c$ and $\Omega 2d$, we can obtain all Reidemeister 3 moves from the moves of type  $\Omega 3a$ and $\Omega 3h$. Indeed,  Lemmas \ref{lem:O3b} and \ref{lem:O3c} show that we can obtain $\Omega 3b$ and $\Omega 3c$, respectively. 

It turns out that the rest of Reidemeister 3 moves can be obtained from the types $\Omega 0$, $\Omega 3h$, $\Omega 2c$ and $\Omega 2d$.

\begin{lemma}
The moves $\Omega 3d$, $\Omega 3e$  and $\Omega 3g$  can be obtained from the types $\Omega 0$, $\Omega 2c$, $\Omega 2d$ and $\Omega 3h$. Similarly, the move  $\Omega 3f$ can be obtained from the types  $\Omega 0$, $\Omega 2c$, $\Omega 2d$ and $\Omega 3a$.
\end{lemma}
\begin{proof}
We can realise $\Omega 3d$ as
\begin{equation*}
\centre{
\labellist \small \hair 2pt
\pinlabel{$\leftrightsquigarrow$}  at 330 370
\pinlabel{{\scriptsize $\Omega 0a$}}  at 335 430
\pinlabel{$\leftrightsquigarrow$}  at 980 370
\pinlabel{{\scriptsize $\Omega 2c1$}}  at 985 430
\pinlabel{$\leftrightsquigarrow$}  at 1610 370
\pinlabel{{\scriptsize $\Omega 3h4$}}  at 1615 430
\endlabellist
\centering
\includegraphics[width=0.7\textwidth]{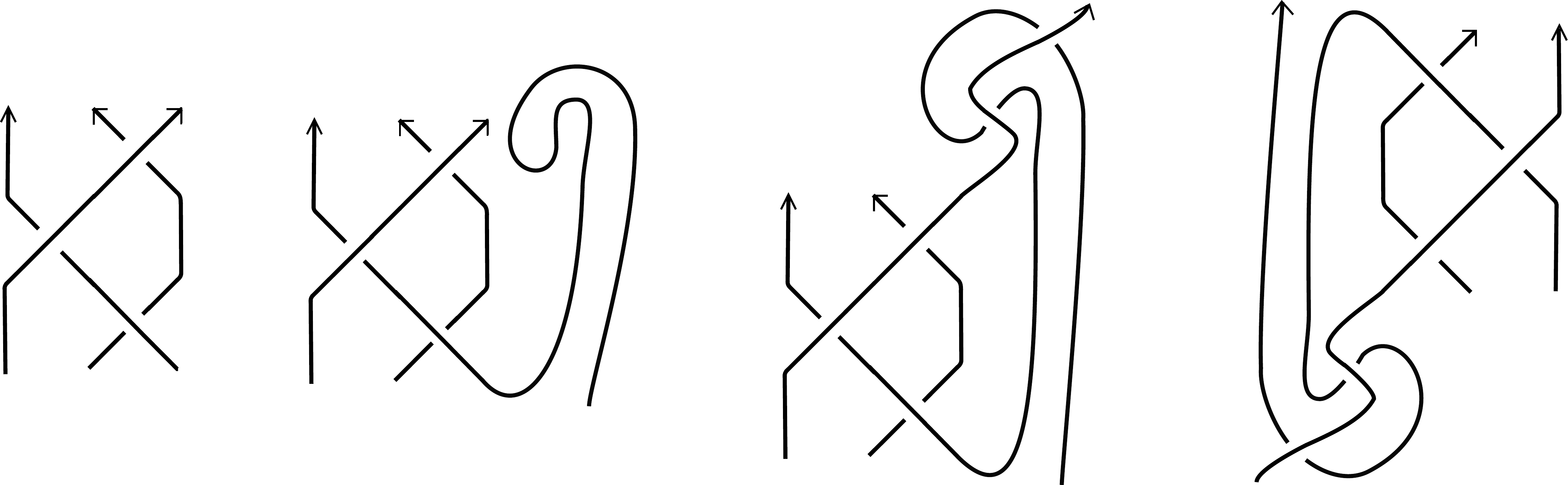}}
\end{equation*}
\begin{equation*}
\centre{
\labellist \small \hair 2pt
\pinlabel{$\leftrightsquigarrow$}  at -60 225
\pinlabel{{\scriptsize $\Omega 2d2$}}  at -55 295
\pinlabel{$\leftrightsquigarrow$}  at 580 225
\pinlabel{{\scriptsize $\Omega 0a$}}  at 585 295
\endlabellist
\centering
\includegraphics[width=0.3\textwidth]{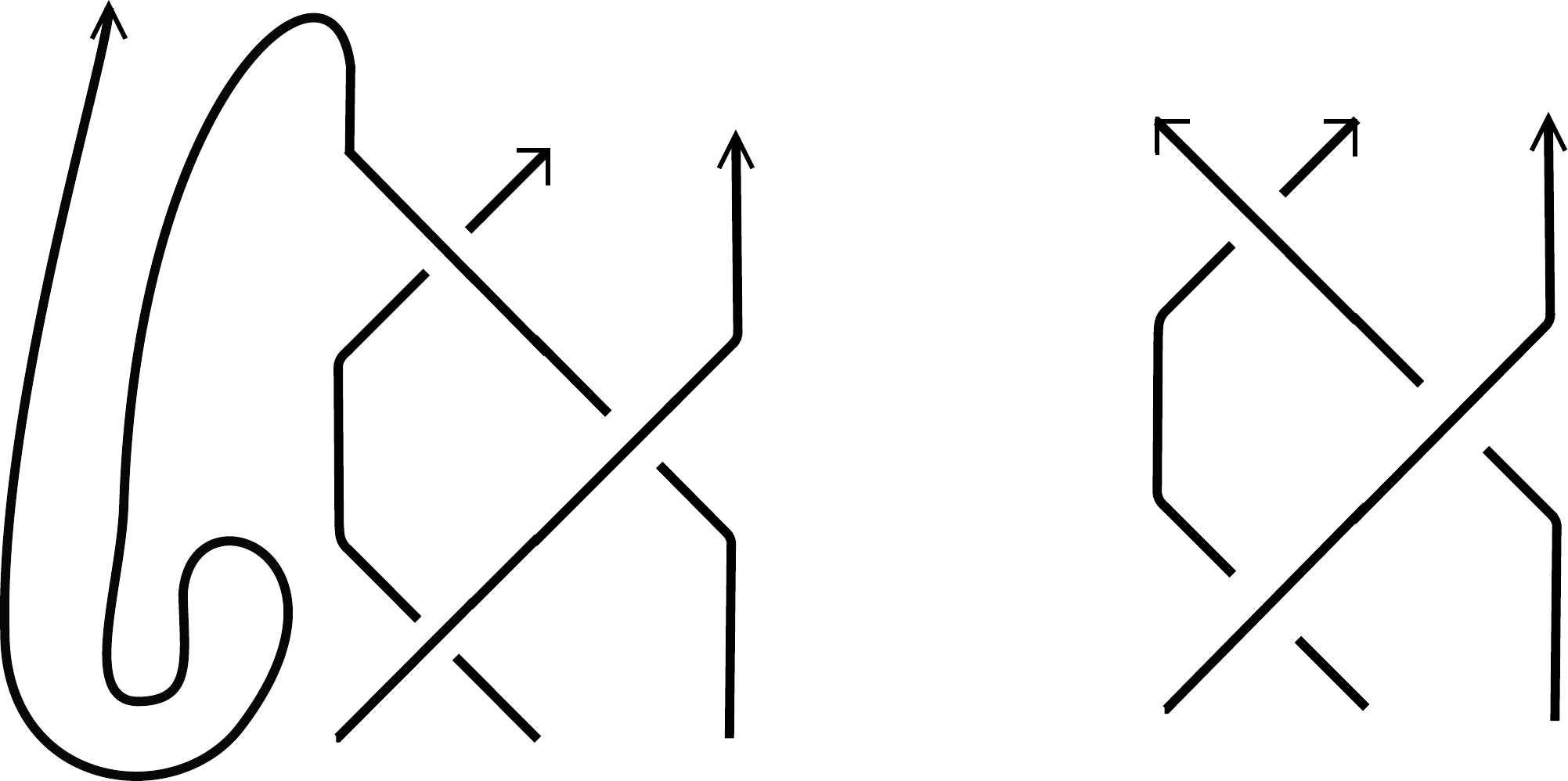}}. \phantom{-------}
\end{equation*}
Similarly, we obtain $\Omega 3e$ as 
\begin{equation*}
\centre{
\labellist \small \hair 2pt
\pinlabel{$\leftrightsquigarrow$}  at 330 370
\pinlabel{{\scriptsize $\Omega 0b$}}  at 335 430
\pinlabel{$\leftrightsquigarrow$}  at 980 370
\pinlabel{{\scriptsize $\Omega 2d1$}}  at 985 430
\pinlabel{$\leftrightsquigarrow$}  at 1610 370
\pinlabel{{\scriptsize $\Omega 3h3$}}  at 1615 430
\endlabellist
\centering
\includegraphics[width=0.7\textwidth]{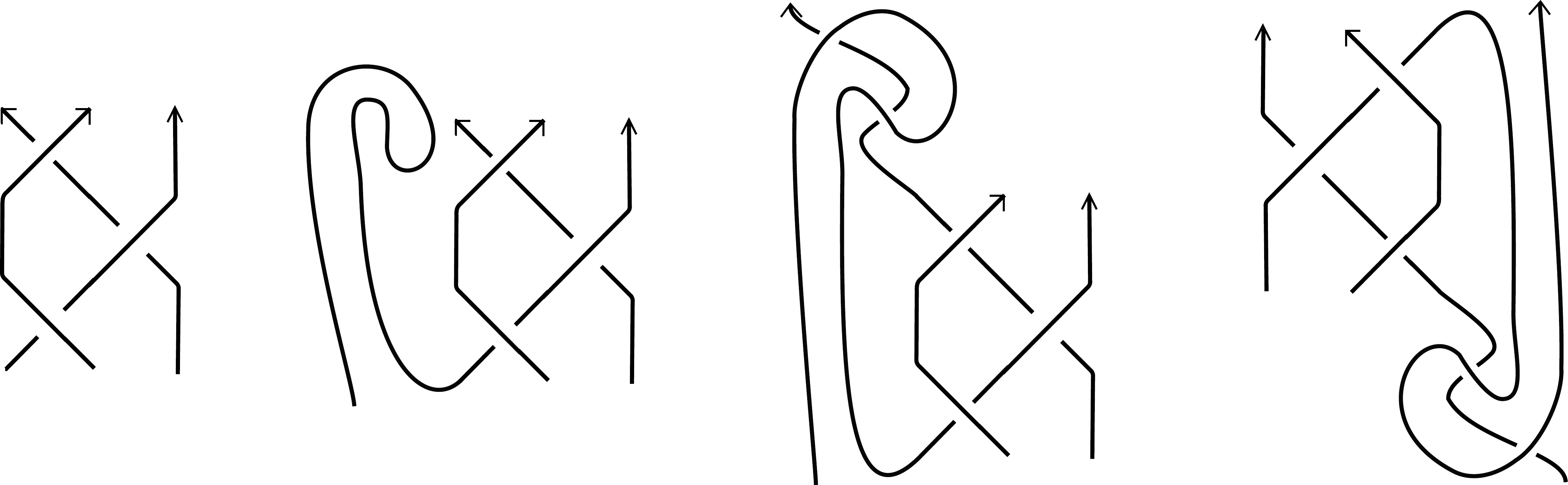}}
\end{equation*}
\begin{equation*}
\centre{
\labellist \small \hair 2pt
\pinlabel{$\leftrightsquigarrow$}  at -90 225
\pinlabel{{\scriptsize $\Omega 2c2$}}  at -85 295
\pinlabel{$\leftrightsquigarrow$}  at 580 225
\pinlabel{{\scriptsize $\Omega 0b$}}  at 585 295
\endlabellist
\centering
\includegraphics[width=0.3\textwidth]{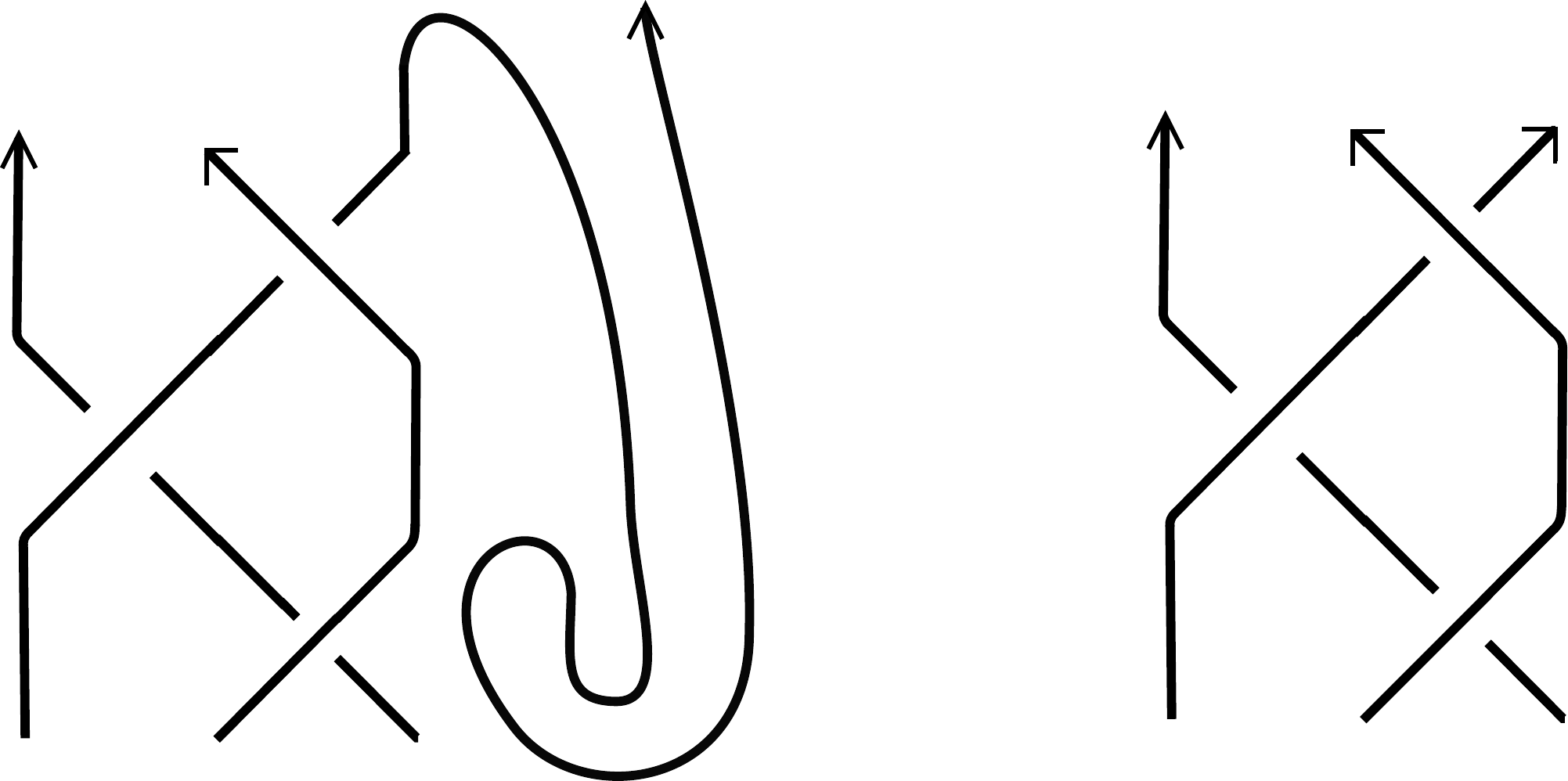}}. \phantom{------}
\end{equation*}
To realise  $\Omega 3f$ we do
\begin{equation*}
\centre{
\labellist \small \hair 2pt
\pinlabel{$\leftrightsquigarrow$}  at 330 370
\pinlabel{{\scriptsize $\Omega 0a$}}  at 335 430
\pinlabel{$\leftrightsquigarrow$}  at 980 370
\pinlabel{{\scriptsize $\Omega 2d1$}}  at 985 430
\pinlabel{$\leftrightsquigarrow$}  at 1610 370
\pinlabel{{\scriptsize $\Omega 3a2$}}  at 1615 430
\endlabellist
\centering
\includegraphics[width=0.7\textwidth]{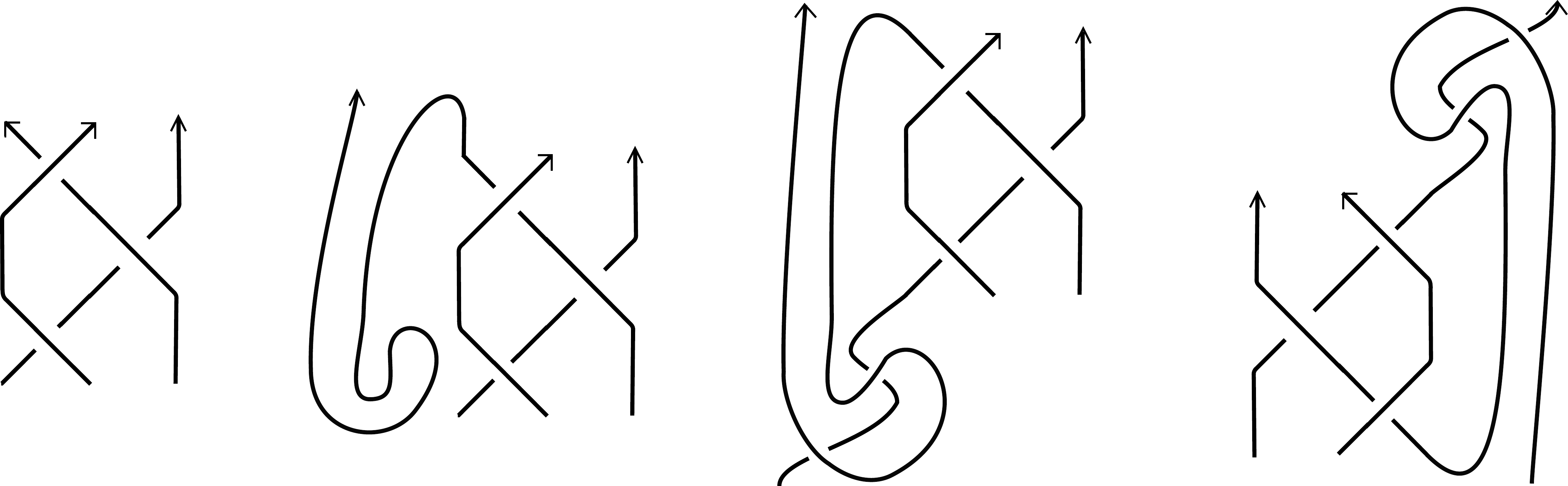}}
\end{equation*}
\begin{equation*}
\centre{
\labellist \small \hair 2pt
\pinlabel{$\leftrightsquigarrow$}  at -90 225
\pinlabel{{\scriptsize $\Omega 2c2$}}  at -85 295
\pinlabel{$\leftrightsquigarrow$}  at 580 225
\pinlabel{{\scriptsize $\Omega 0a$}}  at 585 295
\endlabellist
\centering
\includegraphics[width=0.3\textwidth]{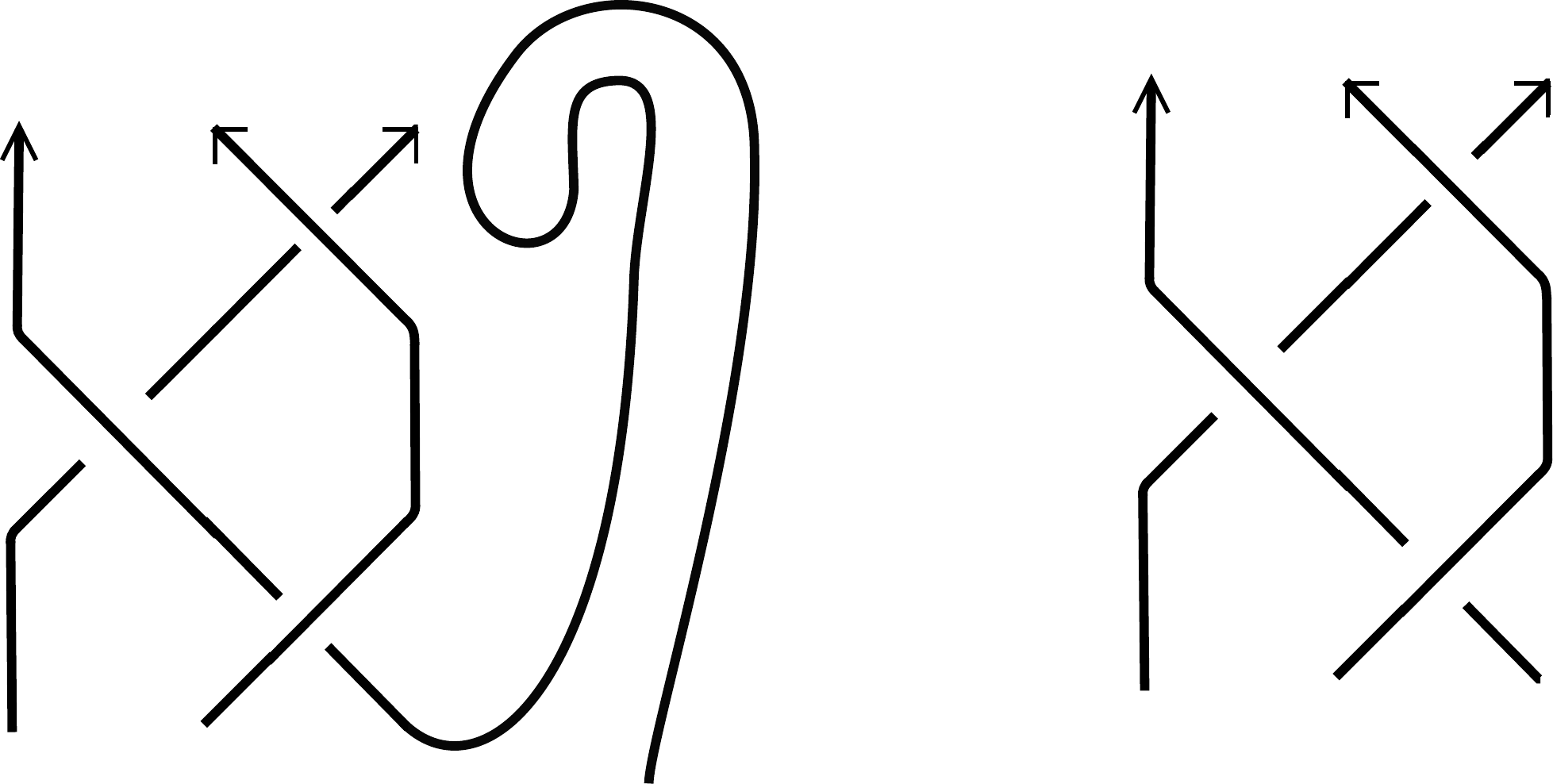}}. \phantom{------}
\end{equation*}
Lastly, $\Omega 3g$ can be realised as
\begin{equation*}
\centre{
\labellist \small \hair 2pt
\pinlabel{$\leftrightsquigarrow$}  at 330 370
\pinlabel{{\scriptsize $\Omega 0b$}}  at 335 430
\pinlabel{$\leftrightsquigarrow$}  at 980 370
\pinlabel{{\scriptsize $\Omega 2d2$}}  at 985 430
\pinlabel{$\leftrightsquigarrow$}  at 1610 370
\pinlabel{{\scriptsize $\Omega 3h1$}}  at 1615 430
\endlabellist
\centering
\includegraphics[width=0.7\textwidth]{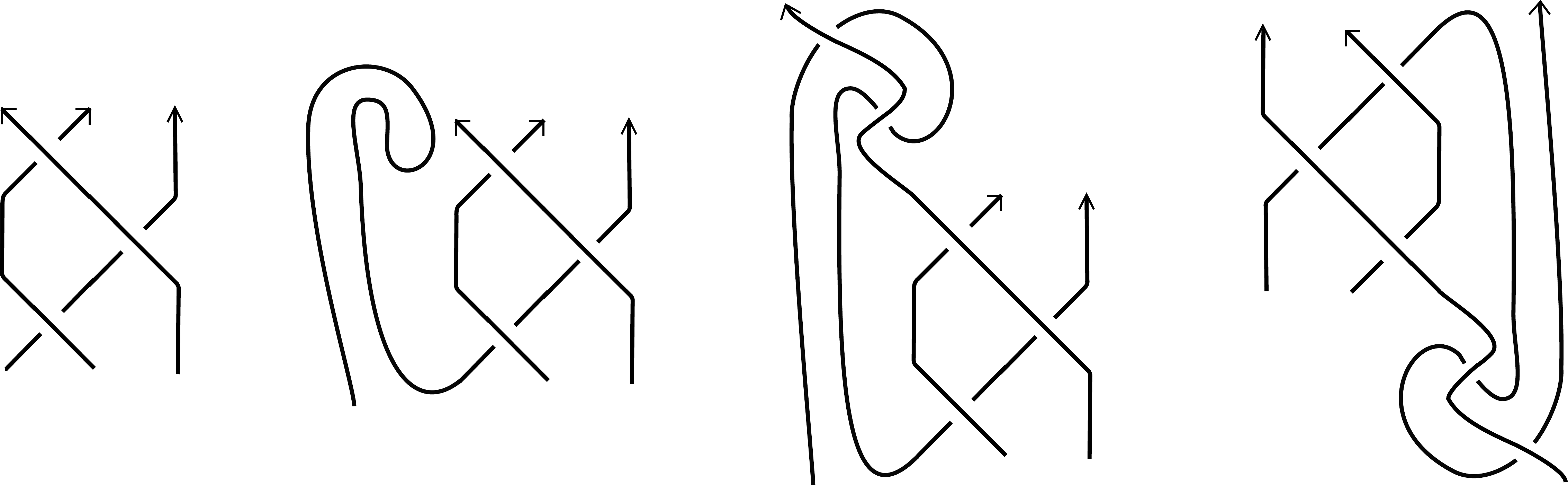}}
\end{equation*}
\begin{equation*}
\centre{
\labellist \small \hair 2pt
\pinlabel{$\leftrightsquigarrow$}  at -90 225
\pinlabel{{\scriptsize $\Omega 2c1$}}  at -85 295
\pinlabel{$\leftrightsquigarrow$}  at 580 225
\pinlabel{{\scriptsize $\Omega 0b$}}  at 585 295
\endlabellist
\centering
\includegraphics[width=0.3\textwidth]{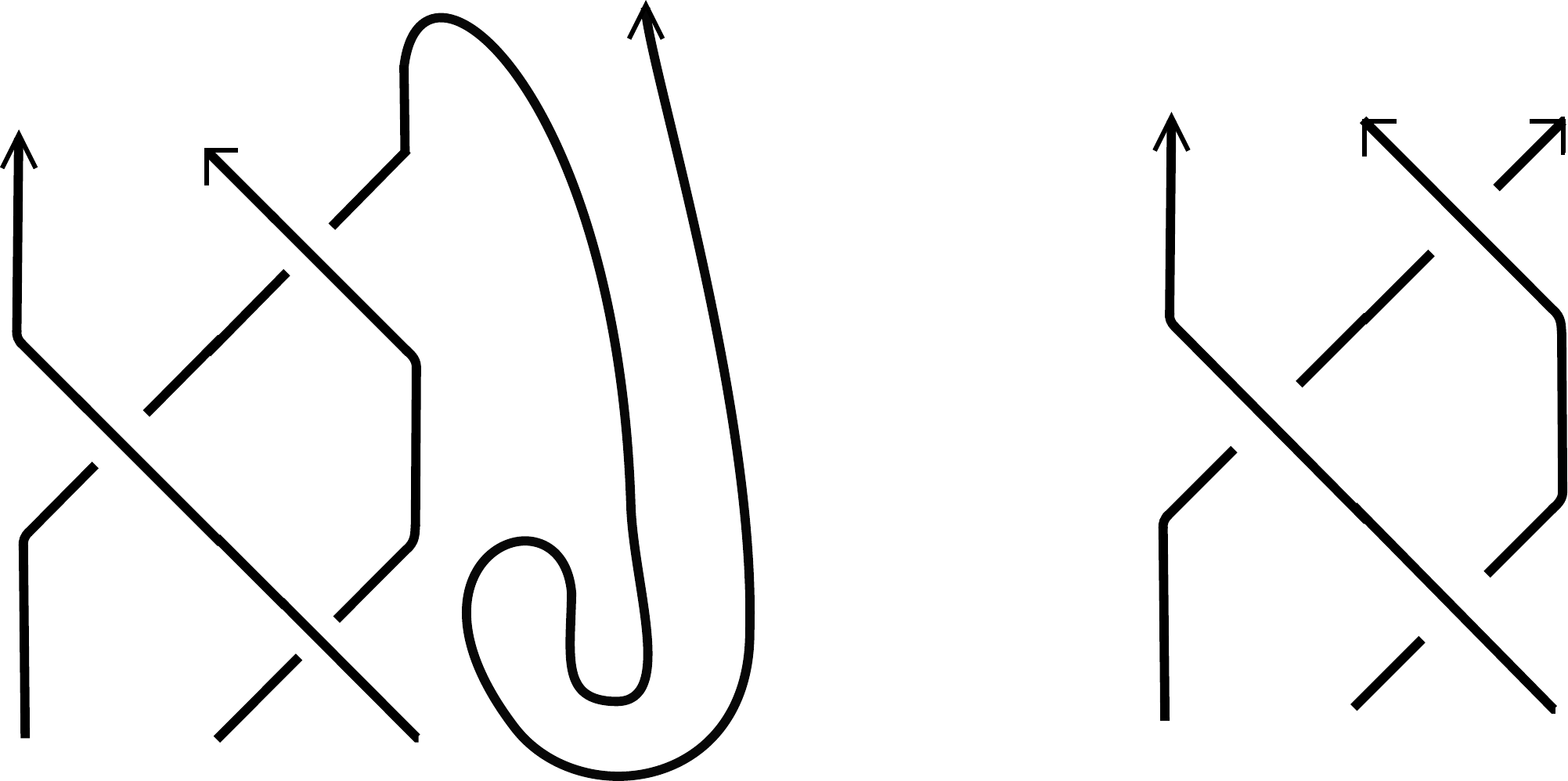}}. \phantom{------}
\end{equation*}
\end{proof}

Endowed with all Reidemeister 3 moves, obtaining the remaining Reidemeister 2 move is easy.

\begin{lemma}
The $\Omega 2b$ move can be realised by a sequence of the moves $\Omega 2d2$, $\Omega 3e$, $\Omega 3f$, $\Omega 2a$ and either $\Omega 1 \textup{f}a$ or $\Omega 1 \textup{f}c$.

Similarly, the $\Omega 2a$ move can be obtained as a sequence of the moves $\Omega 2c1$, $\Omega 3e$, $\Omega 3f$, $\Omega 2b$ and either $\Omega 1 \textup{f}b$ or $\Omega 1 \textup{f}d$.
\end{lemma}
\begin{proof}
If we had $\Omega 1 \text{f}a$, we can realise $\Omega 2b$ as
\begin{equation*}
\centre{
\labellist \small \hair 2pt
\pinlabel{$\leftrightsquigarrow$}  at 310 240
\pinlabel{{\scriptsize $\Omega 1 \text{f}a$}}  at 315 310
\pinlabel{{\scriptsize $\Omega 2d2$}}  at 840 310
\pinlabel{$\leftrightsquigarrow$}  at 845 240
\pinlabel{{\scriptsize $(\times 2)$}}  at 840 170
\pinlabel{{\scriptsize $\Omega 3e$}}  at 1240 310
\pinlabel{$\leftrightsquigarrow$}  at 1245 240
\pinlabel{{\scriptsize $\Omega 3f$}}  at 1670 310
\pinlabel{$\leftrightsquigarrow$}  at 1675 240
\pinlabel{{\scriptsize $\Omega 2a$}}  at 2140 310
\pinlabel{$\leftrightsquigarrow$}  at 2145 240
\pinlabel{{\scriptsize $\Omega 1 \text{f}a$}}  at 2610 310
\pinlabel{$\leftrightsquigarrow$}  at 2615 240
\endlabellist
\centering
\includegraphics[width=0.95\textwidth]{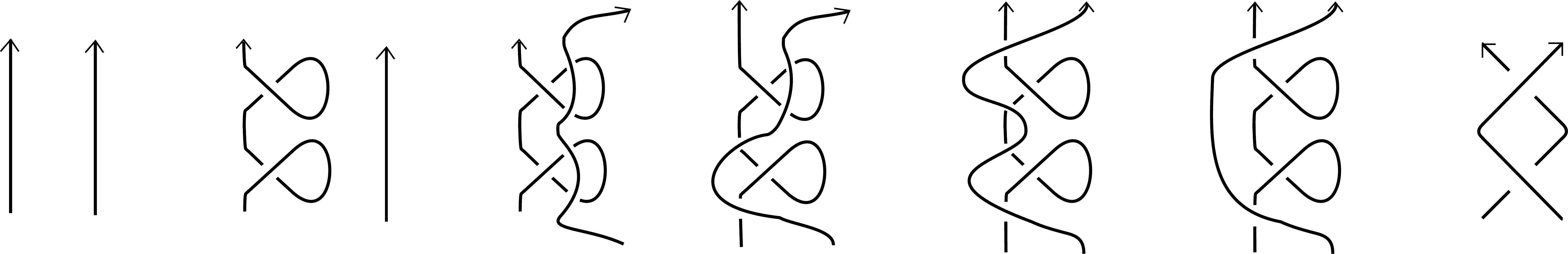}}
\end{equation*}
and with $\Omega 1 \text{f}c$ the only change would be the order of the moves $\Omega 3e$ and $\Omega 3f$.

On the other hand, we can obtain $\Omega 2a$ using  $\Omega 1 \text{f}d$ as
\begin{equation*}
\centre{
\labellist \small \hair 2pt
\pinlabel{$\leftrightsquigarrow$}  at 310 240
\pinlabel{{\scriptsize $\Omega 1 \text{f}d$}}  at 315 310
\pinlabel{{\scriptsize $\Omega 2c1$}}  at 840 310
\pinlabel{$\leftrightsquigarrow$}  at 845 240
\pinlabel{{\scriptsize $(\times 2)$}}  at 840 170
\pinlabel{{\scriptsize $\Omega 3f$}}  at 1240 310
\pinlabel{$\leftrightsquigarrow$}  at 1245 240
\pinlabel{{\scriptsize $\Omega 3e$}}  at 1670 310
\pinlabel{$\leftrightsquigarrow$}  at 1675 240
\pinlabel{{\scriptsize $\Omega 2b$}}  at 2140 310
\pinlabel{$\leftrightsquigarrow$}  at 2145 240
\pinlabel{{\scriptsize $\Omega 1 \text{f}d$}}  at 2610 310
\pinlabel{$\leftrightsquigarrow$}  at 2615 240
\endlabellist
\centering
\includegraphics[width=0.95\textwidth]{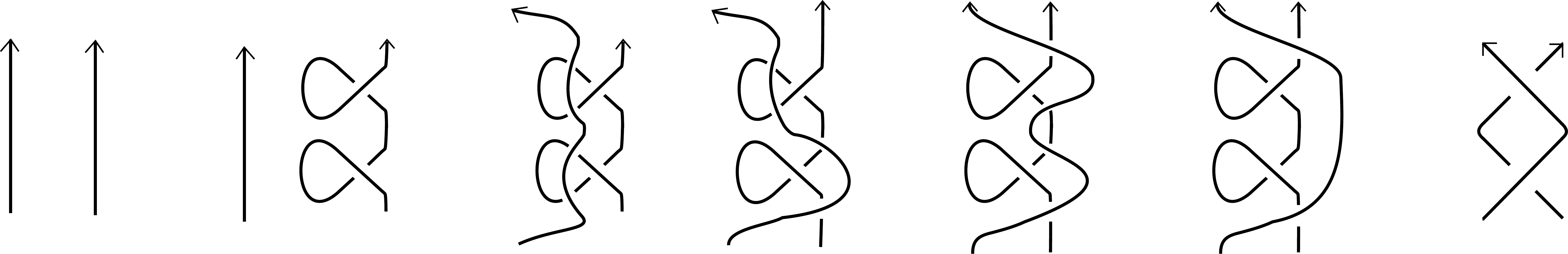}}
\end{equation*}
and with $\Omega 1 \text{f}b$ the order of the Reidemeister 3 moves would be swapped.
\end{proof}

We have now realised all rotational Reidemeister moves of type 2 and 3. In order to realise all framed Reidemeister 1 moves from one of them, we first need the following

\begin{lemma}\label{lem:O23}
    Each of the diagrams
\begin{equation*}
\centre{
\centering
\includegraphics[width=0.65\textwidth]{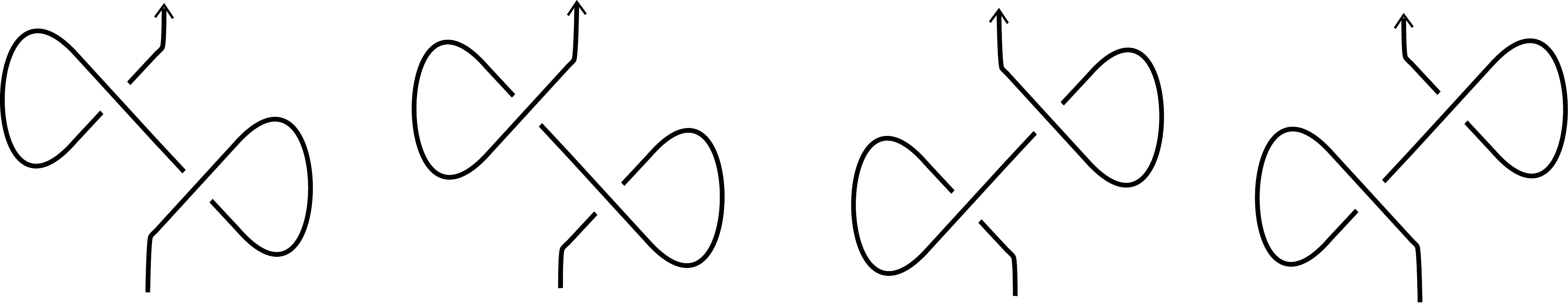}}
\end{equation*}
can be turned into a crosingless, vertical strand by a sequence of rotational Reidemeister moves that does not contain any $\Omega 1 \textup{f}$ move.
\end{lemma}
\begin{proof}
    We have
\begin{equation*}
\centre{
\labellist \small \hair 2pt
\pinlabel{$\leftrightsquigarrow$}  at 470 210
\pinlabel{{\scriptsize $\Omega 2b$}}  at 475 280
\pinlabel{$\leftrightsquigarrow$}  at 850 210
\pinlabel{{\scriptsize $\Omega 3d$}}  at 855 280
\pinlabel{$\leftrightsquigarrow$}  at 1250 210
\pinlabel{{\scriptsize $\Omega 2d1$}}  at 1255 280
\pinlabel{$\leftrightsquigarrow$}  at 1680 210
\pinlabel{{\scriptsize $\Omega 2c1$}}  at 1685 280
\pinlabel{$\leftrightsquigarrow$}  at 2170 210
\pinlabel{{\scriptsize $\Omega 0b$}}  at 2175 280
\endlabellist
\centering
\includegraphics[width=0.8\textwidth]{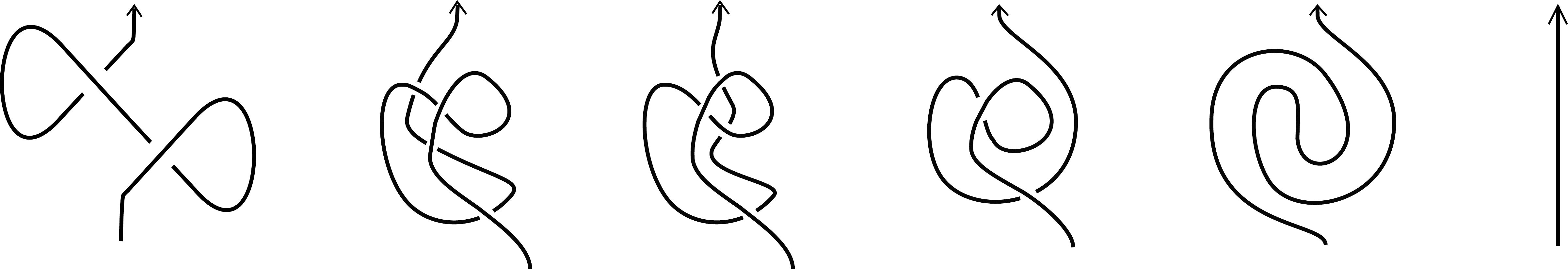}},
\end{equation*}
\begin{equation*}
\centre{
\labellist \small \hair 2pt
\pinlabel{$\leftrightsquigarrow$}  at 470 210
\pinlabel{{\scriptsize $\Omega 2a$}}  at 475 280
\pinlabel{$\leftrightsquigarrow$}  at 850 210
\pinlabel{{\scriptsize $\Omega 3f$}}  at 855 280
\pinlabel{$\leftrightsquigarrow$}  at 1250 210
\pinlabel{{\scriptsize $\Omega 2d2$}}  at 1255 280
\pinlabel{$\leftrightsquigarrow$}  at 1680 210
\pinlabel{{\scriptsize $\Omega 2c2$}}  at 1685 280
\pinlabel{$\leftrightsquigarrow$}  at 2170 210
\pinlabel{{\scriptsize $\Omega 0b$}}  at 2175 280
\endlabellist
\centering
\includegraphics[width=0.8\textwidth]{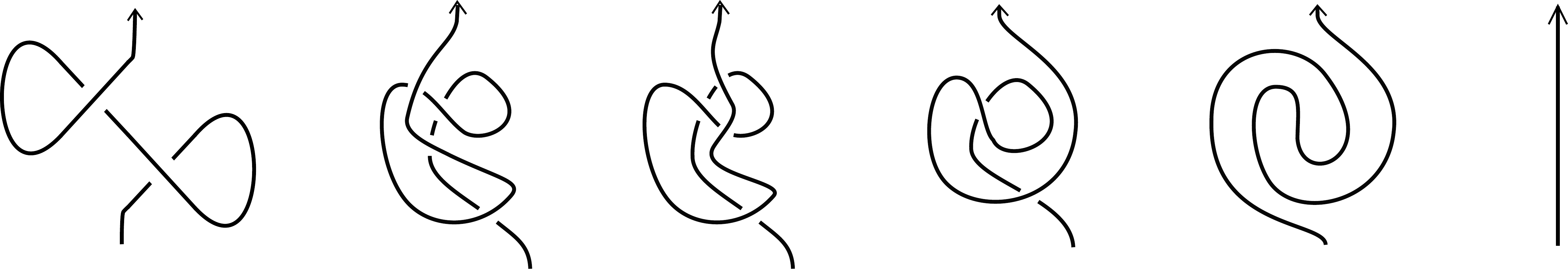}},
\end{equation*}
\begin{equation*}
\centre{
\labellist \small \hair 2pt
\pinlabel{$\leftrightsquigarrow$}  at 470 210
\pinlabel{{\scriptsize $\Omega 2b$}}  at 475 280
\pinlabel{$\leftrightsquigarrow$}  at 850 210
\pinlabel{{\scriptsize $\Omega 3e$}}  at 855 280
\pinlabel{$\leftrightsquigarrow$}  at 1250 210
\pinlabel{{\scriptsize $\Omega 2c1$}}  at 1255 280
\pinlabel{$\leftrightsquigarrow$}  at 1680 210
\pinlabel{{\scriptsize $\Omega 2d1$}}  at 1685 280
\pinlabel{$\leftrightsquigarrow$}  at 2170 210
\pinlabel{{\scriptsize $\Omega 0a$}}  at 2175 280
\endlabellist
\centering
\includegraphics[width=0.8\textwidth]{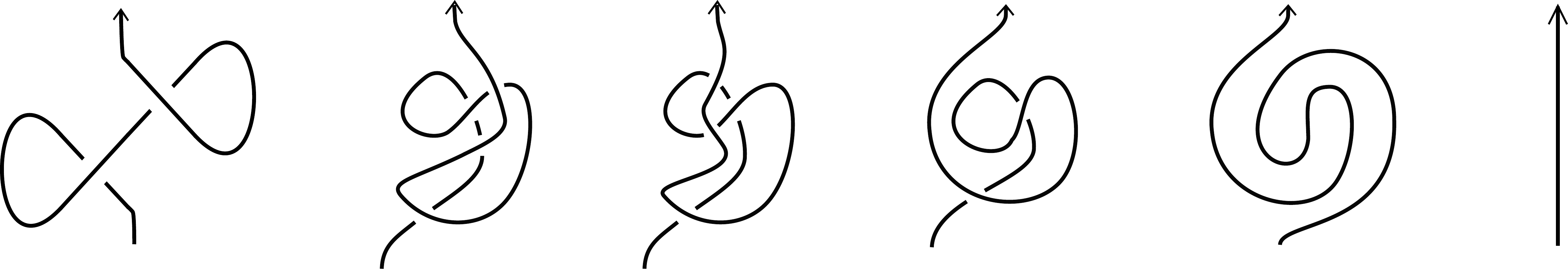}},
\end{equation*}
\begin{equation*}
\centre{
\labellist \small \hair 2pt
\pinlabel{$\leftrightsquigarrow$}  at 470 210
\pinlabel{{\scriptsize $\Omega 2a$}}  at 475 280
\pinlabel{$\leftrightsquigarrow$}  at 850 210
\pinlabel{{\scriptsize $\Omega 3c$}}  at 855 280
\pinlabel{$\leftrightsquigarrow$}  at 1250 210
\pinlabel{{\scriptsize $\Omega 2c2$}}  at 1255 280
\pinlabel{$\leftrightsquigarrow$}  at 1680 210
\pinlabel{{\scriptsize $\Omega 2d2$}}  at 1685 280
\pinlabel{$\leftrightsquigarrow$}  at 2170 210
\pinlabel{{\scriptsize $\Omega 0a$}}  at 2175 280
\endlabellist
\centering
\includegraphics[width=0.8\textwidth]{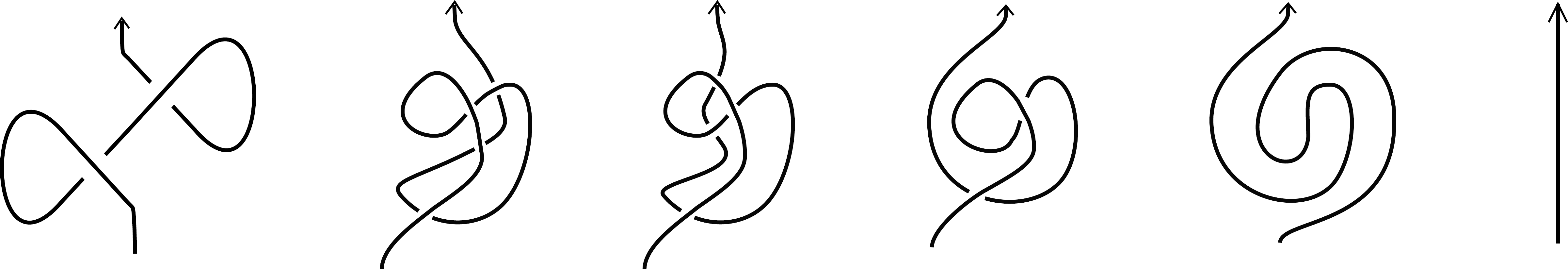}}.
\end{equation*}
\end{proof}

\begin{lemma}\label{lem:23->1f}
In the presence of all Reidemeister 2 and 3 moves, a single $\Omega 1 \textup{f}$ move is enough to generate the rest of $\Omega 1 \textup{f}$ moves.
\end{lemma}
\begin{proof}
Let us start by supposing we have the $\Omega 1 \textup{f}e$ move. Then it follows at once from \cref{lem:O23} that we can realise the moves $\Omega 1 \textup{f}a$ -- $\Omega 1 \textup{f}d$ by applying $\Omega 1 \textup{f}e$ in one of the kinks. In order to obtain $\Omega 1 \textup{f}f$ we simply do
\begin{equation*}
\centre{
\labellist \small \hair 2pt
\pinlabel{$\leftrightsquigarrow$}  at 350 230
\pinlabel{{\scriptsize Lem.\ref{lem:O23}}}  at 355 300
\pinlabel{$\leftrightsquigarrow$}  at 960 230
\pinlabel{{\scriptsize $\Omega 1 \textup{f}c$}}  at 965 300
\endlabellist
\centering
\includegraphics[width=0.5\textwidth]{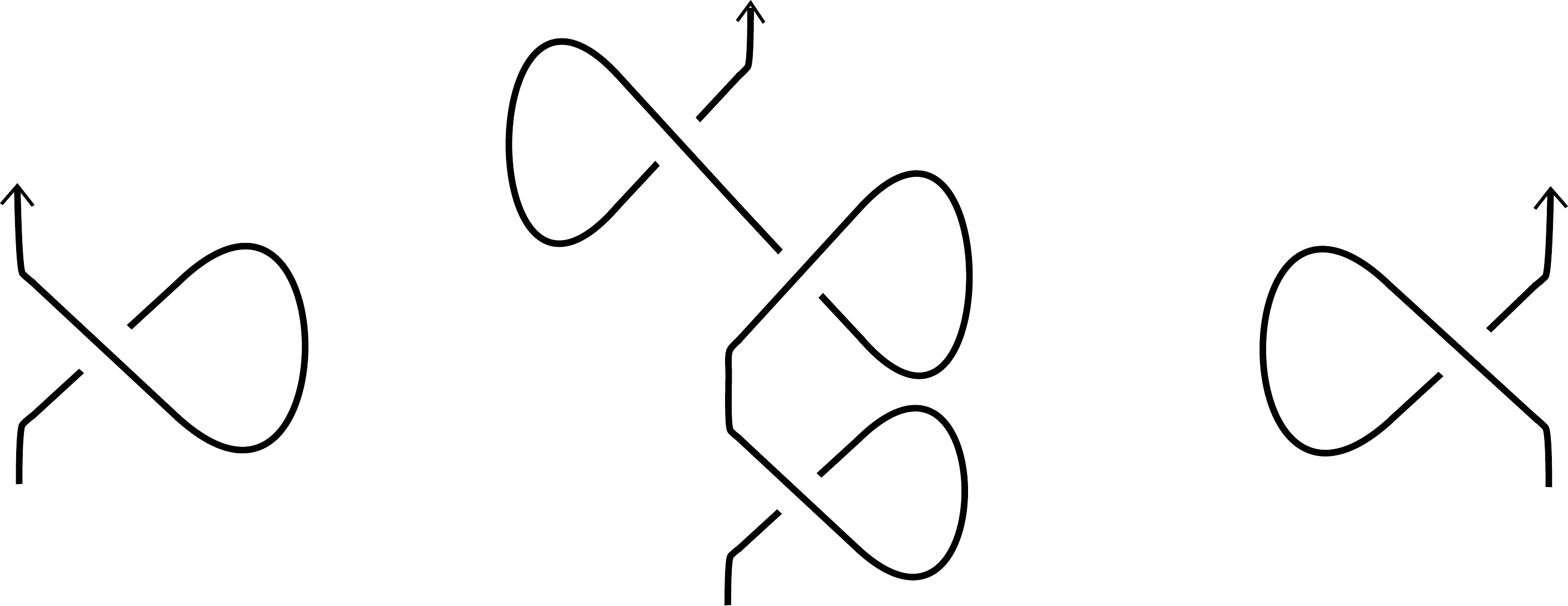}}.
\end{equation*}
If we start with $\Omega 1 \textup{f}f$,  the argument is completely analogous:
\begin{equation*}
\centre{
\labellist \small \hair 2pt
\pinlabel{$\leftrightsquigarrow$}  at 350 230
\pinlabel{{\scriptsize Lem.\ref{lem:O23}}}  at 355 300
\pinlabel{$\leftrightsquigarrow$}  at 960 230
\pinlabel{{\scriptsize $\Omega 1 \textup{f}b$}}  at 965 300
\endlabellist
\centering
\includegraphics[width=0.5\textwidth]{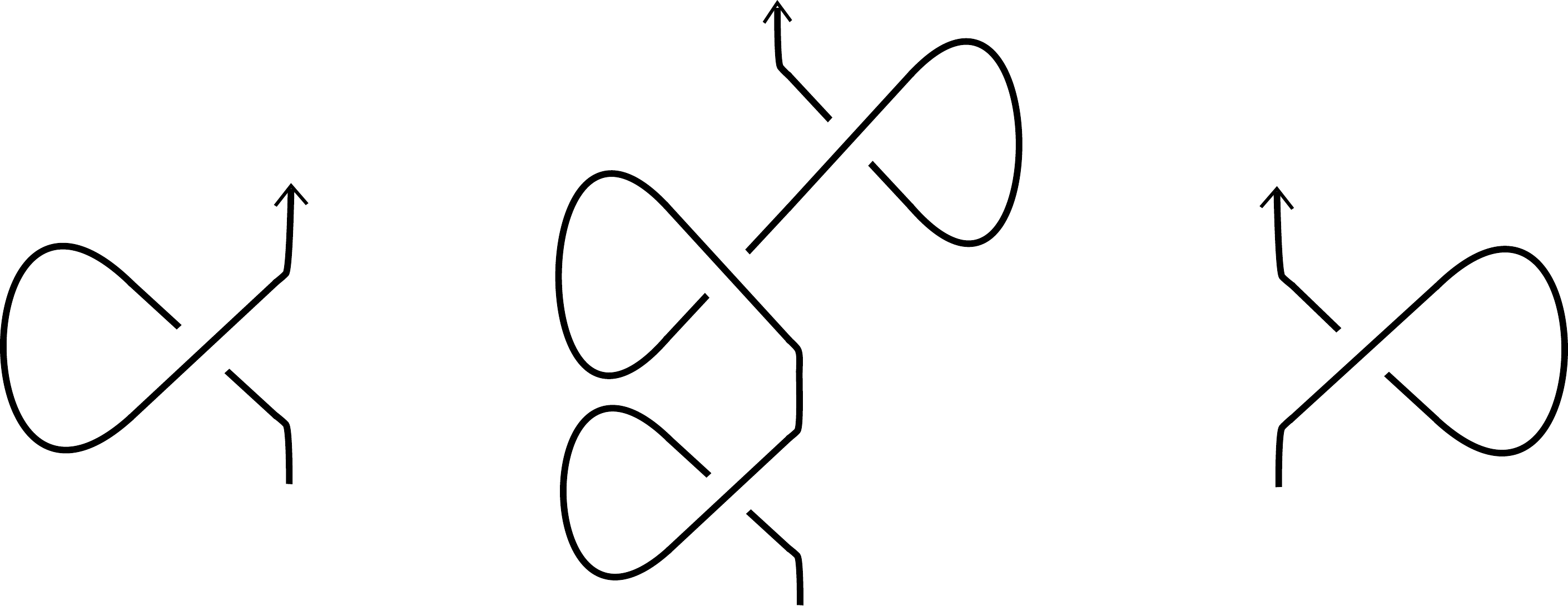}}.
\end{equation*}

Now suppose that we start with $\Omega 1 \textup{f}a$. Then we can first realise  $\Omega 1 \textup{f}e$ as
\begin{equation*}
\centre{
\labellist \small \hair 2pt
\pinlabel{$\leftrightsquigarrow$}  at 350 230
\pinlabel{{\scriptsize Lem.\ref{lem:O23}}}  at 355 300
\pinlabel{$\leftrightsquigarrow$}  at 960 230
\pinlabel{{\scriptsize $\Omega 1 \textup{f}a$}}  at 965 300
\endlabellist
\centering
\includegraphics[width=0.5\textwidth]{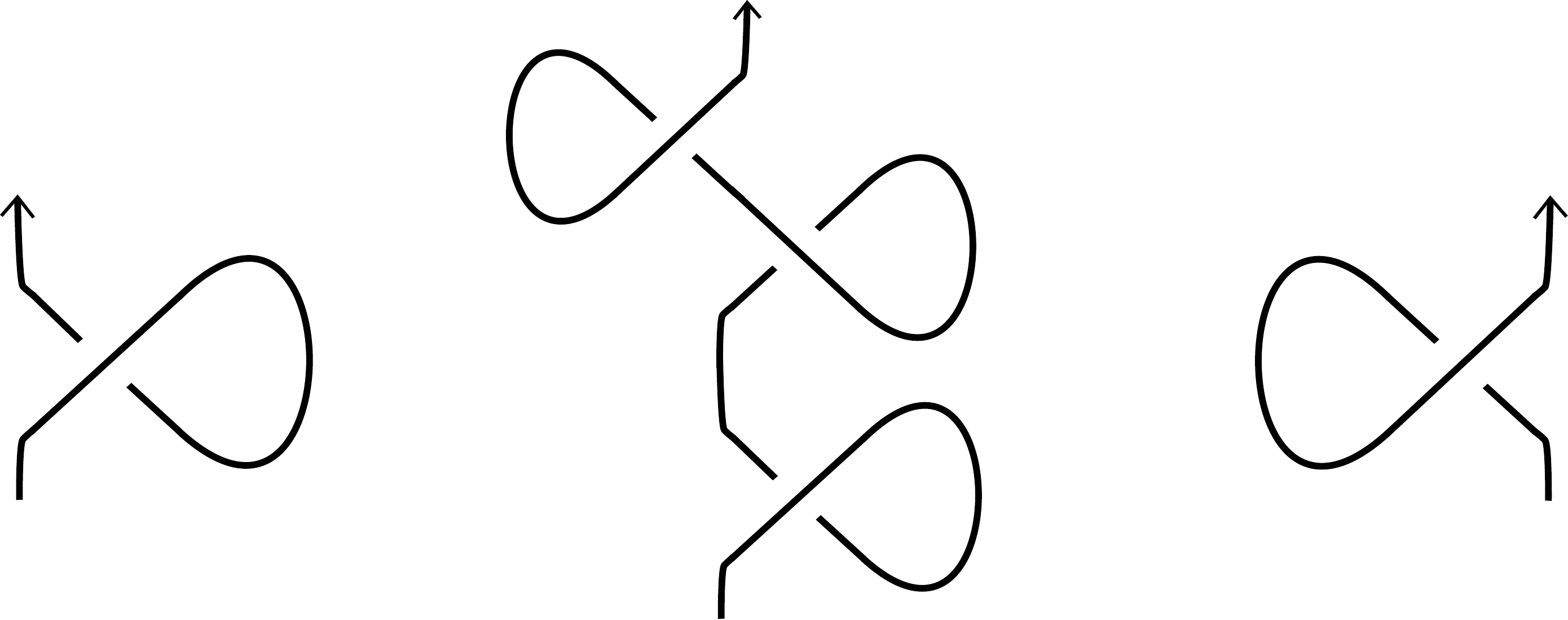}}.
\end{equation*}
which by the argument above allows us to obtain the remaining  $\Omega 1 \textup{f}$ moves. The argument starting with  $\Omega 1 \textup{f}b$ --  $\Omega 1 \textup{f}d$ is completely analogous.
\end{proof}

This concludes the proof of \cref{thm:2}.

\begin{remark}
    We ask the reader to be mindful that we do not claim whatsover that a generating set of (rotational) framed Reidemeister moves must contain four moves of type 2. For instance, the 10-element sets
    $$ S= \mathcal{O} \amalg \{ \Omega 1 \textup{f}a, \Omega 1 \textup{f}b, \Omega 2a, \Omega 2d,  \Omega 3a, \Omega 3h   \}$$ 
    or
        $$ S'= \mathcal{O} \amalg \{ \Omega 1 \textup{f}a, \Omega 2a, \Omega 2b, \Omega 2d,  \Omega 3a, \Omega 3h   \}$$ 
        can be proven to be generating as well. What will be true is that, in the framed setting, we need at least two type 2 moves, see \cref{lem:2_O2} below. This contrasts with the unframed situation, where an easy argument involving the rotation number and writhe of the diagrams shows that we need at least two (unframed) moves of type 1, see \cite[Lemma 3.1]{polyak10}.
\end{remark}



As an immediate consequence, we obtain a framed version of Polyak's minimal set \cite[Theorem 1.1]{polyak10} in the non-rotational setting.

\begin{corollary}\label{cor:pol_fr}
Let $S$ be a set of Reidemeister moves formed by one of the moves $\Omega 2a$ or $\Omega 2b$, one of the pairs
$$(\Omega 1 \textup{f}a, \Omega 2d)  \qquad , \qquad (\Omega 1 \textup{f}c, \Omega 2d)  \qquad , \qquad  (\Omega 1 \textup{f}b, \Omega 2c)  \qquad , \qquad  (\Omega 1 \textup{f}d, \Omega 2c)  $$ and the moves $\Omega 3a$ and $\Omega 3h$.

Then $S$ is generating. That is, two link diagrams represent the same oriented, framed link if and only if they are related by  planar isotopy and a finite sequence of the  Reidemeister moves in $S$.
\end{corollary}

\subsection{Minimality}

In this subsection, we will prove that any of the 5-element generating sets from \cref{cor:pol_fr} is minimal, which will also imply the minimality of the 9-element set of rotational Reidemeister moves from \cref{thm:2}. In particular, this implies that a generating set of framed, oriented Reidemeister moves must contain at least 5 elements  which, to the authors' knowledge, was not known. In the spirit of this paper, we divide the proof in several lemmas.

\begin{lemma}\label{lem:2_O2}
Any generating set of framed, oriented Reidemeister moves contains at least two $\Omega 2$ moves. More precisely, it contains either $\Omega 2a$ or  $\Omega 2b$ and either $\Omega 2c$ or $\Omega 2d$.
\end{lemma}
\begin{proof}
The idea is to prove that a set $S$ of Reidemeister moves formed by all $\Omega 1 \text{f}$ and $\Omega 3$ moves and a single $\Omega 2$ move is not generating.  To this end we will adapt that argument of \cite[Lemma 2.18]{CS}.

Suppose that $S$ contains $\Omega 2a$ as the single Reidemeister 2 move. If we wanted to realise the $\Omega 2c$ move, as it involves oppositely oriented strands, we need to locally change the orientation of one of the strands by means of the $\Omega 1 \text{f}a $ or $\Omega 1 \text{f}c$ move. After applying two Reidemeister 3 moves (concretely $\Omega 3a$ and $\Omega 3h$) to pass the crossingless component over the crossings of the kinks of the other component, we are forced to apply an $\Omega 2d$ move to pass the remaining piece of the component between the two kinks. An argument of the same type shows that to realise $\Omega 2d$ from $\Omega 2a$, we need to apply $\Omega 2c$. Hence $S$ is not generating. The other cases are argued in a similar fashion (adapting \cite[Lemma 2.16]{CS} when starting with $\Omega 2c$ or $\Omega 2d$).
\end{proof}

\begin{lemma}
    A generating set of framed, oriented Reidemeister moves with exactly two $\Omega 2$ moves must contain at least two  $\Omega 3$ moves.

\end{lemma}
\begin{proof}
    Let us suppose that our set contains a single Reidemeister 3 move. In order to realise a second one, \cite[Lemma 2.19]{CS} shows that we must have at our disposal one of the pairs $(\Omega 2a, \Omega 2b)$ or $(\Omega 2c, \Omega 2d)$. By the previous lemma, none of these pairs can be part of a generating set with exactly two Reidemeister 2 moves, which means that before realising a second Reidemeister 3 move we must have realised a third Reidemeister 2 move. But the argument of the previous lemma shows that realising a type 2 move involves two different Reidemeister 3 moves. Hence our starting set is not generating, which yields a contradiction.   
\end{proof}




The two previous lemmas allow us to conclude

\begin{theorem}\label{thm:1.2_minimal}
Any generating set of framed Reidemeister moves contains at least five elements.

Hence, the generating set of \cref{cor:pol_fr} (and therefore  of  \cref{thm:2} in the rotational setting) is minimal.
\end{theorem}

\subsection{Other generating sets of framed rotational Reidemeister moves}

For the sake of completeness, we also describe a family of generating sets of framed rotational Reidemeister moves that includes a single Reidemeister 3 move.

\begin{theorem}\label{thm:another_fr}
Let $S$ be a set of rotational Reidemeister moves formed by all $\Omega 0$ moves, an  $\Omega 2$ move of each type $a$ -- $d$, a single, arbitrary move $\Omega 1 \textup{f}$ and another single, arbitrary $\Omega 3$ move. Then $S$ is generating (an instance is shown below):
\begin{equation*}
\centre{
\labellist \small \hair 2pt
\pinlabel{$\leftrightsquigarrow$}  at 439 190
\pinlabel{{\scriptsize $\Omega 0a$}}  at 445 260
\pinlabel{$\leftrightsquigarrow$}  at 1625 190
\pinlabel{{\scriptsize $\Omega 0b$}}  at 1630 260
\pinlabel{$\leftrightsquigarrow$}  at 2865 190
\pinlabel{{\scriptsize $\Omega 0c$}}  at 2869 260
\endlabellist
\centering
\includegraphics[width=0.9\textwidth]{figures/gens_1}}
\end{equation*}
\begin{equation*}
\centre{
\labellist \small \hair 2pt
\pinlabel{$\leftrightsquigarrow$}  at 439 190
\pinlabel{{\scriptsize $\Omega 0d$}}  at 445 260
\pinlabel{$\leftrightsquigarrow$}  at 1625 190
\pinlabel{{\scriptsize $\Omega 1 \textup{f}e$}}  at 1630 260
\pinlabel{$\leftrightsquigarrow$}  at 2865 190
\pinlabel{{\scriptsize $\Omega 2a$}}  at 2869 260
\endlabellist
\centering
\includegraphics[width=0.9\textwidth]{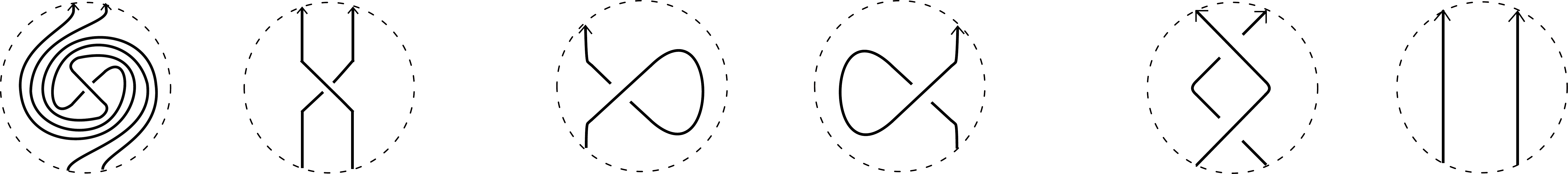}}
\end{equation*}
\begin{equation*}
\centre{
\labellist \small \hair 2pt
\pinlabel{$\leftrightsquigarrow$}  at 439 190
\pinlabel{{\scriptsize $\Omega 2b$}}  at 445 260
\pinlabel{$\leftrightsquigarrow$}  at 1635 190
\pinlabel{{\scriptsize $\Omega 2c1$}}  at 1650 260
\pinlabel{$\leftrightsquigarrow$}  at 2885 190
\pinlabel{{\scriptsize $\Omega 2d1$}}  at 2895 260
\endlabellist
\centering
\includegraphics[width=0.9\textwidth]{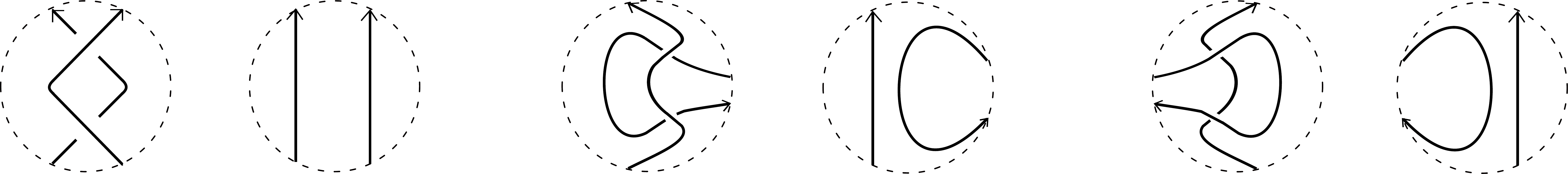}}
\end{equation*}
\begin{equation*}
\centre{
\labellist \small \hair 2pt
\pinlabel{$\leftrightsquigarrow$}  at 439 190
\pinlabel{{\scriptsize $\Omega 3b$}}  at 445 260
\endlabellist
\centering
\includegraphics[width=0.25\textwidth]{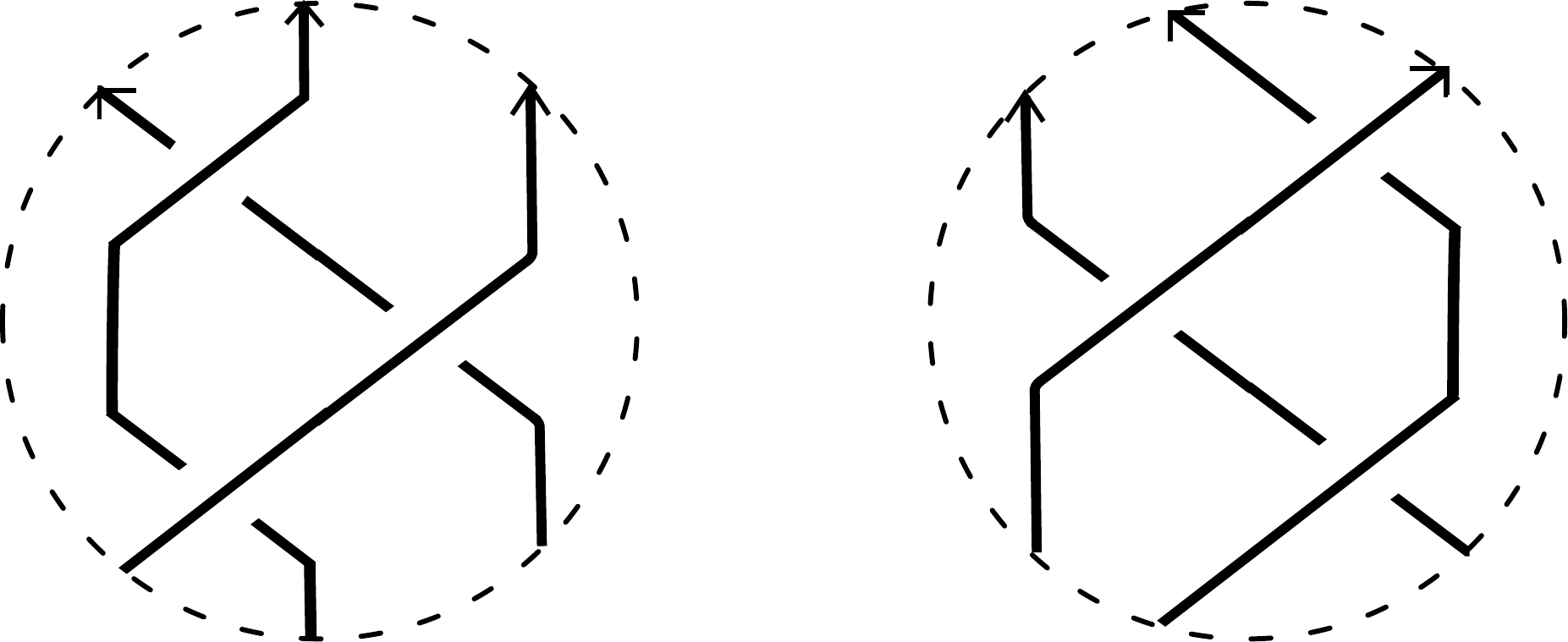}}
\end{equation*}
\end{theorem}
\begin{proof}
The main feature of this family of sets is that they all contain $\Omega 2a$, $\Omega 2b$, either $\Omega 2c1$ or $\Omega 2c2$ and either $\Omega 2d1$ or $\Omega 2d2$. Since all $\Omega 0$ moves are also included, \cref{lem:O2c2_O2d2} allows us to realise the remaining two $\Omega 2$ moves. Furthermore, the arguments of \cref{lem:O3a} and  Lemmas \ref{lem:O3b} -- \ref{lem:O3h} show that in the presence of all $\Omega 2$ moves, a single $\Omega 3$ move is enough to realise the rest. Indeed a direct application of the lemmas show that $\Omega 3a1$ (or any other $\Omega 3a$ really) can be used to realise all $\Omega 3$ moves. If the $\Omega 3b$ was included in the set $S$ instead, then one would use the sequence of moves from the proof of \cref{thm:polyak_1.2_rot} to realise $\Omega 3a1$ and from there obtain the remaining Reidemeister 3 moves using Lemmas \ref{lem:O3b} -- \ref{lem:O3h}. If the move $\Omega 3c$ was included instead, then a similar argument allows to realise one of the $\Omega 3a$ moves from it. It is readily seen that a similar argument can be done for the remaining Reidemeister 3 moves adapting \cref{lem:O3d_g} or \cref{lem:O3h}. Once we have realised all $\Omega 3$ moves, we conclude by \cref{lem:23->1f}.
\end{proof}

\subsection{Application: a universal tangle invariant}\label{sec:application}

We would like to finish off with the main application of the results of this paper, which was in fact the main motivation for it.

A rotational tangle diagram is, by definition, built out of the following five elementary pieces up to planar Morse isotopy:
\begin{equation}\label{eq:el_pieces}
\centre{
\centering
\includegraphics[width=0.5\textwidth]{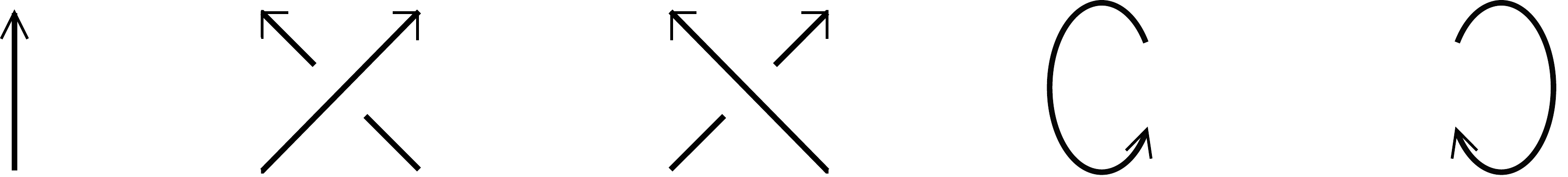}}
\end{equation}
Theorems \ref{thm:1} and \ref{thm:2} above (as well as Theorems \ref{thm:all_min_sets}, \ref{thm:polyak_1.2_rot} and \ref{thm:another_fr}) state that two rotational tangle diagrams represent the same unframed or framed oriented tangle if and only if they are related by some rotational Reidemeister moves in terms of these elementary pieces. Put in a different way, they state that certain combinations of these elementary pieces represent the same tangle, and any other different combinations that represent the same tangle can be derived from those.

We can regard these elementary pieces and the Reidemeister moves, of course, as sets of generators and relations at the diagrammatic level. The success of the rotational diagram approach is that it allows for a straightforward algebraic counterpart.

Let $A$ be an algebra over some commutative ring $\Bbbk$. We aim to associate, to every rotational tangle $D$ with $n$ ordered open components, an element $\mathfrak{Z}_A(D) \in A^{\otimes n}$ as follows: first we choose four elements $$ R= \sum_i \alpha_i \otimes \beta_i \in A \otimes A \quad , \quad  \bar{R}= \sum_i \bar{\alpha}_i \otimes \bar{\beta}_i \in A \otimes A \quad , \quad  \kappa \in A \quad , \quad  \bar{\kappa} \in A$$
and we decorate the elementary pieces \eqref{eq:el_pieces} with beads in each of the strands in the following manner:

\begin{equation}\label{eq:beads}
\centre{
\labellist \small \hair 2pt
\pinlabel{$ \color{violet} \bullet$} at 16 30
\pinlabel{$ \color{violet} 1$} at -26 30
\pinlabel{$ \color{violet} \bullet$} at 294 30
\pinlabel{$ \color{violet} \bullet$} at 406 30
\pinlabel{$ \color{violet} \alpha_i$} [r] at 279 30
\pinlabel{$ \color{violet} \beta_i$} [l] at 421 30
\pinlabel{$ \color{violet} \bullet$} at 711 30
\pinlabel{$ \color{violet} \bullet$} at 825 30
\pinlabel{$ \color{violet} \bar{\beta}_i$} [r] at 696 30
\pinlabel{$ \color{violet} \bar{\alpha}_i$} [l] at 850 30
\pinlabel{$ \color{violet} \bullet$} at 1083 92
\pinlabel{$ \color{violet} \bar{\kappa}$} [r] at 1070 120
\pinlabel{$ \color{violet} \bullet$} at 1599 92
\pinlabel{$ \color{violet} \kappa$} [l] at 1614 104
\endlabellist
\centering
\includegraphics[width=0.5\textwidth]{figures/building_blockss}}
\end{equation}
(note that the ``alpha'' is always on the overstrand). For every $1 \leq i \leq n$, let $\mathfrak{Z}_A(D)_{(i)}$ be the (formal) word given by writing from right  to left the labels of the beads in the $i$-th component according to the  orientation of the strand. Then put
\begin{equation}\label{eq:def_mathfrak_Z}
\mathfrak{Z}_A(D) = \sum \mathfrak{Z}_A(D)_{(1)} \otimes \cdots \otimes \mathfrak{Z}_A(D)_{(n)} \in A^{\otimes n}
\end{equation}
where the summation runs through all subindices in $R$ and $\bar{R}$ (one for each crossing).

Under this passage, the set of algebraic relations that $R$, $\bar{R}$, $\kappa$ and $\bar{\kappa}$ must satisfy to obtain a well-defined isotopy invariant is precisely the one given by the counterparts of the (say, framed) rotational Reidemeister moves. For instance, $\Omega 0a$ and $\Omega 0b$ imply that $\kappa$ is invertible with $\bar{\kappa} = \kappa^{-1}$, and similarly $\Omega 2a$ and $\Omega 2b$ that  $R$ is invertible with $\bar{R} = R^{-1}$.

According to \cref{thm:another_fr}, the remaining relations required to obtain a well-defined  isotopy invariant of oriented, framed tangles are 
\begin{enumerate}[leftmargin=4\parindent, itemsep=2mm]
\item[(XC0c)] $R=(\kappa \otimes \kappa) \cdot R \cdot (\kappa^{-1} \otimes \kappa^{-1})$,
\item[(XC0d)] $R^{-1}=(\kappa \otimes \kappa) \cdot R^{-1} \cdot (\kappa^{-1} \otimes \kappa^{-1})$,
\item[(XC1fe)] $\mu^{[3]}(R_{31}\cdot  \kappa_2 )=\mu^{[3]}(R_{13}\cdot  \kappa^{-1}_2) $,
\item[(XC2c1)] $  1\otimes \kappa^{-1} = (\mu\otimes\mu^{[3]} )(R_{15}\cdot R_{23}^{-1}\cdot \kappa^{-1}_4 )$,
\item[(XC2d1)] $\kappa \otimes 1 = (\mu^{[3]}\otimes\mu )(R_{15}^{-1}\cdot R_{34}\cdot \kappa_2 )$,
\item[(XC3b)] $R_{12}R_{13}R_{23}=R_{23}R_{13}R_{12}$,
\end{enumerate}
where  we have written $\mu$ for the multiplication map of $A$,  $\mu^{[3]}$ for the 3-fold multiplication map and for $1 \leq i,j \leq n$, $i \neq j$ we have put
\begin{equation}\label{eq:R_ij}
 R_{ij}^{\pm 1}:= \begin{cases}  (1^{\otimes i-1} \otimes \id \otimes 1^{\otimes j-i-1} \otimes \id \otimes 1^{n-j})(R^{\pm 1}), & i<j\\
(1^{\otimes j-1} \otimes \id \otimes 1^{\otimes i-j-1} \otimes \id \otimes 1^{n-i})(\mathrm{flip}_{A,A}(R^{\pm 1})), & i>j
\end{cases} 
\end{equation}
and similarly $\kappa^{\pm 1}_i = (1^{\otimes i-1} \otimes \id \otimes  1^{\otimes n-i})(\kappa^{\pm 1})$.

As in \cite{becerra_refined}, we call a triple $(A, R, \kappa)$ with $R$ and $\kappa$ invertible and satisfying the relations above an \textit{XC-algebra}. In \cite{becerra_thesis} fewer axioms appeared, which makes the definition there incomplete. By construction,

\begin{proposition}\label{prop:XC_invariant}
For any XC-algebra $(A,R, \kappa)$, we have that    $\mathfrak{Z}_A$  gives rise to a well-defined isotopy invariant of framed, oriented open tangles in a disc.
\end{proposition}

The construction of $\mathfrak{Z}_A(D)$ is essentially the so-called \textit{universal invariant}, which was defined by Lawrence \cite{lawrence}  and  was further developed by Reshetikhin \cite{reshetikhin}, Lee \cite{lee1,lee2}, Hennings \cite{hennings}, Ohtsuki \cite{ohtsuki1,ohtsukibook} and Habiro \cite{habiro}. Classically, the universal invariant has as input a ribbon Hopf algebra. However, such an  algebra carries extra structure (namely, the comultiplication, counit and antipode) that does not play a role in constructing the invariant. On the other hand, XC-algebras carry the minimal algebraic data to produce a tangle invariant. Of course, the main example of XC-algebras are ribbon Hopf algebras, this is proven in \cite[Proposition 4.4]{becerra_refined}. What is more, if $V$ is a finite-dimensional representation of a ribbon Hopf algebra, then its endomorphism algebra $\mathrm{End}_\Bbbk(V)$ carries a natural XC-algebra structure whose associated universal invariant is essencially the Reshetikhin-Turaev invariant built from $V$, see \cite[Theorem 5.4]{becerra_refined}. In fact, an XC-algebra (and not a ribbon Hopf algebra) is all what is required to promote $\mathfrak{Z}$ to a balanced monoidal functor (i.e., braiding- and twist-preserving) which also respects a so-called ``open-trace'' defined in the category $\Tup$ of upwards, framed, oriented tangles in a cube \cite[\S 4]{becerra_refined}.






 




\bibliographystyle{halpha-abbrv}
\bibliography{bibliography}

\begin{thebibliography}{{Bec}24b}
\expandafter\ifx\csname url\endcsname\relax
  \def\url#1{\texttt{#1}}\fi
\expandafter\ifx\csname doi\endcsname\relax
  \def\doi#1{\burlalt{doi:#1}{http://dx.doi.org/#1}}\fi
\expandafter\ifx\csname urlprefix\endcsname\relax\def\urlprefix{URL }\fi
\expandafter\ifx\csname href\endcsname\relax
  \def\href#1#2{#2}\fi
\expandafter\ifx\csname burlalt\endcsname\relax
  \def\burlalt#1#2{\href{#2}{#1}}\fi

\bibitem[Bec24a]{becerra_gaussians}
J.~Becerra.
\newblock On {B}ar-{N}atan--van der {V}een's perturbed {G}aussians.
\newblock {\em Rev. R. Acad. Cienc. Exactas F\'{\i}s. Nat. Ser. A Mat. RACSAM}, 118(2):Paper No. 46, 58, 2024.
\newblock \doi{10.1007/s13398-023-01536-1}.

\bibitem[{Bec}24b]{becerra_thesis}
J.~{Becerra Garrido}.
\newblock {\em Universal quantum knot invariants}.
\newblock PhD thesis, University of Groningen, 2024.
\newblock \doi{10.33612/diss.989716384}.

\bibitem[Bec25]{becerra_refined}
J.~Becerra.
\newblock A refined functorial universal tangle invariant, 2025, \burlalt{arXiv:2501.17668}{http://arxiv.org/abs/arXiv:2501.17668}.
\newblock \urlprefix\url{https://arxiv.org/abs/2501.17668}.

\bibitem[BV18]{barnatanveenpolytime}
D.~Bar-Natan and R.~van~der Veen.
\newblock A polynomial time knot polynomial.
\newblock {\em Proc. Amer. Math. Soc.}, 147(1):377--397, 2019.
\newblock \doi{10.1090/proc/14166}.

\bibitem[BV21]{barnatanveengaussians}
D.~Bar-Natan and R.~van~der Veen.
\newblock Perturbed {G}aussian generating functions for universal knot invariants, 2021.
\newblock \doi{10.48550/ARXIV.2109.02057}.

\bibitem[BV24]{barnatanveenAPAI}
D.~Bar-Natan and R.~van~der Veen.
\newblock A perturbed-{A}lexander invariant.
\newblock {\em Quantum Topol.}, 15(3):449--472, 2024.
\newblock \doi{10.4171/qt/206}.

\bibitem[CS22]{CS}
C.~Caprau and B.~Scott.
\newblock Minimal generating sets of moves for diagrams of isotopic knots and spatial trivalent graphs.
\newblock {\em J. Knot Theory Ramifications}, 31(12):Paper No. 2250085, 28, 2022.
\newblock \doi{10.1142/S0218216522500857}.

\bibitem[Hab06]{habiro}
K.~Habiro.
\newblock Bottom tangles and universal invariants.
\newblock {\em Algebr. Geom. Topol.}, 6:1113--1214, 2006.
\newblock \doi{10.2140/agt.2006.6.1113}.

\bibitem[Hen96]{hennings}
M.~Hennings.
\newblock Invariants of links and {$3$}-manifolds obtained from {H}opf algebras.
\newblock {\em J. London Math. Soc. (2)}, 54(3):594--624, 1996.
\newblock \doi{10.1112/jlms/54.3.594}.

\bibitem[Law89]{lawrence}
R.~J. Lawrence.
\newblock A universal link invariant using quantum groups.
\newblock In {\em Differential geometric methods in theoretical physics ({C}hester, 1988)}, pages 55--63. World Sci. Publ., Teaneck, NJ, 1989.

\bibitem[Lee92]{lee1}
H.~C. Lee.
\newblock Tangle invariants and centre of the quantum group.
\newblock In {\em Knots 90 ({O}saka, 1990)}, pages 341--361. de Gruyter, Berlin, 1992.

\bibitem[Lee96]{lee2}
H.~C. Lee.
\newblock Universal tangle invariant and commutants of quantum algebras.
\newblock {\em J. Phys. A}, 29(2):393--425, 1996.
\newblock \doi{10.1088/0305-4470/29/2/019}.

\bibitem[Oht93]{ohtsuki1}
T.~Ohtsuki.
\newblock Colored ribbon {H}opf algebras and universal invariants of framed links.
\newblock {\em J. Knot Theory Ramifications}, 2(2):211--232, 1993.
\newblock \doi{10.1142/S0218216593000131}.

\bibitem[Oht02]{ohtsukibook}
T.~Ohtsuki.
\newblock {\em Quantum invariants}, volume~29 of {\em Series on Knots and Everything}.
\newblock World Scientific Publishing Co., Inc., River Edge, NJ, 2002.
\newblock A study of knots, 3-manifolds, and their sets.

\bibitem[Pol10]{polyak10}
M.~Polyak.
\newblock Minimal generating sets of {R}eidemeister moves.
\newblock {\em Quantum Topol.}, 1(4):399--411, 2010.
\newblock \doi{10.4171/QT/10}.

\bibitem[Rei27]{reidemeister}
K.~Reidemeister.
\newblock Knoten und {G}ruppen.
\newblock {\em Abh. Math. Sem. Univ. Hamburg}, 5(1):7--23, 1927.
\newblock \doi{10.1007/BF02952506}.

\bibitem[Res89]{reshetikhin}
N.~Y. Reshetikhin.
\newblock Quasitriangular {H}opf algebras and invariants of links.
\newblock {\em Algebra i Analiz}, 1(2):169--188, 1989.

\end{thebibliography}

\end{document}